\documentclass[a4paper,reqno,11pt]{amsart}

\usepackage{amsmath, amsfonts, amssymb, amsthm, amscd}
\usepackage{graphicx}
\usepackage{psfrag}
\usepackage{perpage}
\usepackage{url}
\usepackage{color}
\usepackage{mathrsfs}
\usepackage{tikz}
\usepackage{mathabx}
\usepackage{scalerel}
\usetikzlibrary{arrows,decorations.pathmorphing,backgrounds,positioning,fit,petri} 

\usepackage{dsfont} 

\usepackage[utf8]{inputenc}
\usepackage[T1]{fontenc}
\usepackage{microtype}

\usepackage[a4paper,scale={0.72,0.74},marginratio={1:1},footskip=7mm,headsep=10mm]{geometry}

\usepackage{hyperref}

\makeatletter
\def\@secnumfont{\bfseries\scshape}

\def\section{\@startsection{section}{1}%
  \z@{.7\linespacing\@plus\linespacing}{.5\linespacing}%
  {\normalfont\large\bfseries\scshape\centering}}

\def\subsection{\@startsection{subsection}{2}%
  \z@{.5\linespacing\@plus.7\linespacing}{-.5em}%
  {\normalfont\bfseries\scshape}}

\def\subsubsection{\@startsection{subsubsection}{3}%
  \z@{.5\linespacing\@plus.7\linespacing}{-.5em}%
  {\normalfont\scshape}}

\def\specialsection{\@startsection{section}{1}%
  \z@{\linespacing\@plus\linespacing}{.5\linespacing}%
  {\normalfont\centering\large\bfseries\scshape}}
\makeatother

\makeatletter

\renewenvironment{proof}[1][\proofname]{\par
\pushQED{\qed}%
\normalfont \topsep4\p@\@plus4\p@\relax
\trivlist
\item[\hskip\labelsep
\bfseries
#1\@addpunct{.}]\ignorespaces
}{%
\popQED\endtrivlist\@endpefalse
}
\makeatother

\makeatletter
\newcommand \Dotfill {\leavevmode \leaders \hb@xt@ 6pt{\hss .\hss }\hfill \kern \z@}
\makeatother

\makeatletter
\def\@tocline#1#2#3#4#5#6#7{\relax
  \ifnum #1>\c@tocdepth 
  \else
    \par \addpenalty\@secpenalty\addvspace{#2}%
    \begingroup \hyphenpenalty\@M
    \@ifempty{#4}{%
      \@tempdima\csname r@tocindent\number#1\endcsname\relax
    }{%
      \@tempdima#4\relax
    }%
    \parindent\z@ \leftskip#3\relax \advance\leftskip\@tempdima\relax
    \rightskip\@pnumwidth plus4em \parfillskip-\@pnumwidth
    #5\leavevmode\hskip-\@tempdima
      \ifcase #1
       \or\or \hskip 1.65em \or \hskip 3.3em \else \hskip 4.95em \fi%
      #6\nobreak\relax
    \Dotfill
    \hbox to\@pnumwidth{\@tocpagenum{#7}}\par
    \nobreak
    \endgroup
  \fi}
\makeatother

\makeatletter
\def\l@section{\@tocline{1}{0pt}{1pc}{}{\scshape}}
\renewcommand{\tocsection}[3]{%
\indentlabel{\@ifnotempty{#2}{\ignorespaces#1 #2.\hskip 0.7em}}#3}
\def\l@subsection{\@tocline{2}{0pt}{1pc}{5pc}{}}

\def\l@subsubsection{\@tocline{3}{0pt}{1pc}{7pc}{}}

\makeatother

\numberwithin{equation}{section}

\newtheoremstyle{mytheorem}{.7\linespacing\@plus.3\linespacing}{.7\linespacing\@plus.3\linespacing}%
     {\itshape}
     {}
     {\bfseries}
     {. }
     {0.3ex}
     {\thmname{{\bfseries #1}}\thmnumber{ {\bfseries #2}}\thmnote{ (#3)}}  

\theoremstyle{mytheorem}

\newtheorem{theorem}{Theorem}[section]
\newtheorem{lemma}[theorem]{Lemma}
\newtheorem{proposition}[theorem]{Proposition}

\newtheorem{remark}[theorem]{Remark}

\newcommand{\bbE}{{\ensuremath{\mathbb E}} }

\newcommand{\bbP}{{\ensuremath{\mathbb P}} }

\newcommand{\bbZ}{{\ensuremath{\mathbb Z}} }

\newcommand\bsA{\boldsymbol{A}}
\newcommand\bsB{\boldsymbol{B}}
\newcommand\bsC{\boldsymbol{C}}
\newcommand\bsD{\boldsymbol{D}}

\newcommand\bsN{\boldsymbol{N}}

\newcommand\bsY{\boldsymbol{Y}}

\newcommand\bsnu{\boldsymbol{\nu}}

\newcommand\bsell{\boldsymbol{\ell}}

\newcommand{\cI}{{\ensuremath{\mathcal I}} }

\newcommand{\cM}{{\ensuremath{\mathcal M}} }

\newcommand{\cR}{{\ensuremath{\mathcal R}} }
\newcommand{\cS}{{\ensuremath{\mathcal S}} }
\newcommand{\cT}{{\ensuremath{\mathcal T}} }
\newcommand{\cU}{{\ensuremath{\mathcal U}} }

\newcommand{\gb}{\beta}

\newcommand{\gl}{\lambda}

\newcommand{\go}{\omega}

\DeclareMathSymbol{\leqslant}{\mathalpha}{AMSa}{"36} 
\DeclareMathSymbol{\geqslant}{\mathalpha}{AMSa}{"3E} 
\DeclareMathSymbol{\eset}{\mathalpha}{AMSb}{"3F}     

\newcommand{\sumtwo}[2]{\sum_{\substack{#1 \\ #2}}} 
\newcommand{\sumthree}[3]{\sum_{\substack{#1 \\ #2 \\ #3}}} 

\newcommand{\be}{\begin{equation}}
\newcommand{\ee}{\end{equation}}

\newcommand{\R}{\mathbb{R}}

\newcommand{\Z}{\mathbb{Z}}
\newcommand{\N}{\mathbb{N}}

\def\bs{\boldsymbol}

\newcommand{\PEfont}{\mathrm}

\newcommand{\p}{\ensuremath{\PEfont P}}

\DeclareMathOperator{\e}{\ensuremath{\PEfont E}}

\newcommand{\E}{\e}
\renewcommand{\P}{\p}

\DeclareMathOperator{\bbvar}{\ensuremath{\mathbb{V}ar}}

\newcommand{\ind}{\mathds{1}}

\newcommand{\eps}{\varepsilon}
\renewcommand{\epsilon}{\varepsilon}
\renewcommand{\theta}{\vartheta}
\renewcommand{\rho}{\varrho}

\newenvironment{myenumerate}{%
\renewcommand{\theenumi}{\arabic{enumi}}%
\renewcommand{\labelenumi}{{\rm(\theenumi)}}%
\begin{list}{\labelenumi}
	{%
	\setlength{\itemsep}{0.4em}%
	\setlength{\topsep}{0.5em}%
	\setlength\leftmargin{2.45em}%
	\setlength\labelwidth{2.05em}%
	\setlength{\labelsep}{0.4em}%
	\usecounter{enumi}%
	}%
	}%
{\end{list}
}

{\end{list}
}

{\end{list}
}

\renewenvironment{enumerate}{
\begin{myenumerate}}%
{\end{myenumerate}}

\newenvironment{myitemize}{%
\begin{list}{$\bullet$}%
 	{%
	\setlength{\itemsep}{0.4em}%
	\setlength{\topsep}{0.5em}%
	\setlength\leftmargin{2.65em}%
	\setlength\labelwidth{2.65em}%
	\setlength{\labelsep}{0.4em}%
	}%
	}%
{\end{list}}

\renewenvironment{itemize}{
\begin{myitemize}}%
{\end{myitemize}}

\MakePerPage[2]{footnote} 

\date{\today}

\newcommand\dd{\mathrm{d}}

\newcommand\hbeta{{\hat{\beta}}}

\newcommand\bx{\boldsymbol{x}}
\newcommand\by{\boldsymbol{y}}

\newcommand\sfc{\mathsf{c}}

\newcommand\bK{\boldsymbol{K}}
\newcommand\bU{\boldsymbol{U}}

\begin{document}

\title[Critical directed polymer and stochastic heat
equation on $\Z^{2+1}$]{On the moments of the $(2+1)$-dimensional\\
directed polymer and stochastic heat equation\\ in the critical window}

\begin{abstract}
The partition function of the directed polymer model on $\Z^{2+1}$
undergoes a phase transition in a suitable continuum and weak disorder limit.
In this paper, we focus on a window around the critical point.
Exploiting local renewal theorems, we
compute the limiting third moment of the space-averaged partition function,
showing that it is uniformly bounded.
This implies that the
rescaled partition functions, viewed as a generalized random field
on $\R^{2}$,
have non-trivial subsequential limits,
and each such limit has the same explicit covariance structure.
We obtain analogous results for the stochastic heat equation on $\R^2$,
extending previous work by Bertini and Cancrini \cite{BC98}.

\end{abstract}

\author[F.Caravenna]{Francesco Caravenna}
\address{Dipartimento di Matematica e Applicazioni\\
 Universit\`a degli Studi di Milano-Bicocca\\
 via Cozzi 55, 20125 Milano, Italy}
\email{francesco.caravenna@unimib.it}

\author[R.Sun]{Rongfeng Sun}
\address{Department of Mathematics\\
National University of Singapore\\
10 Lower Kent Ridge Road, 119076 Singapore
}
\email{matsr@nus.edu.sg}

\author[N.Zygouras]{Nikos Zygouras}
\address{Department of Statistics\\
University of Warwick\\
Coventry CV4 7AL, UK}
\email{N.Zygouras@warwick.ac.uk}

\date{\today}

\keywords{Directed Polymer Model, Marginal Disorder Relevance, Continuum Limit,
Stochastic Heat Equation}
\subjclass[2010]{Primary: 82B44; Secondary: 82D60, 60K35}

\maketitle

\tableofcontents

\section{Introduction and results}

We set $\N := \{1,2,3,\ldots\}$ and $\N_0 := \N \cup \{0\}$.
We write $a_n \sim b_n$ to mean $\lim_{n\to\infty} a_n / b_n = 1$.
We denote by $C_b(\R^d)$ (resp.\ $C_c(\R^d)$) the space of continuous and
bounded (resp.\ compactly supported) real functions defined on $\R^d$,
with norm $|\phi|_\infty := \sup_{x\in\R^d} |\phi(x)|$.

\subsection{Directed polymer in random environment}
One of the simplest, yet also most interesting models of disordered system
is the directed polymer model in random environment on $\Z^{d+1}$, which has
been the subject of the recent monograph by Comets~\cite{C17}.

Let $S=(S_n)_{n\in\N_0}$ be
the simple symmetric random walk on $\Z^d$.
The random environment (or disorder) is a collection $\omega=(\omega_{n,x})_{(n,x)\in \N\times\Z^d}$
of i.i.d.\ random variables. We use $\P$ and $\E$,
resp.\ $\bbP$ and $\bbE$, to denote probability and expectation for $S$,
resp.\ for $\omega$. We assume that
\begin{equation} \label{eq:genomega}
	\bbE[\omega_{n,x}] = 0 \,, \qquad \bbvar[\omega_{n,x}] = 1 \,, \qquad
	\lambda(\beta):=\log \bbE[e^{\beta \omega_{n,x}}] \in \R \quad \text{ for small }
	\beta > 0.
\end{equation}

Given $\omega$, polymer length $N\in\N$, and inverse temperature (or disorder strength) $\beta>0$, the polymer measure $\P^\beta_{N}$ is then defined via a Gibbs change of measure for $S$:
\be
\P^\beta_{N} (S) :=
\frac{e^{\sum_{n=1}^{N-1}
(\beta\omega_{n, S_n} -\lambda(\beta))}}{Z_N^{\beta}} \, \P(S) \,,
\ee
where $Z_N^{\beta}$ is the normalization constant,
called \emph{partition function}:
\be\label{eq:Z}
Z_N^{\beta} := \E\Big[e^{\sum_{n=1}^{N-1}
(\beta\omega_{n, S_n} -\lambda(\beta))}\Big] \,.
\ee
(We stop the sum at $N-1$ instead of $N$, which is immaterial,
for later notational convenience.)
Note that $Z_N^{\beta}$ is a random variable, as a function of $\omega$.

We use $\P_{z}$ and $\E_{z}$ to
denote probability and expectation for the random
walk starting at $S_0=z\in\Z^d$.
We denote by $Z_N^{\beta}(z)$
the corresponding partition function:
\be\label{eq:Zx}
	Z_N^{\beta}(z) := \E_z\Big[e^{\sum_{n=1}^{N-1}
	(\beta\omega_{n, S_n} -\lambda(\beta))}\Big] \,.
\ee
We investigate the behavior as $N\to\infty$ of the diffusively rescaled random field
\begin{equation}\label{eq:diffresc}
	\Big\{ Z_{Nt}^{\beta_N}\big(x \sqrt{N}\big) \,: \quad t > 0 \,, \ x \in \R^d \Big\} \,,
\end{equation}
for suitable $\beta = \beta_N$,
where we agree that $Z_N^{\beta}(z) := Z_{\lfloor N \rfloor}^{\beta}
(\lfloor z \rfloor)$ for non-integer $N, z$.

\smallskip

In dimension $d=1$, Alberts, Khanin and Quastel~\cite{AKQ14} showed that
for $\beta_N=\hbeta N^{-1/4}$,
the random field \eqref{eq:diffresc}
converges in distribution to the Wiener chaos solution $u(t,x)$ of the
one-dimensional stochastic heat equation (SHE)
\be\label{eq:SHE}
\frac{\partial u(t,x)}{\partial t} = \frac{1}{2}\Delta u(t,x) + \hat\beta
\, \dot{W}(t,x) \, u(t,x) \,, \qquad
u(0,x) \equiv 1 \,,
\ee
where $\dot{W}$ is space-time white noise on $\R\times\R$.
The existence of such an intermediate disorder regime is a general phenomenon among models
that are so-called {\em disorder relevant}, see \cite{CSZ17a}, and the directed polymer in
dimension $d=1$ is one such example.

\smallskip

A natural question is whether an intermediate disorder regime also exists
for the directed polymer in
dimension $d=2$. We gave an affirmative answer in \cite{CSZ17b}, although the problem
turns out to be much more subtle than $d=1$. The
standard Wiener chaos approach fails,
because the model in $d=2$ is so-called marginally relevant, or critical.
We will further elaborate on this later.
Let us recall the results from \cite{CSZ17b}, which provide the starting point of this paper.

Henceforth we focus on $d=2$,
so $S = (S_n)_{n\in\N_0}$ is the simple random walk on $\Z^2$.
Let
\begin{equation}\label{eq:Zeven}
	\Z^{k}_{\rm even}:= \{ (z_1, \ldots, z_k)\in \Z^k: z_1+\cdots +z_k \text{ is even}\}.
\end{equation}
Due to periodicity, if we take $S_0\in \Z^2_{\rm even}$,
then $(n, S_n)\in \Z^{3}_{\rm even}$ for all $n\in\N$.
The transition probability kernel of $S$ will be denoted by
\begin{equation}\label{eq:qas}
	q_n(x) := \P(S_n=x\,|\,S_0=0) = \big( \, g_{n/2}(x) + o(\tfrac{1}{n})\big)
	\, 2 \, \ind_{\{(n,x) \in \Z^3_{\mathrm{even}}\}}  \qquad \text{ as $n\to\infty$},
\end{equation}
by the local central limit theorem,
where $g_u(\cdot)$ is the standard Gaussian density on $\R^2$:
\be\label{eq:gu}
	g_u(x) := \frac{1}{2\pi u} \, e^{-\frac{|x|^2}{2u}} \,, \qquad
	u > 0\,, \ x \in \R^2 \,.
\ee
For notational convenience, we will drop
the conditioning in \eqref{eq:qas} when the random walk starts from zero. 
The multiplicative factor $2$ in \eqref{eq:qas} is due to periodicity,
while the Gaussian density $g_{n/2}(x)$
is due to the fact that at time $n$,
the walk has covariance matrix $\frac{n}{2} I$.

The overlap (expected number of encounters)
of two independent
simple symmetric random walks $S$ and $S'$ on $\Z^2$ is defined by
\be\label{eq:olap}
R_N:= \sum_{n=1}^N \P(S_n=S'_n) = \sum_{n=1}^N \sum_{x\in \Z^2} q_n(x)^2
= \sum_{n=1}^N q_{2n}(0) = \frac{\log N}{\pi} \big(1+o(1)\big)
\ee
where the asymptotic behavior follows from \eqref{eq:qas}.
It was shown in \cite{CSZ17b} that the correct choice of the disorder strength
is $\beta=\beta_N= \hbeta/\sqrt{R_N}$. More precisely, denoting by $W_1$
a standard normal, we have the
following convergence in distribution:
\be\label{eq:Zlim}
	Z^{\beta_N}_N
	\xrightarrow[\,N\to\infty\,]{d}
	\begin{cases}
	\exp\Big(\sigma_{\hbeta} W_1- \frac{1}{2}\sigma^2_\hbeta\Big) &
	\text{if } \hat \beta < 1 \\
	0 & \text{if } \hat\beta \ge 1
	\end{cases} \,, \qquad
	\text{where} \quad \sigma^2_\hbeta:= \log \tfrac{1}{1-\hbeta^2} \,.
\ee
This establishes a weak to strong disorder phase transition in $\hbeta$
(with critical point $\hbeta_c=1$), similar to what was known for the directed
polymer model in $\Z^{d+1}$ with $d\geq 3$ \cite{C17}.
It was also proved in \cite[Theorem 2.13]{CSZ17b} that for $\hbeta < 1$,
after centering and rescaling, the random field
of partition functions \eqref{eq:diffresc} converges
to the solution of the SHE with \emph{additive}
space-time white noise, known as Edwards-Wilkinson fluctuation.
Similar results have been recently obtained in \cite{GRZ} for the SHE
with multiplicative noise.

\smallskip

The behavior at the critical
point $\hbeta = \hat\beta_c$,
i.e.\ $\beta_N = 1 / \sqrt{R_N}$, is quite subtle. For each $x\in \R^2$
and $t > 0$, the
partition function $Z^{\beta_N}_{Nt}(x\sqrt{N})$ converges to zero in distribution
as $N\to\infty$, by \eqref{eq:Zlim},
while its expectation is identically one,
see \eqref{eq:Zx}, and its second moment diverges. This suggests that the random field
$x \mapsto Z^{\beta_N}_{Nt}(x\sqrt{N})$ becomes \emph{rough} as $N\to\infty$,
so we should look at it as a \emph{random distribution on $\R^2$}
(actually a random measure, see below).
We thus average the field in space and define
\be \label{eq:Qavgn}
	Z^{\beta_N}_{Nt} (\phi)
	:= \frac{1}{N} \sum_{x \in \frac{1}{\sqrt{N}}\Z^2}
	\phi(x) \, Z^{\beta_N}_{Nt}\big( x \sqrt{N} \big) \,, \qquad
	\text{for} \quad \phi\in C_c(\R^2) \,.
\ee

The first moment of $Z^{\beta_N}_{Nt} (\phi)$ is easily computed by Riemann sum approximation:
\begin{equation}
	\label{eq:CphipsiQfree}
	\lim_{N\to\infty} \bbE\big[ Z^{\beta_N}_{Nt}(\phi) \big] \, = \,
	\lim_{N\to\infty} \frac{1}{N} \sum_{x \in \frac{1}{\sqrt{N}}\Z^2}
	\phi(x) \, = \,
	\int_{\R^2}
	\phi(x) \, \dd x \,.
\end{equation}
Our main result is the sharp asymptotic evaluation
of the \emph{second and third moments}. These will yield
important information on the convergence of the
generalized random field \eqref{eq:Qavgn}.

\smallskip

Let us first specify our choice of $\beta = \beta_N$.
Recalling that $\lambda(\cdot)$ is the log-moment generating function
of the disorder $\omega$, see \eqref{eq:genomega}, we fix $\beta_N$ such that
\be \label{eq:sigmaN}
	\sigma_N^2 := e^{\lambda(2\beta_N) - 2\lambda(\beta_N)} -1
	\underset{N\to\infty}{=} \, \frac{1}{R_N}
	\bigg(1 + \frac{\theta}{\log N} \big(1+o(1)\big) \bigg) \,,
	\quad \text{for some } \theta\in \R.
\ee
Since $\lambda(t) \sim \frac{1}{2} t^2$ as $t \to 0$,
we have $\beta_N \sim 1 / \sqrt{R_N}$,
so we are indeed exploring
a window around the critical point $\hbeta_c=1$.
Let us recall
the \emph{Euler-Mascheroni constant}:
\begin{equation}\label{eq:EM}
	\gamma \,:=\, -\int_0^\infty e^{-u} \, \log u \, \dd u \,\simeq\, 0.577 \,.
\end{equation}

\begin{remark}\label{rm:choose_beta}
The asymptotic behavior in \eqref{eq:olap} can be refined as follows:
\begin{equation} \label{eq:olap+}
	R_N = \tfrac{\log N}{\pi} + \tfrac{\alpha}{\pi} + o(1) \qquad
	\text{where} \qquad \alpha := \gamma + \log 16 - \pi\,,
\end{equation}
see \cite[Proposition 3.2]{CSZ18}.
This leads to an equivalent reformulation of \eqref{eq:sigmaN}:
\begin{equation*}
	\sigma_N^2 =\tfrac{\pi}{\log N}\big(1+\tfrac{\theta-\alpha}{\log N} (1+o(1))\big) \,.
\end{equation*}
It is possible to express this condition in terms of $\beta_N$
(see \cite[Appendix~A.4]{CSZ18}):
\begin{equation}\label{choose_beta}
	\beta_N^2 = \tfrac{\pi}{\log N} - \tfrac{\kappa_3 \, \pi^{3/2}}{(\log N)^{3/2}}
	+ \tfrac{\pi(\theta - \alpha)
	+ \pi^2 (\frac{3}{2} \kappa_3^2 - \frac{1}{2} - \frac{7}{12}\kappa_4)}
	{(\log N)^2} \big(1+o(1)\big)  \,,
\end{equation}
where
$\kappa_3,\kappa_4$ are the disorder cumulants,
i.e.\ $\lambda(t) = \frac{1}{2} t^2 +
\frac{\kappa_3}{3!} t^3 + \frac{\kappa_4}{4!} t^4 + O(t^5)$ as $t \to 0$.
\end{remark}

\smallskip

We define the
following special function:
\be\label{eq:Gtheta}
	G_\theta(w) := \int_0^\infty \frac{e^{(\theta - \gamma) s}
	\,s \, w^{s-1}}{\Gamma(s+1)}
	\, \dd s \,, \qquad w\in (0,\infty) \,.
\ee
We now state our first result,
where we compute the second moment of $Z^{\beta_N}_{Nt} (\phi)$.

\begin{theorem}[Second moment]\label{th:variance}
Let $\phi\in C_c(\R^2)$, $t > 0$, $\theta\in\R$.
Let $\beta_N$ satisfy \eqref{eq:sigmaN}. Then
\begin{align}
	\label{eq:varavgQfree}
	\lim_{N\to\infty} \bbvar\Big[Z^{\beta_N}_{Nt}(\phi) \Big] \ = \
	& \int_{\R^2 \times \R^2} \phi(z) \, \phi(z') \,
	K_{t, \theta}(z-z') \, \dd z \, \dd z' \,,
\end{align}
where the covariance kernel $K_{t, \theta}(\cdot)$ is given by
\begin{equation}\label{eq:Kt}
	K_{t, \theta}(x) := \pi \int\limits_{0 < u < v < t}
	g_u(x) \, G_{\theta}(v-u) \, \dd u \, \dd v \,.
\end{equation}
\end{theorem}

The same covariance kernel $K_{t,\theta}$ was derived by different methods by
Bertini and Cancrini \cite{BC98} for the 2d Stochastic Heat Equation,
see Subsection~\ref{sec:she}.
It is not difficult to see that
\begin{align}\label{corr_ker}
K_{t, \theta}(x)\sim C_t \log \tfrac{1}{|x|}, \quad \text{as} \quad |x|\to 0,
\end{align}
with $C_t \in (0,\infty)$,
and hence the integral in \eqref{eq:varavgQfree} is finite.

\begin{remark}[Scaling covariance]\label{rm:ZNscaling}
It is easily checked from \eqref{eq:Kt} that for any $t>0$,
\be\label{eq:Ktidentity}
K_{t, \theta}(x) = K_{1, \theta_t}(x/\sqrt t) \qquad \text{with}
\qquad \theta_t:= \theta+ \log t.
\ee
This is also clear because we can
write $Z^{\beta_N}_{Nt} (\phi) = Z^{\beta_N}_{M} (\phi_t)$
with $M := Nt$ and $\phi_t(x) := t \, \phi(\sqrt{t} x)$, see \eqref{eq:Qavgn},
and note that $\beta_N$ can be expressed as $\beta_M$,
provided $\theta$ is replaced by $\theta_t = \theta + \log t$
(just set $N$ equal to $Nt$
in \eqref{eq:sigmaN}, and recall \eqref{eq:olap+}).
\end{remark}

\smallskip

The starting point
of the proof of Theorem~\ref{th:variance} is a polynomial chaos expansion of the
partition function. The variance computation
can then be cast in a \emph{renewal theory
framework}, which is the cornerstone of our approach (see Subsection~\ref{sec:outline}
for an outline).
This allows us to capture the
much more challenging \emph{third moment}
of the field.
Let us extend the function $G_\theta(w)$ in \eqref{eq:Gtheta}
with a spatial component, recalling \eqref{eq:gu}:
\begin{gather}
	\label{eq:Gthetaux}
	G_\theta(w, x) \,:=\, G_\theta(w) \, g_{w/4}(x) \,, \qquad w > 0 \,,
	\ x \in \R^2 \,.
\end{gather}
We can now state the main result of this paper.

\begin{theorem}[Third moment]\label{th:3rdmom}
Let $\phi\in C_c(\R^2)$, $t > 0$, $\theta\in\R$.
Let $\beta_N$ satisfy \eqref{eq:sigmaN}. Then
\begin{equation} \label{eq:3rdQfree}
	\lim_{N\to\infty} \bbE\Big[ \big(\, Z_{Nt}^{\beta_N}(\phi) -
	\bbE\big[ Z_{Nt}^{\beta_N}(\phi) \big] \,\big)^3 \Big] \,=\,
	\! \int\limits_{(\R^2)^3} \!
	\phi(z) \, \phi(z') \, \phi(z'')
	\, M_{t, \theta}(z,z',z'') \, \dd z \, \dd z' \, \dd z''
	\, < \, \infty \,,
\end{equation}
where the kernel $M_{t, \theta}(\cdot)$ is given by
\begin{equation}\label{eq:kernel1}
	M_{t, \theta}(z,z',z'') \, := \,
	\sum_{m=2}^\infty
	2^{m-1} \, \pi^{m} \, \big\{ \cI^{(m)}_{t, \theta}(z, z', z'') +
	\cI^{(m)}_{t, \theta}(z', z'', z) + \cI^{(m)}_{t, \theta}(z'', z, z') \big\} \,,
\end{equation}
with $\cI^{(m)}_{t, \theta}(\cdot)$ defined as follows:
\begin{equation} \label{eq:ImQfree}
\begin{split}
	\cI^{(m)}_{t, \theta}(z, z', z'')  & :=
	\!\!\!\!\!\!\!\!
	\idotsint\limits_{\substack{0 < a_1 < b_1 < \ldots < a_m < b_m < t \\
	x_1, y_1, \ldots, x_m, y_m \in \R^2}} \!\!\!\!\!\!\!\!\!
	\dd \vec a\, \dd \vec b\, \dd \vec x\, \dd \vec y \ \,
	g_{\frac{a_1}{2}}(x_1 - z) \, g_{\frac{a_1}{2}}(x_1 - z') \,
	g_{\frac{a_2}{2}}(x_2 - z'')
	\\
	& \qquad\ \ \cdot
	\rule{0pt}{1.3em}G_{\theta}(b_1 - a_1, y_1 - x_1) \,
	g_{\frac{a_2 - b_1}{2}}(x_2 - y_1) \,
	G_{\theta}(b_2 - a_2, y_2 - x_2)   \\
	& \qquad\ \cdot  \prod_{i=3}^{m}
	g_{\frac{a_i - b_{i-2}}{2}}(x_i - y_{i-2}) \, g_{\frac{a_i - b_{i-1}}{2}}(x_i - y_{i-1})
	\, G_{\theta}(b_i - a_i, y_i - x_i) \,.
\end{split}	
\end{equation}
\end{theorem}

The expression \eqref{eq:ImQfree} reflects a key combinatorial structure
which emerges from our renewal framework.
Establishing the convergence of the series in \eqref{eq:kernel1}
is highly
non-trivial, which shows how delicate things become in the critical window.

\smallskip

We remark that
relation \eqref{eq:3rdQfree} holds also for the mixed centered third moment
with different
test functions $\phi^{(1)}, \phi^{(2)}, \phi^{(3)} \in C_c(\R^2)$,
with the same kernel $M_{t,\theta}(z, z', z'')$.
Note that this kernel is invariant under
\emph{any} permutation of its variables, because $\cI_{t,\theta}^{(m)}(z, z', z'')$
is symmetric in $z$ and $z'$ (but not in $z''$,
hence the need of symmetrization in \eqref{eq:kernel1}).

\smallskip

Let us finally come back to the convergence of the
random field $Z_{Nt}^{\beta_N}(x \sqrt{N})$ of diffusively rescaled partition functions.
By averaging with respect to a test function, as in \eqref{eq:Qavgn},
we regard this field
as a \emph{random measure on $\R^2$}. More explicitly, if we define
\be\label{eq:muab}
	\mathscr{Z}^{\beta_N}_{Nt}(\dd x)
	:= \frac{1}{N} \sum_{y\in \frac{1}{\sqrt{N}} \Z^2}
	Z^{\beta_N}_{Nt}\big(y \sqrt{N}\big) \,
	\delta_{y}(\dd x) \,,
\ee
we can write $Z_{Nt}^{\beta_N}(\phi) = \int_{\R^2}
\phi(x) \, \mathscr{Z}^{\beta_N}_{Nt}(\dd x)$, see \eqref{eq:Qavgn}.
Note that $(\mathscr{Z}^{\beta_N}_{Nt})_{N\in\N}$
is a sequence of random variables taking values in $\cM(\R^2)$, the Polish space
of locally finite measures on $\R^2$
with the vague topology (i.e.\
$\nu_n \to \nu$ in $\cM(\R^2)$ if and only if $\int \phi \, \dd \nu_n \to
\int \phi \, \dd \nu$ for any $\phi\in C_c(\R^2)$).
We can make the following remarks.
\begin{itemize}
\item The convergence of the first moment \eqref{eq:CphipsiQfree} implies
tightness of $(\mathscr{Z}^{\beta_N}_{Nt})_{N\in\N}$, see
\cite[Lemma~14.15]{Kal}. This yields the existence of weak subsequential limits:
\begin{equation*}
	\mathscr{Z}^{\beta_N}_{Nt}(\dd x) \xrightarrow[]{\ d \ }
	\bs{\mathcal{Z}}(\dd x) \qquad
	\text{as $N\to\infty$ along a subsequence} \,,
\end{equation*}
where the limit $\bs{\mathcal{Z}}(\dd x) =
\bs{\mathcal{Z}}_{t,\theta}(\dd x)$ can in principle depend
on the subsequence.

\item The convergence of the  second moment \eqref{eq:varavgQfree}
implies uniform integrability of $Z_{Nt}^{\beta_N}(\phi)$. It follows that
any subsequential limit $\bs{\mathcal{Z}}(\dd x)$ has mean
measure given by Lebesgue measure: $\bbE\big[ \int_{\R^2} \phi(x) \, \bs{\mathcal{Z}}(\dd x)
\big] = \int \phi(x) \, \dd x$.
Moreover, by \eqref{eq:varavgQfree} and Fatou's Lemma,
\be\label{eq:varFatou}
	\bbvar\bigg[
	\int_{\R^2} \phi(x) \, \bs{\mathcal{Z}}(\dd x)
	\bigg] \,\leq\, \int_{\R^2 \times \R^2} \phi(z) \, \phi(z') \,
	K_{t, \theta}(z-z') \, \dd z \, \dd z' \, < \, \infty \,.
\ee
However, this does not rule out that the variance
in \eqref{eq:varFatou} might actually vanish,
in which case the limit $\bs{\mathcal{Z}}(\dd x)$ would just be
the trivial Lebesgue measure.

\item \emph{The convergence of the third moment \eqref{eq:3rdQfree} rules out
this triviality}. Indeed, \eqref{eq:3rdQfree} implies that
$\bbE[|Z_{Nt}^{\beta_N}(\phi)|^3] \le
\bbE [Z_{Nt}^{\beta_N}(|\phi|)^3]$ is bounded, so
the squares $Z^{\beta_N}_{Nt}(\phi)^2$ are uniformly integrable
and the inequality in \eqref{eq:varFatou} is actually an equality.
\end{itemize}
We can combine the previous considerations in the following result.

\begin{theorem}\label{th:field}
Let $t > 0$, $\theta \in \R$. Let $\beta_N$ satisfy \eqref{eq:sigmaN}.
The random measures $(\mathscr{Z}^{\beta_N}_{Nt}(\dd x) )_{N\in\N}$
in \eqref{eq:muab} admit weak subsequential limits
$\bs{\mathcal{Z}}_{t,\theta}(\dd x)$, and any such limit satisfies
\begin{align}
	\bbE\bigg[ \int_{\R^2} \phi(x) \, \bs{\mathcal{Z}}_{t,\theta}(\dd x)
	\bigg] &\,=\, \int \phi(x) \, \dd x
	\label{eq:nutheta1}\\
	\bbvar\bigg[ \int_{\R^2} \phi(x) \, \bs{\mathcal{Z}}_{t,\theta}(\dd x)
	\bigg]
	& \,=\, \int_{\R^2 \times \R^2} \phi(z) \, \phi(z') \,
	K_{t, \theta}(z-z') \, \dd z \, \dd z' \, \label{eq:nutheta2}\\
	\bbE\bigg[ \, \bigg| \int_{\R^2} \phi(x) \,
	\bs{\mathcal{Z}}_{t,\theta}(\dd x) \bigg|^3 \,
	\bigg]
	& \,<\,\infty \,. \label{eq:nutheta3}
\end{align}
\end{theorem}

In particular, every weak subsequential limit
$\bs{\mathcal{Z}}_{t,\theta}(\dd x)$ is a random measure
with the same covariance structure. It is natural to conjecture that
the whole sequence
$(\mathscr{Z}^{\beta_N}_{Nt}(\dd x) )_{N\in\N}$ has a weak limit,
but this remains to be proved.

\smallskip

We conclude with a remark on \emph{intermittency}.
As the asymptotics behavior \eqref{corr_ker} suggests,
when we fix the starting point
of the partition function instead of
averaging over it, i.e.\ we consider $Z_N^{\beta_N}$
defined in \eqref{eq:Z},
the second moment blows up like $\log N$.
More precisely, in \cite[Proposition~A.1]{CSZ18} we have
shown that as $N\to\infty$
\begin{equation} \label{eq:inter2} 	
	\bbE\big[ (Z_{N}^{\beta_N})^2 \big]
	\sim c \,(\log N) \,, \qquad
	\text{with} \qquad c = \textstyle \int_0^1 G_{\theta}(t) \, \dd t \,.
\end{equation}
This is a signature of intermittency, because it shows that
$\bbE\big[ (Z_{N}^{\beta_N})^2 \big]
\gg \bbE[Z_N^{\beta_N}]^2 = 1$.
It also implies that for any $q \ge 2$ we have the bound
\begin{equation} \label{eq:inter2b} 	
	\bbE\big[ (Z_{N}^{\beta_N})^q \big]
	\ge c' \,(\log N)^{q-1} \,.
\end{equation}
Indeed, since
$\bbE[Z_N^{\beta_N}] = 1$,
we can introduce the size-biased probability $\bbP^*(A) :=
\bbE[ \ind_A \, Z_N^{\beta_N}]$ and note that
$\bbE\big[ (Z_{N}^{\beta_N})^q \big] = \bbE^*\big[ (Z_{N}^{\beta_N})^{q-1} \big]
\ge \bbE^*\big[ Z_{N}^{\beta_N} \big]^{q-1}
= \bbE\big[ (Z_{N}^{\beta_N})^2 \big]^{q-1}$ by Jensen.

\begin{remark}
We formulated our results only for the directed polymer on $\Z^{2+1}$, but
our techniques carry
through for other marginally relevant directed polymer type models,
such as the disordered pinning
model with tail exponent $1/2$, and the directed polymer on $\Z^{1+1}$
with Cauchy tails (see~\cite{CSZ17b}).
\end{remark}

\subsection{The $2d$ stochastic heat equation}
\label{sec:she}
An analogue of Theorem~\ref{th:variance} for the stochastic heat
equation (SHE) in $\R^2$
was proved by Bertini and Cancrini in
\cite{BC98}, although they did not obtain the analogue of Theorem \ref{th:3rdmom}. We
formulate these results next.

\smallskip

The SHE as written in \eqref{eq:SHE} is ill-posed due to the product $\dot{W}\cdot u$.
To make sense of it, we mollify the space-time white noise $\dot{W}$ in the
space variable. Let
$j\in C^\infty_c(\R^2)$ be a probability density on $\R^2$ with $j(x)=j(-x)$,
and let
$$
J := j*j.
$$
For $\epsilon>0$, let $j_\epsilon(x) := \epsilon^{-2} j(x/\epsilon)$.
Then the space-mollified noise $\dot{W}^\epsilon$ is defined
by $\dot{W}^\epsilon(t, x)  := \int_{\R^2} j_\epsilon(x-y) \dot{W}(t, y)\dd y$. We
consider the mollified equation
\begin{equation}\label{eq:2dSHEeps}
	\frac{\partial u^\eps}{\partial t} = \frac{1}{2}\Delta u^\eps + \beta_\eps
	\, u^\eps \,
	\dot{W}^\eps \,, \qquad u^\eps(0, x)= 1 \ \ \forall\, x\in \R^2 \,,
\end{equation}
which admits a unique mild solution (with Ito integration).

It was shown in \cite{CSZ17b} that if we rescale
$\beta_\eps:= \hbeta \sqrt{\frac{2\pi}{\log \eps^{-1}}}$, then
for any fixed $(t,x)\in \R^+\times\R^2$
the mollified solution $u^\eps(t,x)$ converges in distribution as $\epsilon \to 0$
to the same limit as in \eqref{eq:Zlim} for the directed polymer partition function,
with $\hbeta_c=1$ being the critical point.

In \cite{BC98}, Bertini and
Cancrini considered the critical window around $\hat\beta_c =1$ given by
\be\label{eq:betaeps2}
	\beta_\eps^2 = \frac{2\pi}{\log \frac{1}{\eps}}
	+ \frac{\rho+o(1)}{(\log \frac{1}{\eps})^2} \,, \qquad \text{with } \rho \in \R \,.
\ee
This is comparable to our choice of $\beta_N$,
see \eqref{eq:sigmaN} and \eqref{choose_beta}, if we make the
identification $\eps^2=1/N$ (note that the third cumulant $\kappa_3=0$ for Gaussian random
variables). In this critical window, $u^\eps(t,x)$ converges to $0$ in distribution, while
its expectation is constant:
\begin{equation}\label{eq:SHE1stmom}
	\bbE\bigg[\int_{\R^2} \phi(x) u^\eps(t, x)\dd x\bigg] \ \equiv \
	\int_{\R^2} \phi(x) \, \dd x \,.
\end{equation}
Bertini and Cancrini showed that when interpreted as
a random distribution on $\R^2$, $u^\eps(t,\cdot)$ admits subsequential weak limits, and they
computed the limiting covariance.
This is the analogue of our Theorem~\ref{th:variance}, which we now
state explicitly.
Let us set
\begin{equation*}
	u^\epsilon(t,\phi):=\int_{\R^2} \phi(x) \, u^\epsilon(t,x)\,\dd x \,,
	\qquad \text{for } \phi \in C_c(\R^2) \,.
\end{equation*}

\begin{theorem}[\cite{BC98}]\label{th:SHEvariance}
Let $\beta_\eps$ be chosen as in \eqref{eq:betaeps2}.
Then, for any $\phi\in C_c(\R^2)$,
\begin{align}
	\label{eq:SHE2ndmom}
	\lim_{\eps \to 0^+} \bbvar\big[ u^\epsilon(t,\phi) \big] \ = \
	& 2 \, \int_{\R^2 \times \R^2} \phi(z) \, \phi(z') \,
	K_{t, \theta}\big(\tfrac{z-z'}{\sqrt 2}\big) \, \dd z \, \dd z' \,,
\end{align}
where $K_{t, \theta}$ is defined as in Theorem \ref{th:variance}, with
\be\label{formula_theta}
  \theta =  \log 4 +2\int_{\R^2}\int_{\R^2} J(x) \log \frac{1}{|x-y|} \,J(y)\,\dd x \dd y -\gamma+\frac{\rho}{\pi}.
\ee
\end{theorem}

In Section \ref{sec:SHE} we provide an independent proof of
Theorem \ref{th:SHEvariance}, which employs the renewal framework
of this paper.
Note that, by Feynman-Kac formula, the mollified solution $u^\epsilon(t,\phi)$
can be interpreted as the partition function of
a continuum directed polymer model.

\begin{remark}
The  covariance kernel
in \eqref{eq:SHE2ndmom} coincides with the one in \cite[eq. (3.14)]{BC98},
provided we identify the parameter $\beta$
in \cite{BC98} with
$e^{\theta - \gamma}$. If we plug $\beta = e^{\theta - \gamma}$ into \cite[eq. (2.6)]{BC98},
with $\theta$ given by \eqref{formula_theta}, we obtain precisely \eqref{eq:betaeps2}.
\end{remark}

Our renewal framework leads to analogues of Theorems~\ref{th:3rdmom}
and~\ref{th:field} for the SHE. For simplicity, we content ourselves
with showing that the third moment is bounded, but the same techniques
would allow to compute its sharp asymptotic behavior, as
in \eqref{eq:3rdQfree}-\eqref{eq:ImQfree}.

\begin{theorem}\label{th:SHE3rdmom}
Follow the same assumptions and notation as in Theorem \ref{th:SHEvariance}.
Then
\begin{align*}
	\sup_{\epsilon>0} \, \E \Big[ \big(u^\epsilon(t,\phi) -
	{\textstyle\int_{\mathbb{R}^2}} \phi(x)\,\dd x \big)^3 \Big] <\infty.
\end{align*}
If $\bs{u}_\theta(t,\cdot)$ is any subsequential weak limit
in $\cM(\R^2)$ of
$u^\eps(t, \cdot)$ as $\eps\to 0^+$, then
$\bs{u}_\theta(t,\cdot)$ satisfies
the analogues of \eqref{eq:nutheta1}--\eqref{eq:nutheta3},
with $K_{t, \theta}(z-z')$ in \eqref{eq:nutheta2} replaced by
 $2K_{t, \theta}\big(\frac{z-z'}{\sqrt 2}\big)$.
\end{theorem}

\subsection{Outline of the proof strategy}
\label{sec:outline}

We present the key ideas of our approach. First we
compute the \emph{second moment} of the partition function,
sketching the proof of \eqref{eq:inter2}. Then we describe
the combinatorial structure of the \emph{third moment}, which
leads to Theorem~\ref{th:3rdmom}.
This illustrates how \emph{renewal theory} emerges
in our problem.

\medskip
\noindent
\emph{Second moment.}
We start from a {\it polynomial chaos expansion}
of the partition function $Z_{N}^{\gb}$, which arises from a binomial
expansion of the exponential in
\eqref{eq:Z}
(see Subsection~\ref{sec:poly}):
\begin{equation}\label{eq:polydecom}
\begin{split}
	Z_{N}^{\gb}  =  1
	\,+\, \sum_{k\ge 1} \, \sumtwo{0 < n_1 < \ldots < n_k \le N}{x_1, \ldots, x_k \in \Z^2}
	q_{n_1}(x_1) \, \xi_{n_1,x_1} \, \cdot & \,
	q_{n_2 - n_1}(x_2 - x_1) \, \xi_{n_2,x_2} \, \cdot \\
	& \qquad \cdot \,\ldots \, \cdot \, q_{n_k - n_{k-1}}(x_k - x_{k-1}) \, \xi_{n_k,x_k} \,,
\end{split}
\end{equation}
where we set $\xi_{n,x} = e^{\beta_N\go_{n,x} -\lambda(\beta_N)}-1$ for $n\in \N,x\in\Z^2$.
Note that $\xi_{n,x}$ are i.i.d.\ with mean zero and variance
$\sigma^2 = e^{\lambda(2\beta)-2\lambda(\beta)}-1$, see \eqref{eq:sigmaN}. Then
\begin{align}
	\notag
	\bbvar[Z_{N}^\beta]
	& =  \sum_{k\ge 1} \, (\sigma^2)^k
	\sumtwo{0 < n_1 < \ldots < n_k \le N}{x_1, \ldots, x_k \in \Z^2}
	q_{n_1}(x_1)^2 \, \cdot \,
	q_{n_2 - n_1}(x_2 - x_1)^2  \, \cdots \,
	q_{n_k - n_{k-1}}(x_k - x_{k-1})^2 \\
	\label{eq:quasi-ren}
	&=  \sum_{k\ge 1} \, (\sigma^2)^k \sum_{0 < n_1 < \ldots < n_k \le N}
	u_{n_1}^2 \, \cdot \,
	u_{n_2 - n_1}^2 \, \cdots \,
	u_{n_k - n_{k-1}}^2 \,,
\end{align}
where we define
\begin{equation} \label{eq:un}
	u_{n}^2 := \sum_{x\in\Z^2} q_n(x)^2 = q_{2n}(0) = \frac{1}{\pi n}
	\, + \, O\bigg(\frac{1}{n^2}\bigg) \,.
\end{equation}
Incidentally, \eqref{eq:quasi-ren} coincides with the variance of the partition
function of the one-dimensional
disordered pinning model based on the simple random walk on $\Z$
\cite{CSZ18}.

The key idea is to view the series of convolutions \eqref{eq:quasi-ren}
through the lenses of \emph{renewal theory}.
The sequence $u_n^2$ is not summable, but we can normalize it to a probability
on $\{1,\ldots, N\}$.
We thus define a triangular array of independent random variables
$(T_i^{(N)})_{i\in \N}$ by
\begin{align} \label{eq:XN}
	\P\big(T_i^{(N)}=n\big)
	= \frac{1}{R_N} \, u_n^2 \,\ind_{\{1\leq n \leq N\}}  \,,
	\qquad \text{where} \qquad R_N:=
	\sum_{n=1}^N u_n^2 \,.
\end{align}
We stress that $R_N = \frac{1}{\pi} \log N  + O(1)$ is the same as in \eqref{eq:olap}.
If we fix $\beta_N$ satisfying \eqref{eq:sigmaN},
and define the renewal process
\begin{equation}\label{eq:tau}
	\tau^{(N)}_k = T^{(N)}_1 + \ldots + T^{(N)}_k \,,
\end{equation}
we can rewrite \eqref{eq:quasi-ren} for $\beta = \beta_N$ as follows:
\begin{equation}\label{eq:quasi-ren2}
	\bbvar\big[ Z_{N}^{\beta_N} \big] = \sum_{k\ge 1}
	\big( \sigma_N^2 \, R_N \big)^k \, \P\big( \tau^{(N)}_k \le N \big)
	= \sum_{k\ge 1} e^{\theta \frac{k}{\log N}
	+ O\left(\frac{k}{(\log N)^2}\right)}
	\, \P\big( \tau^{(N)}_k \le N \big) \,.
\end{equation}
This shows that $\bbvar\big[ Z_{N}^{\beta_N} \big]$
can be interpreted
as a (weighted) \emph{renewal function} for $\tau_k^{(N)}$.

\smallskip

The renewal process $\tau_k^{(N)}$ is investigated in
\cite{CSZ18}, where we proved that
$(\tau_{\lfloor s \log N \rfloor}^{(N)}/N)_{s \ge 0}$ converges
in law as $N\to\infty$ to a special L\'evy process $Y = (Y_s)_{s \ge 0}$,
called the \emph{Dickman subordinator},
which admits an explicit density:
\begin{align} \label{eq:fst}
	f_s(t) := \frac{\P(Y_s \in \dd t)}{\dd t}
	=\frac{e^{-\gamma s} \, s \, t^{s-1}}{\Gamma(s+1)} \qquad
	\text{for } t \in (0,1) \,.
\end{align}
Then $\P(\tau_{\lfloor s \log N \rfloor}^{(N)} \le N)
\to \P(Y_s \le 1) = \int_0^1 f_s(t) \, \dd t$, and by
Riemann sum approximation
\begin{equation*}
	\lim_{N\to\infty} \frac{\bbvar\big[ Z_{N}^{\beta_N} \big]}{\log N}
	\,=\, \int_0^1 \dd t \,
	\bigg( \int_0^\infty \dd s \, e^{\theta s} \, f_s(t) \bigg)
	\,=\, \int_0^1 \dd t \, G_{\theta}(t) \,,
\end{equation*}
where $G_\theta(\cdot)$ is the same as in \eqref{eq:Gtheta},
which can now be interpreted as a renewal function for the L\'evy process $Y$.
This completes the derivation of \eqref{eq:inter2}.

Similar arguments can be applied
to the partition function $Z_N^{\beta_N}(\phi)$ averaged over
the starting point, to prove Theorem~\ref{th:variance} using renewal theory.

\medskip

\noindent\emph{Third moment.}
The proof of Theorem~\ref{th:3rdmom} is more challenging.
In the second moment computation, the
spatial variables $x_1, \ldots, x_k$ have
been summed over to get \eqref{eq:quasi-ren},
reducing the analysis to a one-dimensional
renewal process.
Such a reduction is not possible for Theorem~\ref{th:3rdmom}.
In addition to the
``point-to-plane'' partition functions \eqref{eq:Z}-\eqref{eq:Zx},
it will be important to consider \emph{point-to-point partition functions},
where we also fix the endpoint $S_N$:
\begin{equation*}
	Z_N^{\beta}(0,y)
	:= \E_0\Big[e^{\sum_{n=1}^{N-1} (\beta\omega_{n, S_n} -\lambda(\beta))}
	\, \ind_{\{S_N = y\}}\Big] \,.
\end{equation*}
We need to extend our renewal theory framework,
enriching the process $\tau^{(N)}_k$
with a spatial component $S^{(N)}_k$
(see \eqref{eq:bsXN}-\eqref{eq:tauS} below).
This will yield the following analogue of \eqref{eq:quasi-ren2}:
\begin{equation}\label{eq:quasi-ren3}
	\sigma_N^2 \, \bbE\big[ Z_{M}^{\beta_N}(0,y)^2 \big] = \sum_{k\ge 1}
	\big( \sigma_N^2 \, R_N \big)^k \, \P\big( \tau^{(N)}_k = M,
	S^{(N)}_k = y \big) \,,
\end{equation}
which is now a \emph{local} (weighted) renewal function for the
random walk $(\tau^{(N)}_k, S^{(N)}_k)_{k\ge 0}$.
Its asymptotic behavior as $N\to\infty$ was determined in
\cite{CSZ18}:
\begin{equation}\label{eq:aszptp}
	\sigma_N^2 \, \bbE\big[ Z_{M}^{\beta_N}(0,y)^2 \big]
	\sim \, \frac{\log N}{N^2} \, G_{\theta}
	\big( \tfrac{M}{N}, \tfrac{y}{\sqrt{N}} \big) \,,
\end{equation}
where $G_\theta(t,z)$, defined in \eqref{eq:Gthetaux}, is
a continuum local renewal function.

\medskip

We now explain how the second moment of the point-to-point
partition function \eqref{eq:quasi-ren3} enters in the third moment computation.
We consider the partition function $Z_N^{\beta}$ started
at the origin, see \eqref{eq:Z},
but everything extends to the averaged partition
function $Z_N^{\beta}(\phi)$.

We compute $\bbE[(Z_N^{\beta} - 1)^3]$ using the expansion
\eqref{eq:polydecom}. This leads to a sum
over \emph{three sets of coordinates} $(n^a_i,x^a_i)$,
$(n^b_j,x^b_j)$, $(n^c_l,x^c_l)$,
with associated random variables $\xi_{n,x}$, say
\begin{equation}\label{eq:30}
	\bbE[(Z_N^{\beta} - 1)^3] =
	\sumthree{k^a \ge 1}{k^b \ge 1}{k^c \ge 1} \,
	\sumthree{(n^a_i,x^a_i)_{i=1, \ldots, k^a}}
	{(n^b_j,x^b_j)_{j=1, \ldots, k^b}}
	{(n^c_l,x^c_l)_{l=1, \ldots, k^c}}
	\!\! c_{N, \{(n^a_i,x^a_i), (n^b_j,x^b_j), (n^c_l,x^c_l)\}}
	\, \bbE \bigg[
	\prod_{i,j,l} \xi_{n^a_i,x^a_i} \, \xi_{n^b_j,x^b_j} \, \xi_{n^c_l,x^c_l}
	\bigg] ,
\end{equation}
for suitable (explicit) coefficients $c_{N,\{\ldots\}}$.
The basic observation is that if a coordinate, say $(n^a_i, x^a_i)$,
is distinct from all other coordinates,
then it gives no contribution to \eqref{eq:30},
because the random variable $\xi_{n_i,x_i}$ is independent
of the other $\xi_{n,x}$'s and it has $\bbE[\xi_{n_i,x_i}]=0$.
This means that \emph{the coordinates in \eqref{eq:30}
have to match}, necessarily in pairs or in triples.\footnote{Note that
coordinates $(n^\alpha_i,x^\alpha_i)$ with the same label
$\alpha \in \{a,b,c\}$
are distinct, by $n^\alpha_i < n^\alpha_{i+1}$, see \eqref{eq:polydecom},
hence more than triple matchings cannot occur.}
We will show that triple matchings can be neglected,
so we restrict to pairwise matchings.

Let $\bs{D} \subseteq \{1,\ldots,N\}\times\Z^2$ be the subset of space-time
points given by the union of all coordinates
$(n^a_i,x^a_i)$, $(n^b_j,x^b_j)$,
$(n^c_l,x^c_l)$ in \eqref{eq:30}.
By the pairwise matching constraint, any index $(n,x) \in \bs{D}$
must appear exactly twice among the three sets of coordinates
with labels $a,b,c$. So we can label each index in $\bs{D}$
as either $ab$, $bc$ or $ac$, and we say
that consecutive indexes with the same label form a \emph{stretch}.
This decomposition into stretches will lead to the integral
representation \eqref{eq:ImQfree} for the third moment, as we now explain.

Let us write $\bs{D} = \{(n_i, x_i): \ i=1,\ldots, r\}$ and
consider the case when the first stretch has, say, label $ab$ and length $k \le r$
(this means that
$(n_i, x_i) = (n^a_i,x^a_i) = (n^b_i, x^b_i)$ for $i=1, \ldots, k$).
The key observation is that, if we fix the last index
$(n_k, x_k) = (M, y)$ and sum over the number $k$
and the locations $(n_i,x_i)$ of previous indexes inside the stretch,
then we obtain an expression similar to \eqref{eq:quasi-ren},
except that the last index is not summed
but rather fixed to $(n_k, x_k) = (M, y)$
(see Section~\ref{sec:witout3} for the details).
But this turns out to be precisely the second moment
\eqref{eq:quasi-ren3} of the point-to-point partition function
$Z_M^{\beta}(0,y)$.

In summary, when computing the third moment from \eqref{eq:30},
\emph{the contribution of each stretch of pairwise matchings
is given asymptotically by \eqref{eq:aszptp}}.
This is also the case
when we consider the partition function $Z_N^{\beta_N}(\phi)$ averaged over
the starting point.

We can finally explain
qualitatively the structure of the kernel \eqref{eq:kernel1}-\eqref{eq:ImQfree}
in Theorem~\ref{th:3rdmom}:
\begin{itemize}
\item the index $m$ of the sum in \eqref{eq:kernel1} corresponds
	to the number of stretches;
\item each stretch gives rise to a
kernel $G_{\theta}(b_i - a_i, y_i - x_i)$ in \eqref{eq:ImQfree},
by \eqref{eq:aszptp};
\item the switch from a stretch to the following consecutive stretch gives rise to
the remaining kernels $g_{\frac{a_i - b_{i-2}}{2}}(x_i - y_{i-2}) \,
g_{\frac{a_i - b_{i-1}}{2}}(x_i - y_{i-1})$ in \eqref{eq:ImQfree}.
\end{itemize}
We stress that the knowledge of precise asymptotic estimates
such as \eqref{eq:aszptp} is crucial to compute the limiting
expression \eqref{eq:kernel1}-\eqref{eq:ImQfree} for the third moment.

We refer to Section~\ref{sec:witout3} for a more detailed exposition of the combinatorial
structure in the third moment calculation, which
lies at the heart of the present paper.

\subsection{Discussion}
To put our results in perspective, we explain here some background. The key
background notion is
disorder relevance/irrelevance. The directed polymer is an example of a disordered system
that arises as a disorder perturbation of an underlying pure model, the random walk $S$ in
this case. A fundamental question is whether the disorder
perturbation, however small $\beta>0$ is, changes the qualitative behavior of the
pure model as
$N\to\infty$. If the answer is affirmative, then disorder is said to be {\em relevant};
otherwise disorder is said to be {\em irrelevant}. For further background,
see e.g. the monograph~\cite{G10}.

For the directed polymer on $\Z^{d+1}$, the underlying random walk $S$ is diffusive with
$|S_N|\approx N^{1/2}$, while under the polymer measure $P^\beta_{N}$, it has been
shown that for $d\geq 3$, there exists a critical value $\beta_c(d)>0$ such that
for $\beta<\beta_c(d)$, $|S_N|\approx N^{1/2}$ (see e.g.~\cite{CY06});
while for any $\beta>0$ in $d=1, 2$ and for $\beta>\beta_c(d)$ in $d\geq 3$,
it is believed that $|S_N|\gg N^{1/2}$. Thus the directed polymer model
should be disorder irrelevant in $d\geq 3$, disorder relevant in $d=1$, while $d=2$
turns out to be the critical dimension separating disorder relevance vs irrelevance,
and disorder should be {\em marginally relevant}.

In \cite{AKQ14}, Alberts, Khanin and Quastel showed that on the intermediate disorder scale
$\beta_N= \hbeta/N^{1/4}$, the rescaled partition functions of the
directed polymer on $\Z^{1+1}$ converges to the solution of the 1-dimensional
SHE \eqref{eq:SHE}.
We note that the idea of considering polymers with scaled temperature had already
appeared in the physics literature \cite{BD00, CDR10}.

Inspired in particular by \cite{AKQ14}, we developed in \cite{CSZ17a} a new perspective
on disorder relevance vs irrelevance
(see also \cite{CSZ16}). The heuristic is that, if a model is disorder
relevant, then under coarse graining and renormalization of space-time, the
effective disorder strength of the coarse-grained model diverges. Therefore to compensate,
it should be possible to choose the disorder strength $\beta_N\downarrow 0$ (known as
weak disorder limit) as the lattice spacing $\delta:=1/N\downarrow 0$ (known as
continuum limit) in such a way that we obtain a continuum disordered model. In
particular, the partition function $Z^\omega_{N, \beta_N}$ should admit a non-trivial
random limit for suitable choices of $\beta_N\downarrow 0$. In \cite{CSZ17a}, we
formulated general criteria for the partition functions of a disordered system to have
non-trivial continuum and weak disorder limits. These criteria were then verified for
the disordered pinning model, a family of
(possibly long-range) directed polymer on $\Z^{1+1}$, and
the random field perturbation of the critical Ising model on $\Z^2$. However, the
general framework developed in \cite{CSZ17a}
does not include models where disorder is only marginally relevant, such
as the directed polymer on $\Z^{2+1}$,
which led to our previous work \cite{CSZ17b} and to our current work.

Disorder relevance/irrlevance is also closely linked to the classification of
singular stochastic partial differential equations (SPDE), such as the SHE or the
KPZ equation, into sub-critical, critical, or super-critical ones, which
correspond respectively to disorder relevance, marginality and
disorder irrelevance. For sub-critical singular SPDEs, a general solution theory called
{\em regularity structures} has been developed in seminal work by Hairer in
\cite{H13, H14}, and alternative approaches have been developed by Gubinelli, Imkeller,
and Perkowski~\cite{GIP15}, and also by Kupiainen~\cite{K14}. However, for critical
singular SPDEs such as the SHE in $d=2$, the only known results so far are: our previous
work \cite{CSZ17a}, which established a phase transition in the intermediate disorder
scale $\beta_\eps = \hbeta (2\pi/\log \frac{1}{\eps})^{1/2}$ and identified the limit
in distribution of
the solution $u^\eps(t,x)$ in the subcritical regime $\hbeta<1$; the work of Bertini
and Cancrini~\cite{BC98}, which computed the limiting covariance of the random
field $u^\eps(t, \cdot)$ at the critical point $\hbeta=1$; and our current work,
which establishes the non-triviality of subsequential weak limits of the random field at
the critical point $\hbeta=1$.

Let us mention some related work on the directed polymer model on the hierarchical lattice.
In particular, for the marginally relevant case, Alberts, Clark and Koci\'c in \cite{ACK17}
established the existence of a phase transition, similar to \cite{CSZ17a}.
And more recently, Clark \cite{Cl17}
computed the moments of the partition function around a critical
 window for the case of bond disorder.
 The computations in the hierarchical lattice case employ the independence structure
inherent in hierarchical models,
which is not available on $\Z^d$.

\medskip
\noindent
\emph{Note added in publication.
More recently, Gu, Quastel and Tsai \cite{GQT19} proved the existence of all moments
for the 2-dimensional SHE in the critical window. They use different, functional analytic methods inspired by Dimock and Rajeev
\cite{DR04}.}

\subsection{Organization of the paper}
In Section \ref{sec:variance}, we recall the polynomial chaos expansion for the partition functions
and introduce the renewal framework, which are then used in Section \ref{sec:thmvar} to prove
Theorem \ref{th:variance} on the limiting second moment of the partition function. In
Section \ref{sec:thirdmom}, we derive a series expansion for the third moment of the averaged
point-to-point partition functions, whose terms are separated into two groups: ones with so-called
{\em triple intersections}, and ones with {\em no triple intersection}. Terms with no triple intersection
is shown in Section \ref{sec:witout3} to converge to the desired limit, while terms with triple intersections
are shown to be negligible in Section \ref{sec:triple}, using bounds developed in
Section \ref{sec:furtherbds}. Lastly, in Section \ref{sec:SHE}, we prove Theorems \ref{th:SHEvariance}
and \ref{th:SHE3rdmom} for the stochastic heat equation.

\section{Polynomial chaos and renewal framework}
\label{sec:variance}

In this section, we describe two key elements that form the basis of our analysis:
\begin{enumerate}
\item \emph{polynomial chaos expansions},
which represent the partition function
as a multilinear polynomial of modified
disorder random varibles, see Subsection~\ref{sec:poly}.

\item a \emph{renewal theory framework}, which allows to relate
the second moment of the partition function to suitable renewal functions,
see Subsection~\ref{sec:renewal}.
\end{enumerate}

We will use $\P_{a,x}$ and $\E_{a,x}$ to
denote probability and expectation for the random walk
$S$ starting at time $a$ from position $S_a=x\in\Z^2$, with the
subscript omitted when $(a,x)=(0,0)$.
Recalling
\eqref{eq:Zeven},
we define the family of \emph{point-to-point partition functions} by
\be\label{eq:Qab}
	Z^{\beta}_{a,b}(x, y) := \E_{a,x}\Big[e^{\sum_{n=a+1}^{b-1}
	(\beta\omega_{n, S_n} -\lambda(\beta))} \ind_{\{S_b=y\}}\Big],
	\qquad (a,x), (b, y)\in \Z^{3}_{\rm even}, \ a<b \,.
\ee
The original \emph{point-to-plane partition function}
$Z_N^\beta(x)$, see \eqref{eq:Zx}, can be recovered as follows:
\be\label{eq:Qabfree}
	Z_N^\beta(x) = \sum_{ y\in \Z^2}
	Z^{\beta}_{0,N}(x,y) \,.
\ee

We note that the point-to-plane partition function
has $\bbE[Z_{a,b}^\beta(x)] \equiv 1$, while
for the point-to-point partition function we have
\begin{equation} \label{eq:EZptp}
	\bbE\big[Z^{ \gb}_{a,b}(x,y) \big] =
	q_{a,b}(x,y) := q_{b-a}(y-x) \,,
\end{equation}
the transition probability kernel defined in \eqref{eq:qas}.
We will need to average the partition functions
$Z_{a,b}^\beta(x,y)$ over either $x$ or $y$, or both,
on the diffusive scale.
More precisely, we define for $N\in\N$
\begin{align}
	Z^{N,\gb}_{a,b}(x, \psi) &:=
	\sum_{y \in \Z^2}
	 Z^{\gb}_{a,b}(x,y)  \, \psi\big(\tfrac{y}{\sqrt{N}}\big) \,,
	 \qquad \psi\in C_b(\R^2) \,,
	 \label{eq:avgxpsi}\\
	Z^{N,\gb}_{a,b}(\phi, y) &:=
	\sum_{x \in \Z^2}
	\phi\big(\tfrac{x}{\sqrt{N}}\big)  \, Z^{\gb}_{a,b}(x,y)  \, ,
	\qquad \phi\in C_c(\R^2) \,,
	\label{eq:avgphiy}\\
	Z^{N,\gb}_{a,b}(\phi, \psi) &:=
	\frac{1}{N} \sum_{x, y \in \Z^2}
	\phi\big(\tfrac{x}{\sqrt{N}}\big)
	\, Z^{ \gb}_{a,b}(x,y)  \, \psi\big(\tfrac{y}{\sqrt{N}}\big) \,,
	\qquad \phi\in C_c(\R^2) \,, \ \psi\in C_b(\R^2)
	\, . \label{avgform22}
\end{align}
The reason that the terminal function $\psi$ is only required to be bounded and continuous,
while the initial function $\phi$ is compactly supported is that, we would like to include the case
$\psi\equiv 1$, which corresponds to the point-to-plane polymer partition function. On the other hand,
the initial function $\phi$ plays the role of a test function used to average the partition function.
(In general, the fact that at least one between $\phi$ and $\psi$ is compactly supported
ensures finiteness of the average \eqref{avgform22}.)
Note that $Z_{Nt}^\beta(\phi)$ in \eqref{eq:Qavgn} coincides with
$Z_{0,Nt}^{N,\beta}(\phi,\psi)$ with $\psi \equiv 1$.
From \eqref{eq:EZptp} we compute
\begin{align} \label{eq:qNavg0}
	\bbE\big[ Z^{N,\gb}_{a,b}(x, \psi) \big] =
	q_{a,b}^N(x, \psi) & :=	\sum_{y \in \Z^2} q_{b-a}(y-x) \,
	\psi\big(\tfrac{y}{\sqrt{N}}\big), \\
	\label{eq:qNavg}
	\bbE\big[ Z^{N,\gb}_{a,b}(\phi, y) \big] =
	q_{a,b}^N(\phi, y) & :=	\sum_{x \in \Z^2} \phi\big(\tfrac{x}{\sqrt{N}}\big)	
	q_{b-a}(y-x), \\
	\label{eq:qNavg1}
	\bbE\big[ Z^{N,\gb}_{a,b}(\phi, \psi) \big] =
	q_{a,b}^N(\phi, \psi) & := \frac{1}{N} \sum_{x,y \in \Z^2}
	\phi\big(\tfrac{x}{\sqrt{N}}\big)  \, q_{b-a}(y-x)
	\, \psi\big(\tfrac{y}{\sqrt{N}}\big) \,.
\end{align}
Note that these expectations are of order $1$
for $a = 0$ and $b = N$, because $q_N(y-x) \approx 1/N$ for $x,y = O(\sqrt{N})$,
see \eqref{eq:qas}-\eqref{eq:gu}.
This explains the normalizations in \eqref{eq:avgxpsi}-\eqref{avgform22}.

\smallskip

 \subsection{Polynomial chaos expansion}
\label{sec:poly}
Let us start by rewriting the point-to-point partition function from \eqref{eq:Qab} as
 \begin{align*}
 Z_{a,b}^{\gb}(x,y)=\E_{a,x}\Big[e^{\sum_{n=a+1}^{b-1} (\beta\omega_{n, S_n} -\lambda(\beta))} \ind_{\{S_b=y\}}\Big] =
 \E_{a,x}\Big[\prod_{a<n<b}\prod_{z\in\bbZ^2} e^{ (\beta\omega_{n, z} -\lambda(\beta)) \ind_{\{S_n=z\}} } \,\ind_{\{S_b=y\}}\Big].
 \end{align*}
 Using the fact that $e^{x \ind_{\{n\in\tau\}}} = 1 + (e^x-1) \ind_{\{n\in\tau\}}$ for $x\in\R$,
we can write
\begin{equation} \label{eq:xi}
\begin{split}
	Z_{a,b}^{ \gb}(x,y)
	= & \
	\E_{a,x} \Bigg[ \prod_{n=a+1}^{b-1} \prod_{z\in \bbZ^2}
	\big(1 +  \xi_{n,z} \, \ind_{\{S_n =z\}} \big)\ind_{S_b=y}
	 \Bigg] \,, \\
	 & \ \text{where} \quad
	\xi_{n,z} := e^{\gb\omega_{n,z} -\gl(\beta)} - 1\,.
\end{split}
\end{equation}
The random variables $\xi_{n,z}$ are i.i.d. with mean zero (thanks to the normalization by $\gl(\beta)$) and with variance
 $\bbvar[\xi_{n,z}] = e^{\gl(2\gb)-2\gl(\gb)}-1$.
 Recalling \eqref{eq:EZptp}
and expanding the product, we obtain the following polynomial chaos expansion:
\begin{equation} \label{eq:polypointQ}
\begin{split}
	Z_{a,b}^{\gb}(x,y)  & =  q_{a,b}(x,y)
	\,+\, \sum_{k\ge 1} \, \sumtwo{a < n_1 < \ldots < n_k < b}{x_1, \ldots, x_k \in \Z^2}
	q_{a,n_1}(x,x_1) \, \xi_{n_1,x_1} \,\cdot\\
	& \qquad \qquad \qquad
	\qquad \qquad  \cdot \bigg\{
	\prod_{j=2}^k q_{n_{j-1}, n_j}(x_{j-1}, x_j) \, \xi_{n_j,x_j} \bigg\}
	\, q_{n_k, b}(x_k,y) \,,
\end{split}
\end{equation}
with the convention that the product equals $1$ when $k=1$.
We have written $Z_{a,b}^{\gb}(x,y)$ as a multilinear polynomial
of the random variables $\xi_{n,x}$.

Analogous expansions hold for the averaged point-to-point partition
functions: by \eqref{avgform22}
\be\label{avgmulti}
\begin{split}
	Z^{N, \beta}_{a,b}(\phi, \psi) & =  q^N_{a,b}(\phi, \psi)
	+ \frac{1}{N} \, \sum_{k\ge 1} \, \sumtwo{a < n_1 < \ldots < n_k < b}{x_1, \ldots, x_k \in \Z^2}
	q^N_{a,n_1}(\phi,x_1) \, \xi_{n_1,x_1}  \cdot\\
	& \qquad \qquad \qquad
	\qquad \qquad \,\cdot\bigg\{
	\prod_{j=2}^k q_{n_{j-1}, n_j}(x_{j-1}, x_j) \, \xi_{n_j,x_j} \bigg\}
	\, q^N_{n_k, b}(x_k,\psi) \,.
\end{split}
\ee
Similar expansions hold for $Z^{N, \beta}_{a,b}(x, \psi)$
and $Z^{N, \beta}_{a,b}(\phi, y)$, without the factor $\frac{1}{N}$.

\subsection{Renewal theory framework}
\label{sec:renewal}

Given $N\in\N$, we define a sequence
of i.i.d.\ random variables $\big( (T^{(N)}_i, X^{(N)}_i) \big)_{i\in\N}$
taking values in $\N \times \Z^2$, with marginal law
\begin{equation} \label{eq:bsXN}
	\P\big((T^{(N)}_i, X^{(N)}_i) = (n,x)\big) := \frac{q_n(x)^2}{R_N}
	\, \ind_{\{1,\ldots, N\}}(n) \,,
\end{equation}
where we recall that $q_n(x)$ is defined in \eqref{eq:qas} and
$R_N=\sum_{n=1}^N\sum_{x\in \bbZ^2} q_n(x)^2$
is the replica overlap, see \eqref{eq:olap}.
We then define the corresponding random walk\footnote{$S^{(N)}$ should not be
confused with the random walk $S$ in the definition of the directed polymer model.}
on $\N \times \Z^2$
\begin{equation} \label{eq:tauS}
	\big( \tau^{(N)}_k \,, \ S^{(N)}_k \big) :=
	\Big( T^{(N)}_1 + \ldots + T^{(N)}_k \,, \
	X^{(N)}_1 + \ldots + X^{(N)}_k \Big) \,, \qquad k \in \N \,.
\end{equation}
Note that the first component $\tau^{(N)}_k$
is the renewal process that we introduced
in Subsection~\ref{sec:outline},
see \eqref{eq:XN}-\eqref{eq:tau}.

\smallskip

We now describe
the link with our model.
We note that $\sigma_N^2$, see \eqref{eq:sigmaN},
is the variance of the random variables
$\xi_{n,x} = e^{\beta \omega_{n,x} - \lambda(\beta)}-1$
which appear in \eqref{eq:polypointQ}.
Recalling \eqref{eq:Qab} and \eqref{eq:EZptp}, we introduce a crucial
quantity $U_N(n, x)$, that will appear repeatedly in our analysis,
which is a
suitably rescaled second moment of the point-to-point partition function:
\begin{equation}\label{eq:UNVar}
\begin{split}
	U_N(n, x) & := \sigma_N^2 \, \bbE\big[ Z^{\gb_N}_{0,n}(0,x)^2 \big]
	= \sigma_N^2 \, \big\{q_n(x)^2 + \bbvar\big[Z^{\gb_N}_{0,n}(0,x) \big]\big\}
	\,, \qquad n \ge 1 \,, \\
	\rule{0pt}{1.1em}U_N(0, x) & := \delta_{x,0} = \ind_{\{x=0\}} \,.
\end{split}
\end{equation}
By \eqref{eq:polypointQ}, we then have
\begin{align} \label{def:UN}
	U_N(n, x) & = \sigma_N^2 \, q_{0,n}(0,x)^2 \, +\, \sum_{k\ge 1} (\sigma_N^2)^{k+1}
	\sumtwo{0  < n_1 < \ldots < n_k < n}{x_1, \ldots, x_k \in \Z^2}
	q_{0,n_1}(0,x_1)^2  \,\,\cdot \\
	& \qquad \qquad \qquad \qquad \qquad
	\qquad \quad \cdot \bigg\{
	\prod_{j=2}^k q_{n_{j-1}, n_j}(x_{j-1}, x_j)^2 \bigg\}
	\, q_{n_k, n}(x_k,x)^2 \,. \notag
\end{align}
Looking at \eqref{eq:bsXN}-\eqref{eq:tauS}, we have the following key
probabilistic representation:
\begin{gather} \label{def:UNrenewal}
	U_N(n,x) = \sum_{r\ge 1}
	(\lambda_N)^r \,
	\P\big(\tau^{(N)}_r = n \,, \,
	S^{(N)}_r = x \big) \,, \qquad \text{where} \quad
	\lambda_N := \sigma_N^2 \, R_N \, .
\end{gather}
It is also convenient to define
\begin{gather}\label{eq:UNnrenewal}
	U_N(n) := \sum_{x\in\Z^2} U_N(n,x) = \sum_{r\ge 1} (\lambda_N)^r \,
	\P\big(\tau^{(N)}_r = n\big) \, .
\end{gather}
Thus $U_N(n,x)$
and $U_N(n)$ can be viewed as (exponentially weighted) \emph{local
renewal functions}.

\smallskip

We investigated the asymptotic properties of the random walk
$(\tau^{(N)}_k , S^{(N)}_k)$ in \cite{CSZ18}.
In particular, introducing the rescaled process
\begin{equation} \label{eq:bsYN}
	\bsY^{(N)}_s := (Y^{(N)}_s, V^{(N)}_s):=  \Bigg(
	\frac{\tau^{(N)}_{\lfloor s \, \log N \rfloor}}{N},
	\frac{S^{(N)}_{\lfloor s \, \log N \rfloor}}{\sqrt{N}} \Bigg) \,,
	\qquad  s\geq 0 \,,
\end{equation}
we proved in \cite{CSZ18} that $\bsY^{(N)}$ converges in distribution as $N\to\infty$ to
the L\'evy process $\bsY$ on $[0,\infty) \times \R^2$
with L\'evy measure
$$
\bsnu(\dd t, \dd x) = \frac{\ind_{(0,1)}(t)}{t} \, g_{t/4}(x) \, \dd t \, \dd x,
$$
where $g_u(x)$ is the standard Gaussian density on $\R^2$, see \eqref{eq:gu}.
Remarkably, the process $\bsY$ admits an explicit density:
\begin{equation*}
	f_s(t,x) := \frac{\P(\bsY_s \in (\dd t, \dd x))}{\dd t \, \dd x}
	= \frac{e^{-\gamma s} \, s \, t^{s-1}}{\Gamma(s+1)} \,
	g_{t/4}(x) \qquad
	\text{for } t \in (0,1) \,, \ x \in \R^2 \,,
\end{equation*}
which leads to a corresponding explicit expression for the
(weighted) local renewal function
\begin{equation*}
	G_\theta(t,x) \,:=\, \int_0^\infty e^{\theta s} \, f_s(t,x) \, \dd s
	\,=\, \bigg( \int_0^\infty \frac{e^{(\theta-\gamma) s} \, s \, t^{s-1}}{\Gamma(s+1)}
	\, \dd s \bigg) \, g_{t/4}(x) \,=\, G_\theta(t) \, g_{t/4}(x) \,,
\end{equation*}
where the functions $G_\theta(t)$ and $G_\theta(t,x)$ match with
\eqref{eq:Gtheta} and \eqref{eq:Gthetaux}.

We showed in \cite{CSZ18} that the sharp asymptotic behavior of $U_N(n,x)$ and
$U_N(n)$ is captured by the functions $G_{ \theta}(n,x)$ and $G_{ \theta}(x)$.
Note that for the weight $\lambda_N$ in \eqref{def:UNrenewal}-\eqref{eq:UNnrenewal}
we can write $\lambda_N = 1 + \frac{\theta}{\log N}(1+o(1))$
as $N\to\infty$, by our assumption \eqref{eq:sigmaN}.
Then we can rephrase \cite[Theorem~1.4 and Theorems~2.3-2.4]{CSZ18} as follows.

\begin{proposition}\label{UNtime}
Fix $\beta_N$ such that \eqref{eq:sigmaN} holds, for some $\theta \in \R$.
Let $U_N(n)$ be defined as in \eqref{eq:UNnrenewal}.
For any fixed $\delta > 0$, as $N\to\infty$ we have
\begin{equation}\label{eq:UnLLT}
	U_N(n) = \frac{\log N}{N} \,
	\big(G_{\theta} \big(\tfrac{n}{N}\big) + o(1)\big) \,, \qquad
	\text{uniformly for } \ \delta N \le n \le N \,,
\end{equation}
where $G_\theta$ is defined in \eqref{eq:Gtheta}.
Moreover, there exists $C \in (0,\infty)$ such that for all $N\in\N$
\begin{align}\label{est:UNunif}
	U_N(n) \leq C \, \frac{\log N}{N} \,
	G_{\theta} \big(\tfrac{n}{N}\big) \,, \qquad \forall 1 \le n \le N \,.
\end{align}
\end{proposition}

\begin{proposition}\label{prop:UNest}
Fix $\beta_N$ such that \eqref{eq:sigmaN} holds, for some $\theta \in \R$.
Let $U_N(n,x)$ be defined as in \eqref{def:UN}-\eqref{def:UNrenewal}.
For any fixed $\delta > 0$, as $N\to\infty$ we have
\begin{equation}\label{eq:Unas}
\begin{split}
	& U_N(n,x) = \frac{\log N}{N^2} \,
	\big(G_{\theta} \big(\tfrac{n}{N}, \tfrac{x}{\sqrt{N}}\big) + o(1)\big) \,
	2 \, \ind_{\{(n,x) \in \Z^3_{\mathrm{even}}\}} \,, \\
	& \quad \rule{0pt}{1.1em}\text{uniformly for } \
	\delta N \le n \le N \,, \ |x| \le \tfrac{1}{\delta} \sqrt{N} \,,
\end{split}
\end{equation}
where $G_\theta(t,x)$ is defined in \eqref{eq:Gthetaux}.
Moreover, there exists $C \in (0,\infty)$ such that for all $N\in\N$
\begin{equation}
	\label{eq:Udiffusive}
	\sum_{x \in \Z^2: \; |x| > M \sqrt{n}}
	\frac{U_{N}(n,x)}{U_{N}(n)}
	\le \frac{C}{M^2} \,, \qquad \forall 1 \le n \le N \,, \
	\forall M > 0  \,.
\end{equation}
\end{proposition}

We will also need the following asymptotic behavior on $G_\theta(t)$ from \cite{CSZ18}.
\begin{proposition}\label{lem:Gthetaas}
For every fixed $\theta \in \R$, we have that
\begin{equation}\label{eq:Gthetaas}
	G_\theta(t) =
	\frac{1}{t \, (\log \frac{1}{t})^2}
	+ \frac{2\theta + o(1)}{t \, (\log \frac{1}{t})^3}  \qquad \text{as}\quad t \to 0 \,.
\end{equation}
It follows that there exists $c_\theta \in (0,\infty)$ such that
\begin{equation}\label{eq:Gthetaunif}
	G_\theta(t) \le \hat G_\theta(t) :=
	\frac{c_\theta}{t \, (2 + \log \frac{1}{t})^2}
	= \frac{c_\theta}{t \, (\log \frac{e^2}{t})^2} \,, \qquad \forall t \in (0,1] \,.
\end{equation}
By direct computation $\frac{\dd}{\dd t} \hat G(t) < 0$ for all $t \in (0,1)$,
hence $\hat G_\theta(\cdot)$ is strictly decreasing.
\end{proposition}

\smallskip
\section{Proof of Theorem \ref{th:variance}}\label{sec:thmvar}

Recall the definition \eqref{eq:gu} of $g_t(x)$.
Given a bounded function $\phi: \R^2 \to \R$, we define
\begin{equation}
	\label{eq:Phiax}
	\Phi_s(x) \,:=\, (\phi * g_{s/2})(x) = \int_{\R^2} \phi(x-y) \, g_{s/2}(y) \, \dd y \,,
	\qquad s > 0 \,, \ x \in \R^2 \,.
\end{equation}

The averaged partition function $Z_{Nt}^{\beta_N}(\phi)$ in Theorem~\ref{th:variance},
see \eqref{eq:Qavgn},
coincides with $Z_{0,Nt}^{N,\beta_N}(\phi,\psi)$ with $\psi \equiv 1$,
see \eqref{avgform22}. By the expansion \eqref{avgmulti} with
$\psi\equiv 1$, we obtain
\begin{equation} \label{eq:Q2avg}
	\bbvar[Z_{Nt}^{\beta_N}(\phi)] = \frac{1}{N^2} \, \sum_{k\ge 1} \,
	\,(\sigma_N^2)^k  \!\!\!\!\!\! \sumtwo{0 < n_1 < \ldots < n_k < Nt}
	{x_1, \ldots, x_k \in \Z^2} \!\!\!\!\!\!
	  q^{N}_{0,n_1}(\phi,x_1)^2
	\prod_{j=2}^k q_{n_{j-1}, n_j}(x_{j-1}, x_j)^2 .
\end{equation}
We isolate the term $k=1$,
because given $(n_1, x_1)=(m, x)$ and
$(n_k, x_k)=(n, y)$, the sum over $k\ge 2$
gives $\bbE[Z^{\gb_N}_{m,n}(x, y)^2]=U_N(n-m, y-x)/\sigma_N^2$,
by \eqref{eq:UNVar}-\eqref{def:UN}.
Therefore
\begin{align} \label{eq:Q2avgbis}
	\bbvar[Z_{Nt}^{\beta_N}(\phi)]
	 &= \frac{\sigma_N^2}{N^2} \sumtwo{0 < n < Nt}{x \in \Z^2}
	q^{N}_{0,n}(\phi,x)^2
    + \, \frac{\sigma_N^4}{N^2} \sumtwo{0 < m < n < Nt}{x,y\in\Z^2}
	q^{N}_{0,m}(\phi,x)^2 \, \bbE[Z_{m,n}^{\gb_N}(x,y)^2] \notag  \\
	&=\frac{\sigma_N^2}{N^2} \sumtwo{0 < n < Nt}{x \in \Z^2}
	q^{N}_{0,n}(\phi,x)^2
    + \, \frac{\sigma_N^2}{N^2} \sumtwo{0 < m < n < Nt}{x\in\Z^2}
	q^{N}_{0,m}(\phi,x)^2 \,U_{N}(n-m),
\end{align}
where in the second equality we summed over $y\in\Z^2$ -- this is the reason that only $U_N(n-m)$ appears instead of
$U_N(n-m,y-x)$; recall \eqref{def:UNrenewal} and \eqref{eq:UNnrenewal}.

We now let $N\to\infty$.
We first show that
the first term in the RHS of \eqref{eq:Q2avgbis}
vanishes as $O(\sigma_N^2) = O(\frac{1}{\log N})$,
see \eqref{eq:olap} and \eqref{eq:sigmaN}. Note that
for $v \in (0,1)$ and $x \in \R^2$ we have
\begin{equation}\label{eq:convPhi}
	\lim_{N\to\infty} q^{N}_{0,Nv}(\phi,\sqrt{N} x) = \Phi_v(x) \,,
	\qquad \sup_{m\in\N, \, z\in\Z^2} q^{N}_{0,m}(\phi, z) \le |\phi|_\infty < \infty \,,
\end{equation}
see \eqref{eq:qNavg}, \eqref{eq:qas} and \eqref{eq:Phiax}.
Then, by Riemann sum approximation, we have
\begin{equation*}
\begin{split}
	\frac{1}{N^2} \sumtwo{0 < n < Nt}{x \in \Z^2}
	q^{N}_{0,n}(\phi,x)^2 \,
	&
	\, \underset{N\to\infty}{\sim} \,
	\frac{1}{N^2} \sumtwo{0 < n < Nt}{x \in \Z^2}
	\Phi_{\frac{n}{N}} \big(\tfrac{x}{\sqrt{N}} \big)^2 \,
	 \\
	& \, \xrightarrow[N\to\infty]{} \,
	 \int_{(0,t) \times \R^2}
	\Phi_{v}(x)^2 \, \dd v\, \dd x \in (0,\infty) \,.
\end{split}
\end{equation*}
Indeed, the approximation is uniform for
$N\epsilon < n < Nt$, with fixed $\epsilon > 0$,
while the contribution of $n \le N \epsilon$ is small,
for $\epsilon > 0$ small, by the uniform bound in \eqref{eq:convPhi}.

It remains to focus on the second term in the RHS of \eqref{eq:Q2avgbis}.
By \eqref{eq:UnLLT}-\eqref{est:UNunif} and \eqref{eq:convPhi}, together
with $\sigma_N^2 \sim \frac{\pi}{\log N}$,
see \eqref{eq:sigmaN} and \eqref{eq:olap}, another Riemann sum approximation
gives
\begin{align}\label{limvar}
\bbvar[Z_{Nt}^{\beta_N}(\phi)] \xrightarrow[N\to\infty]{}
	&\,\, \pi \, \int\limits_{\substack{0 < u < v < t \\ x \in \R^2}}
	\Phi_{u}(x)^2 \, G_{\theta}(v-u)  \, \dd u \, \dd v
	\, \dd x \,,
\end{align}
Integrating out $x$, we obtain
\begin{equation*}
\begin{split}
	\int_{\R^2} 	\Phi_{u}(x)^2 \, \dd x & =
	\int_{\R^2} \bigg( \int_{\R^2 \times \R^2}
	\phi(z) \, \phi(z') \, g_{u/2}(x-z) \, g_{u/2}(x-z') \, \dd z \, \dd z'
	\bigg) \dd x  \\
	& = \int_{\R^2 \times \R^2}
	\phi(z) \, \phi(z') \, g_{u}(z-z') \, \dd z \, \dd z' \,,
\end{split}
\end{equation*}
which plugged into \eqref{limvar} proves \eqref{eq:varavgQfree}.\qed

\smallskip
\section{Expansion for the third moment}\label{sec:thirdmom}

In this section, we give an expansion for the third moment of the partition function,
which forms the basis of our proof of Theorem~\ref{th:3rdmom}.
We actually prove a more general version
for the averaged \emph{point-to-point}
partition functions,
which is of independent interest.

\begin{theorem}[Third moment, averaged point-to-point]\label{th:3rdmom2}
Let $t > 0$, $\theta\in\R$
and $\beta_N$ satisfy \eqref{eq:sigmaN}.
Fix a compactly supported $\phi\in C_c(\R^2)$ and a bounded
$\psi\in C_b(\R^2)$. Then
\begin{equation} \label{eq:3rdQfree2}
	\lim_{N\to\infty} \bbE\Big[
	\big( Z_{0,Nt}^{N,\beta_N}(\phi, \psi) -
	\bbE\big[ Z_{0,Nt}^{N,\beta_N}(\phi, \psi) \big] \big)^3 \Big]
	= M_{t}(\phi, \psi) :=  3 \,  \sum_{m=2}^\infty
	2^{m-1} \, \pi^{m} \,\,  \cI^{(m)}_{t}(\phi, \psi) <\infty,
\end{equation}
where we set $\Phi_s := \phi * g_{s/2}$ and $\Psi_s := \psi * g_{s/2}$,
see \eqref{eq:Phiax}, and define
\begin{equation} \label{eq:ImQfree2}
\begin{split}
	\cI^{(m)}_{t}(\phi, \psi)  & :=
	\idotsint\limits_{\substack{0 < a_1 < b_1 < a_2 < b_2 < \ldots < a_m < b_m < t \\
	x_1, y_1, x_2, y_2, \ldots, x_m, y_m \in \R^2}}
	\Phi_{a_1}^2(x_1) \, \Phi_{a_2}(x_2) \, \cdot \\
	& \qquad\ \cdot\
	G_{\theta}(b_1 - a_1, y_1 - x_1) \, g_{\frac{a_2 - b_1}{2}}(x_2 - y_1) \,
	G_{\theta}(b_2 - a_2, y_2 - x_2) \,\cdot  \\
	& \qquad\ \cdot\ \prod_{i=3}^{m}
	g_{\frac{a_i - b_{i-2}}{2}}(x_i - y_{i-2}) \, g_{\frac{a_i - b_{i-1}}{2}}(x_i - y_{i-1})
	\, G_{\theta}(b_i - a_i, y_i - x_i) \,\cdot \\
    & \qquad \ \cdot \Psi_{t-b_{m-1}}(y_{m-1}) \Psi^2_{t-b_m}(y_m) \ \dd \vec a\, \dd \vec b\,
    \dd \vec x\, \dd \vec y \,.
\end{split}	
\end{equation}
\end{theorem}

We observe that Theorem~\ref{th:3rdmom} is a special case of Theorem~\ref{th:3rdmom2}:
it suffices to take $\psi \equiv 1$ so that
$Z_{0,Nt}^{N,\beta_N}(\phi, \psi) = Z_{Nt}^{\beta_N}(\phi)$,
see \eqref{avgform22} and \eqref{eq:Qavgn},
and it is easy to check that
\eqref{eq:3rdQfree2}-\eqref{eq:ImQfree2} match with \eqref{eq:3rdQfree}-\eqref{eq:ImQfree},
since $\Psi_s \equiv 1$.

\smallskip

It remains to prove Theorem~\ref{th:3rdmom2}.
This will be reduced to Propositions~\ref{prop:3rdmomconv} and~\ref{prop:notripleQ}
below.
We exploit the multilinear expansion in \eqref{avgmulti} for the partition function,
which leads to the following representation for the centered third moment
(recall \eqref{eq:qNavg0}-\eqref{eq:qNavg1}):
\begin{equation}
	\label{eq:MNQ}
\begin{split}
	& \E\Big[\big( Z^{N, \beta_N}_{s,t}(\phi, \psi)
	- \bbE[Z^{N, \beta_N}_{s,t}(\phi, \psi) ] \big)^3\Big] \\
    & \quad =  \sumtwo{\bsA, \, \bsB, \, \bsC \subseteq \{s+1, \ldots, t-1\} \times \Z^2}
	{|\bsA \ge 1, \, |\bsB| \ge 1, \, |\bsC| \ge 1}
	\frac{1}{N^{3}} \,
	q^N_{s, a_1}(\phi, x_1) \, q^N_{s, b_1}(\phi,y_1) \,\cdot
	\, q^N_{s, c_1}(\phi, z_1) \,\cdot
	\\
	& \qquad \ \cdot\ \bbE  \Bigg[
	\xi_{A_1}  \prod_{i=2}^{|\bsA|}
	\xi_{A_i} \, q(A_{i-1}, A_i)  \cdot \xi_{B_1} \,
	\prod_{j=2}^{|\bsB|}  \xi_{B_j} \, q(B_{j-1}, B_j)
	\cdot  \xi_{C_1}
	\prod_{k=2}^{|\bsC|} \xi_{C_k} \, q(C_{k-1}, C_k)
	\Bigg]\,\cdot   \\
	& \qquad \ \cdot \, q^N_{a_{|\bsA|}, t}(x_{|\bsA|}, \psi)\,
	q^N_{b_{|\bsB|}, t}(y_{|\bsB|}, \psi) \,
	q^N_{c_{|\bsC|}, t}(z_{|\bsC|}, \psi)  \,,
\end{split}
\end{equation}
where we agree that
$\bsA=(A_1, \ldots, A_{|\bsA|})$ with $A_i=(a_i, x_i)\in \Z^3_{\rm even}$,
and $\bsB$, $\bsC$ are defined similarly, with $B_j=(b_j, y_j)$, $C_k=(c_k, z_k)$, and
we set for short
\begin{equation*}
	q(A_{i-1}, A_i):=q_{a_i-a_{i-1}}(x_i-x_{i-1}) \,.
\end{equation*}
(When $|\bsA| = 1$, the product $\prod_{i=2}^{|\bsA|} \ldots$
equals $1$, by definition, and similarly for $\bsB$ and $\bsC$.)

\smallskip

We now split the sum in \eqref{eq:MNQ} into two parts:
\begin{equation}\label{eq:trinotri}
	\E\Big[\big( Z^{N, \beta_N}_{s,t}(\phi, \psi)
	- \bbE[Z^{N, \beta_N}_{s,t}(\phi, \psi) ] \big)^3\Big]
	= M_{s, t}^{N, \rm NT}(\phi, \psi) + M_{s, t}^{N, \rm T}(\phi, \psi) \,,
\end{equation}
defined as follows:
\begin{itemize}
\item $M_{s, t}^{N, \rm NT}(\phi, \psi)$ is the sum in \eqref{eq:MNQ}
restricted to $\bsA, \bsB, \bsC$
such that $\bsA \cap \bsB\cap \bsC=\emptyset$, which we call
the case with {\em no triple intersections};

\item $M_{s, t}^{N, \rm T}(\phi, \psi)$ is the sum in \eqref{eq:MNQ}
restricted to $\bsA, \bsB, \bsC$
such that $\bsA \cap \bsB\cap \bsC \ne \emptyset$, which we call
the case with {\em triple intersections}.
\end{itemize}
These parts are analyzed in the following propositions,
which together imply Theorem~\ref{th:3rdmom2}.

\begin{proposition}[Convergence with no triple intersections]\label{prop:3rdmomconv}
Let the assumptions of Theorem~\ref{th:3rdmom2} hold.
Then
\be \label{eq:3rdm}
\lim_{N\to\infty} M^{N, \rm NT}_{0, Nt}(\phi, \psi) = M_t(\phi, \psi) =
3\,  \sum_{m=2}^\infty 2^{m-1} \, \pi^m \,  \cI^{(m)}_{t}(\phi, \psi) <\infty.
\ee
\end{proposition}

\begin{proposition}[Triple intersections are negligible]\label{prop:notripleQ}
Let the assumptions of Theorem~\ref{th:3rdmom2} hold.
Then
\be
	\lim_{N\to\infty} M^{N, \rm T}_{0, Nt}(\phi, \psi) = 0.
\ee
\end{proposition}

Proposition~\ref{prop:3rdmomconv}
is proved in the next section. The proof of Proposition~\ref{prop:notripleQ}
will be given later, see Section~\ref{sec:triple}.

\smallskip
\section{Convergence without triple intersections}
\label{sec:witout3}

In this section, we prove Proposition~\ref{prop:3rdmomconv}
and several related results.

\subsection{Proof of Proposition~\ref{prop:3rdmomconv}}

We first derive a
representation for $M^{N, \rm NT}_{s, t}(\phi, \psi)$,
which collects the terms in the expansion \eqref{eq:MNQ} with
$\bsA\cap \bsB\cap \bsC=\emptyset$.

\smallskip

Denote $\bsD :=\bsA\cup \bsB\cup\bsC\subset \{s+1, \ldots, t-1\}\times \Z^2$,
with $\bsD=(D_1, \ldots, D_{|\bsD|})$ and $D_i=(d_i, w_i)$. Since $\bbE[\xi_z]=0$,
the contributions to
$M^{N, \rm NT}_{s, t}(\phi, \psi)$ come only from $\bsA, \bsB, \bsC$ where the points in $\bsA\cup\bsB\cup\bsC$ pair up. In particular,
$$
k:= |\bsD| = \frac{1}{2} (|\bsA| + |\bsB| + |\bsC|) \geq 2,
$$
and each point $D_j$ belongs to exactly two of the three sets $\bsA, \bsB, \bsC$, and
hence we can associate a vector $\bsell = (\ell_1, \ldots, \ell_k)$ of
labels $\ell_j \in \{AB,\, BC,\, AC\}$. Note that there is a one to
one correspondence between $(\bsA, \bsB, \bsC)$ and $(\bsD, \bsell)$.
We also recall that $\xi_{n,z} = e^{\beta_N \omega(n,z) - \lambda(\beta_N)}-1$,
hence $\sigma_N^2=\bbE[\xi_z^2]$, see \eqref{eq:sigmaN}.
From \eqref{eq:MNQ} we can then write
\begin{equation} \label{eq:prelQ}
\begin{split}
	M^{N, \rm NT}_{s,t}(\phi,\psi) & =
	\frac{1}{N^{3}} \sum_{k=2}^{\infty} \, \sigma_N^{2k} \,
    \sumtwo{\bsD \subseteq \{s+1, \ldots, t-1\} \times \Z^2}{|\bsD|=k\geq 2}
    \sum_{\bsell\in \{AB,\, BC,\, AC\}^k} \\
	& \qquad \ q^N_{s, a_1}(\phi, x_1) \, q^N_{s, b_1}(\phi,y_1) \,
	q^N_{s, c_1}(\phi, z_1) \,\cdot\\
	& \qquad \cdot \prod_{i=2}^{|\bsA|}
	q(A_{i-1}, A_i) \, \prod_{j=2}^{|\bsB|}
	q(B_{j-1}, B_j) \, \prod_{m=2}^{|\bsC|} q(C_{m-1}, C_m) \,\cdot\\
	& \qquad \cdot q^N_{a_{|\bsA|}, t}(x_{|\bsA|}, \psi)\,
	q^N_{b_{|\bsB|}, t}(y_{|\bsB|}, \psi) \, q^N_{c_{|\bsC|}, t}(z_{|\bsC|}, \psi)  \,,
\end{split}
\end{equation}
where we agree that $\bsA, \bsB, \bsC$ are implicitly determined by $(\bsD, \bsell)$.
\smallskip

We now make a combinatorial observation. The
sequence $\bsell=(\ell_1, \ldots, \ell_{k})$ consists of consecutive \emph{stretches}
$(\ell_1, \ldots, \ell_i)$, $(\ell_{i+1}, \ldots, \ell_j)$, etc., such that the
labels are constant in each stretch and
change from one stretch to the next.
Any stretch, say $(\ell_p, \ldots, \ell_q)$,
has a first point $D_p = (a,x)$ and a last point $D_q = (b,y)$.
Let $m$ denote the number of stretches and
let $(a_i, x_i)$ and $(b_i, y_i)$,
with $a_i \le b_i$, be the first and last points of the $i$-th stretch.

We now rewrite \eqref{eq:prelQ} by summing over $m\in\N$, $(a_1, b_1, \ldots, a_m, b_m)$, and $(x_1, y_1, \ldots, x_m, y_m)$.
The sum over the labels of $\bsell$ leads to a combinatorial factor $3 \cdot 2^{m-1}$, because there are $3$ choices for the label of the first stretch and two choices for the label of the following stretches. Once we fix $(a_1, x_1)$ and $(b_1, y_1)$, summing over all possible configurations inside the first stretch then gives the factor
\begin{equation*}
	\sum_{r=1}^\infty \sigma_N^{2(r+1)} \!\!\!\!\!\!\!\!
	\sumtwo{a_1 = t_0 < t_1 < \ldots < t_r = b_1}{z_0 = x_1, \, z_1, z_2, \ldots, z_{r-1} \in \Z^2,
	\, z_r = y_1}
	\prod_{i=1}^r q_{t_{i-1}, t_i}(z_{i-1},z_i)^2 = \sigma_N^2 \,
	U_N(b_1 - a_1, \, y_1 - x_1) \,,
\end{equation*}
where we recall that $U_N$ is defined in
\eqref{eq:UNVar}-\eqref{def:UN}.
A similar factor arises from each stretch, which leads to the
following crucial identity (see Figure~\ref{figure}):
\begin{equation} \label{eq:prel2Q}
\begin{split}
	M^{N, \rm NT}_{s, t}(\phi,\psi) & = \sum_{m=2}^{\infty} 3 \cdot 2^{m-1}
	I^{(N, m)}_{s, t}(\phi, \psi), \qquad \text{where} \\
    I^{(N, m)}_{s, t}(\phi, \psi) & :=	\frac{\sigma_N^{2m}}{N^{3}}
	\sumtwo{s < a_1 \le b_1 < a_2 \le b_2 < \ldots < a_m \le b_m < t}
	{x_1, y_1, x_2, y_2, \ldots, x_m, y_m \in \Z^2} \!\!\!\!\!\!\!\!\!\!
	q^N_{s,a_1}(\phi,x_1)^2 \, q^N_{s,a_2}(\phi,x_2) \, \cdot
	\\
	& \qquad \cdot U_N(b_1-a_1, y_1 - x_1) \, q_{b_1, a_2}(y_1,x_2) \,
	U_N(b_2-a_2, y_2 - x_2)\,\cdot
	\\
	& \qquad \cdot \prod_{i=3}^{m}
	\Big\{ q_{b_{i-2}, a_i}(y_{i-2},x_i) \, q_{b_{i-1}, a_i}(y_{i-1},x_i) \,
	U_N(b_i-a_i, y_i - x_i) \Big\} \,\cdot\\
	& \qquad
	\cdot q^N_{b_{m-1}, t}(y_{m-1}, \psi) \,
	q^N_{b_{m}, t}(y_{m}, \psi)^2 \,,
\end{split}
\end{equation}
with the convention that $\prod_{i=3}^m \{\ldots\} = 1$ for $m=2$. Note that the sum starts with $m=2$ because in \eqref{eq:prelQ}, we have $|\bsA|, |\bsB|, |\bsC|\geq 1$.

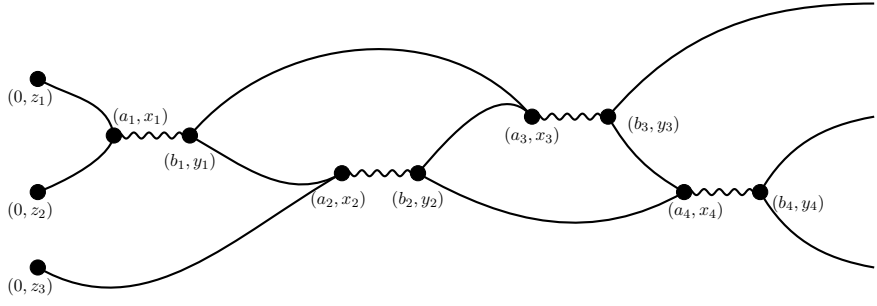
\begin{figure}
\begin{tikzpicture}[scale=0.5]
\draw  [fill] (0, -0.5)  circle [radius=0.2]; \node at (-0.1,-1.2) {\scalebox{0.6}{$(a_2,x_2)$}};
\draw [-,thick, decorate, decoration={snake,amplitude=.4mm,segment length=2mm}] (0,-0.5) -- (2,-0.5);
\draw  [fill] (2, -0.5)  circle [radius=0.2]; \node at (2,-1.2) {\scalebox{0.6}{$(b_2,y_2)$}};
\draw[thick] (2,-0.5) to [out=50,in=130] (5,1);
\draw  [fill] (5, 1)  circle [radius=0.2]; \node at (5,0.5) {\scalebox{0.6}{$(a_3,x_3)$}};
 \draw [-,thick, decorate, decoration={snake,amplitude=.4mm,segment length=2mm}] (5,1) -- (7,1);
\draw  [fill] (7, 1)  circle [radius=0.2]; \node at (8.2,0.8) {\scalebox{0.6}{$(b_3,y_3)$}};
\draw[thick] (-4,0.5) to [out=50,in=130] (5,1);
\draw  [fill] (-4,0.5)  circle [radius=0.2]; \node at (-4,-0.2) {\scalebox{0.6}{$(b_1,y_1)$}};
 \draw [-,thick, decorate, decoration={snake,amplitude=.4mm,segment length=2mm}] (-6,0.5) -- (-4,0.5);
 \draw[thick] (-8,-1) to [out=200,in=260] (-6,0.5);
  \draw  [fill] (-8,-1)  circle [radius=0.2];  \node at (-8.2,1.5) {\scalebox{0.6}{$(0,z_1)$}};
   \draw[thick] (-8,2) to [out=-30,in=100] (-6,0.5);
   \draw  [fill] (-8,2)  circle [radius=0.2];  \node at (-8.2,-1.5) {\scalebox{0.6}{$(0,z_2)$}};
 \draw  [fill] (-6,0.5)  circle [radius=0.2]; \node at (-5.3,1) {\scalebox{0.6}{$(a_1,x_1)$}};
 \draw[thick] (-4,0.5) to [out=-30,in=210] (0,-0.5);
  \draw[thick] (-8,-3) to [out=-30,in=210] (0,-0.5); \node at (-8.2,-3.5) {\scalebox{0.6}{$(0,z_3)$}};
   \draw  [fill] (-8,-3)  circle [radius=0.2];
  \draw[thick] (2,-0.5) to [out=-30,in=210] (9,-1);
   \draw  [fill] (9,-1)  circle [radius=0.2];  \node at (9.3,-1.5) {\scalebox{0.6}{$(a_4,x_4)$}};
    \draw[thick] (7,1) to [out=-60,in=150] (9,-1);
    \draw [-,thick, decorate, decoration={snake,amplitude=.4mm,segment length=2mm}] (9,-1) -- (11,-1);
     \draw  [fill] (11,-1)  circle [radius=0.2];  \node at (12,-1.3) {\scalebox{0.6}{$(b_4,y_4)$}};
      \draw[thick] (11,-1) to [out=60,in=190] (14,1);
       \draw[thick] (11,-1) to [out=-60,in=-190] (14,-3);
       \draw[thick] (7,1) to [out=50,in=180] (14,4);
\end{tikzpicture}
\caption{Diagramatic representation of the expansion \eqref{eq:prel2Q} of the third moment. Curly lines
between nodes $(a_i,x_i)$ and $(b_i,y_i)$ have weight $U_N(b_i-x_i,y_i-x_i)$, coming for pairwise matchings between a single pair of copies $AB, BC$ or $CA$, while
solid, curved lines between nodes $(a_i,x_i)$ and $(b_{i-1},y_{i-1})$ or between $(a_i,x_i)$ and $(b_{i-2},y_{i-2})$
indicate a weight $q_{b_{i-1},a_i}(y_{i-1},x_i)$ and $q_{b_{i-2},a_i}(y_{i-2},x_i)$, respectively.
\label{figure}}
\end{figure}

\smallskip

If we compare \eqref{eq:prel2Q} with \eqref{eq:3rdm} and \eqref{eq:ImQfree2},
we see that Proposition~\ref{prop:3rdmomconv} follows from the following result
and dominated convergence.\qed

\begin{lemma}\label{lem:INmbd}
For $m\geq 2$, let
$I^{(N, m)}_{Nt}(\phi, \psi):=I^{(N, m)}_{0, Nt}(\phi, \psi)$ be defined
as in \eqref{eq:prel2Q}, and let
$\cI^{(m)}_t(\phi, \psi)$ be defined as in \eqref{eq:ImQfree2}. Then
\be\label{eq:INmconv}
	\lim_{N\to\infty} I^{(N, m)}_{Nt}(\phi, \psi) = \pi^m\,  \cI^{(m)}_t(\phi, \psi)
	\qquad \forall\, m\geq 2 \,.
\ee
Furthermore, for any $C>0$ we have
\be\label{eq:INmbound}
	|I^{(N, m)}_{Nt}(\phi, \psi)| \leq e^{-Cm} \qquad
	\text{for all $m, N$ sufficiently large} \,.
\ee
\end{lemma}

\smallskip
The proof of Lemma~\ref{lem:INmbd} is given later,
see Subsection~\ref{sec:INmbd}.
We first prove the next result on $\cI^{(m)}_{t}(\phi, \psi)$,
which will reveal a structure that
will be used in the proof of Lemma~\ref{lem:INmbd}.

\begin{lemma}\label{lem:Imconv}
For $\phi\in C_c(\R^2)$, $\psi \in C_b(\R^2)$, and $\cI^{(m)}_t(\phi, \psi)$
defined as in \eqref{eq:ImQfree2}, we have:
\be\label{eq:Imconv}
\forall\, a \in (0,\infty) \,,
 \qquad \sum_{m=2}^\infty a^{m} \,  |\cI^{(m)}_{t}(\phi, \psi)| <\infty.
\ee
\end{lemma}

\subsection{Proof of Lemma  \ref{lem:Imconv}}

In light of Remark \ref{rm:ZNscaling}, we may
assume $t=1$. Recall that
\begin{equation} \label{eq:Imphipsi}
\begin{split}
	\cI^{(m)}(\phi, \psi) & := \cI^{(m)}_1(\phi, \psi)  :=
	\idotsint\limits_{\substack{0 < a_1 < b_1 < a_2 < b_2 < \ldots < a_m < b_m < 1 \\
	x_1, y_1, x_2, y_2, \ldots, x_m, y_m \in \R^2}}
	\Phi_{a_1}^2(x_1) \, \Phi_{a_2}(x_2) \, \cdot\\
	& \qquad\ \cdot\
	G_{\theta}(b_1 - a_1, y_1 - x_1) \, g_{\frac{a_2 - b_1}{2}}(x_2 - y_1) \,
	G_{\theta}(b_2 - a_2, y_2 - x_2)  \,\cdot \\
	& \qquad\ \cdot \ \prod_{i=3}^{m}
	g_{\frac{a_i - b_{i-2}}{2}}(x_i - y_{i-2}) \, g_{\frac{a_i - b_{i-1}}{2}}(x_i - y_{i-1})
	\, G_{\theta}(b_i - a_i, y_i - x_i)  \,\cdot \\
    & \qquad \ \cdot\ \Psi_{1-b_{m-1}}(y_{m-1}) \Psi^2_{1-b_m}(y_m) \ \dd \vec a\,
    \dd \vec b\, \dd \vec x\, \dd \vec y,
\end{split}	
\end{equation}
where $G_\theta(t, x):=G_\theta(t) g_{t/4}(x)$, with $g_{t/4}(x)$ being the heat
kernel, see \eqref{eq:gu},
and $G_\theta$ defined in \eqref{eq:Gtheta}.
We also recall that
$\Phi_a(x):=(\phi * g_{a/2})(x), \Psi_{1-b}(y)=(\Psi* g_{(1-b)/2} )(y)$.

\smallskip

Note that we obtain an upper bound if we replace $\phi$ by $|\phi|$,
so we may assume that $\phi \ge 0$.
Similarly, we may replace $\psi$ by the constant $|\psi|_\infty$,
and we take $|\psi|_\infty \le 1$ for simplicity.
We thus bound $\cI^{(m)}(\phi, \psi) \leq \cI^{(m)}(\phi, 1)$,
with $\phi \ge 0$, and we focus on $\cI^{(m)}(\phi, 1)$.

\smallskip

We first show that, by integrating out the space variables, we can bound
\begin{equation}\label{eq:Imphipsibd2}
\begin{split}
	\cI^{(m)}(\phi, 1)  & \leq C_\phi \, J^{(m)}, \qquad \text{where} \quad
	C_\phi := |\phi|_\infty^2
	\, \int_{\R^2} \phi(z) \, \dd z \,, \qquad \text{and} \\
    J^{(m)} & := \!\!\!\!\!\!\!\!\!\!\!\!\!\!\!
	\idotsint\limits_{0 < a_1 < b_1 < \ldots < a_m < b_m < 1} \!\!\!\!\!\!\!\!\!\!\!\!\!\!\!
	G_{\theta}(b_1 - a_1) G_{\theta}(b_2 - a_2) \prod_{i=3}^{m}
	\frac{G_{\theta}(b_i - a_i)}{\sqrt{(a_i-b_{i-1})(a_i-b_{i-2})}}\ \dd \vec a\,
	\dd \vec b \,.
\end{split}	
\end{equation}
Note that in \eqref{eq:Imphipsi} we have $\Psi \equiv 1$
(by $\psi \equiv 1$)
and $y_m$ appears only in $G_{\theta}(b_m-a_m, y_m-x_m)$.
Then we can integrate out $y_m\in\R^2$ to obtain
$$
	\int_{\R^2} G_{\theta}(b_m-a_m, y_m-x_m)
	\, \dd y_m = G_{\theta}(b_m-a_m).
$$
We are then left with two factors containing $x_m$, and the
corresponding integral is
\begin{align*}
	& \int_{\R^2} g_{\frac{a_m-b_{m-1}}{2}}(x_m-y_{m-1})
	\,g_{\frac{a_m-b_{m-2}}{2}}(x_m-y_{m-2}) \dd x_m
	=  g_{\frac{2a_m-b_{m-1}-b_{m-2}}{2}}(y_{m-1}-y_{m-2}) \\
	& \quad
	\leq g_{\frac{2a_m-b_{m-1}-b_{m-2}}{2}}(0) = \frac{1}{\pi(2a_m-b_{m-1}-b_{m-2})}
	\leq \frac{1}{2\pi}
	\, \frac{1}{\sqrt{(a_m-b_{m-1})(a_m-b_{m-2})}} \,,
\end{align*}
having used $\alpha \beta \le \frac{1}{2}(\alpha^2+ \beta^2)$ in
the last inequality.

We now iterate. Integrating out each
$y_i$, for $i \ge 2$, replaces $G_{\theta}(b_i-a_i, y_i-x_i)$
by $G_{\theta}(b_i-a_i)$, while integrating out each $x_i$, for $i\geq 3$,
replaces $g_{\frac{a_i - b_{i-2}}{2}}(x_i - y_{i-2}) \,
g_{\frac{a_i - b_{i-1}}{2}}(x_i - y_{i-1})$ by
$(2\pi \sqrt{(a_i-b_{i-1})(a_i-b_{i-2})})^{-1}$.
This leads to
\begin{equation}\label{eq:Imphipsibd1}
\begin{split}
	\cI^{(m)}(\phi, 1)  & \leq \frac{1}{(2\pi)^{m-2}} \,
	\idotsint\limits_{\substack{0 < a_1 < b_1 < a_2 < b_2 < \ldots < a_m < b_m < 1 \\
	x_1, y_1, x_2 \in \R^2}}
	\dd \vec a\,\, \dd \vec b\,\, \dd x_1 \, \dd y_1\, \dd x_2 \\
	& \qquad\qquad\ \
	\Phi_{a_1}^2(x_1) \, \Phi_{a_2}(x_2) \,
	g_{\frac{b_1-a_1}{4}}(y_1 - x_1)
	\, g_{\frac{a_2 - b_1}{2}}(x_2 - y_1) \, \cdot
	\\
	& \qquad\qquad\ \cdot\ G_{\theta}(b_1 - a_1) \, G_{\theta}(b_2 - a_2) \, \prod_{i=3}^{m}
	\frac{G_{\theta}(b_i - a_i)}{\sqrt{(a_i-b_{i-1})(a_i-b_{i-2})}}\
	\,.
\end{split}	
\end{equation}
We finally bound $\Phi_{a_2}(x_2) \le |\phi|_\infty$,
see \eqref{eq:Phiax},
then perform the integrals over $x_2$ and $y_1$, which both give $1$,
and note that $\int_{\R^2} \Phi_{a_1}(x_1)^2 \, \dd x_1
\le |\phi|_\infty \int_{\R^2}
\phi(z) \, \dd z$, which yields \eqref{eq:Imphipsibd2}.

\smallskip

We can now bound the quantity in Lemma~\ref{lem:Imconv}
using \eqref{eq:Imphipsibd2}, to get
\begin{equation}\label{ImJm}
	\sum_{m=2}^\infty a^m |\cI^{(m)}_t(\phi, \psi)| \leq
	(2\pi)^2\, C_\phi \sum_{m=2}^\infty \Big(\frac{a}{2\pi}\Big)^m J^{(m)} \,.
\end{equation}
It remains to show that $J^{(m)}$ decay super-exponentially fast.
For any $\lambda>0$, we have
$$
J^{(m)} \leq e^{\lambda} \!\!\!\!\!\!\!\!\!\!\!\!\!\!\!
	\idotsint\limits_{0 < a_1 < b_1 < \ldots < a_m < b_m < 1} \!\!\!\!\!\!\!\!\!\!\!\!\!\!\!
	e^{-\lambda b_1} G_{\theta}(b_1 - a_1) e^{-\lambda (b_2-b_1)} G_{\theta}(b_2 - a_2) \prod_{i=3}^{m}
	\frac{e^{-\lambda(b_i-b_{i-1})}G_{\theta}(b_i - a_i)}{\sqrt{(a_i-b_{i-1})(a_i-b_{i-2})}}\ \dd \vec a\, \dd \vec b.
$$
Denote $u_i:=a_i-b_{i-1}$ and $v_i:=b_i-a_i$ for $1\leq i\leq m$, where $b_0:=0$.
Then observe that $a_i-b_{i-2}=u_{i-1}+v_{i-1}+u_i\geq u_{i-1}+u_i$.
Since $b_i - b_{i-1} \ge v_i$, we can bound $J^{(m)}$ by
\begin{equation}\label{eq:Jmbd}
\begin{aligned}
J^{(m)} & \leq e^{\lambda}
	\idotsint\limits_{u_i,v_i\in (0,1) \atop \sum_{i=1}^m (u_i+v_i)<1}
    \Bigg\{ \prod_{i=1}^m e^{-\lambda v_i} G_{\theta}(v_i) \Bigg\}
	\,   \Bigg\{
	\prod_{i=3}^m \frac{1}{\sqrt{(u_{i-1}+u_i) u_i}}
	\Bigg\} \, \dd \vec u\, \dd \vec v \\
	& \leq e^{\lambda}
	\, \bigg( \int_0^1 e^{-\lambda v} \, G_{\theta}(v) \, \dd v\bigg)^m
	\
	\idotsint\limits_{u_2, \ldots, u_m\in (0,1)} \,
	\prod_{i=3}^m
	\frac{1}{\sqrt{u_i(u_i+u_{i-1})}}\ \dd \vec u \\
	& \le e^\lambda \,
	\bigg(\int_0^1 e^{-\lambda v}  \, \hat G_{\theta}(v) \, \dd v\bigg)^m  \,
	\int_0^1 \phi^{(m-2)}(u_2) \, \dd u_2,
\end{aligned}
\end{equation}
where in the last inequality we have bounded
$G_{\theta}(\cdot) \le \hat G_{\theta}(\cdot)$, see \eqref{eq:Gthetaunif},
and we define
\begin{equation} \label{eq:phik}
	\phi^{(0)}(u) := 1 \,, \qquad \text{and} \qquad
	\phi^{(k)}(u) := \int_0^1 \frac{ 1 }{\sqrt{s(s + u)}}
	\, \phi^{(k-1)}(s) \,  \dd s \,,
	\qquad \forall k \in \N \,.
\end{equation}
We will show the following results.

\begin{lemma}\label{lem:GthIntegral}
There is a constant $\sfc_\theta < \infty$ such that
for every $\lambda \ge 1$
\begin{equation} \label{eq:Clambda}
	\int_0^1 e^{-\lambda v} \, \hat G_{\theta}(v)
	\, \dd v \leq C_\lambda:= \frac{\sfc_\theta}{2 + \log \lambda} \,.
\end{equation}
\end{lemma}

\begin{lemma}\label{lem:phikbd}
For all $k\in\N$, the function $\phi^{(k)}(\cdot)$ is decreasing on $(0,1)$
and satisfies
\begin{equation} \label{eq:estphiklast}
	\phi^{(k)}(v) \le 32^k  \sum_{i=0}^k \frac{1}{i!}
	\bigg( \frac{1}{2} \log \frac{e^2}{v}\bigg)^{i} \le
	\displaystyle 32^k  \frac{e}{\sqrt{v}} \,,
	\qquad \forall v \in (0,1) \,.
\end{equation}
\end{lemma}
With Lemmas \ref{lem:GthIntegral} and \ref{lem:phikbd}, it follows
from \eqref{eq:Jmbd} that
\begin{equation}\label{eq:Jmexpbd}
	J^{(m)}
	\le e^{\lambda} \, C_\lambda^m
	\, 32^{m-2} \int_0^1 \frac{e}{\sqrt{v}} \, \dd v
	\le e^{\lambda} \, C_\lambda^m
	\, 32^{m-2} \, 2 e
	\le e^{\lambda} \, (32 \, C_\lambda) ^m \,.
\end{equation}
If we choose $\lambda=m$, then by \eqref{ImJm} and the definition \eqref{eq:Clambda}
of $C_\lambda$ we get
\begin{equation} \label{eq:Jmex2}
	\sum_{m=2}^\infty a^m \, |\cI^{(m)}_t(\phi, \psi)|
	\leq C_\phi \sum_{m=2}^\infty a^m J^{(m)} \leq
	C_\phi \sum_{m=2}^\infty
	\bigg( \frac{32 \, a \, e \, \sfc_\theta}{2+\log m}
	\bigg)^m < \infty \,,
\end{equation}
which concludes the proof of Lemma \ref{lem:Imconv}.
\qed

\medskip

It remains to prove Lemmas \ref{lem:GthIntegral} and \ref{lem:phikbd}.

\medskip

\begin{proof}[Proof of Lemma \ref{lem:GthIntegral}]
Recall that $\hat G_{\theta}(\cdot)$ is defined in \eqref{eq:Gthetaunif}
and it is decreasing. Then
\begin{equation*}
	\int_{\frac{1}{\lambda}}^1 e^{-\lambda v} \, \hat G_{\theta}(v) \, \dd v \le
	\hat G_{\theta}(\tfrac{1}{\lambda}) \int_{\frac{1}{\lambda}}^\infty e^{-\lambda v} \, \dd v
	= \hat G_{\theta}(\tfrac{1}{\lambda}) \, \frac{e^{-1}}{\lambda}
	\le e^{-1} \, \int_0^{\frac{1}{\lambda}} \hat G_{\theta}(v) \, \dd v \,,
\end{equation*}
hence
\begin{equation} \label{eq:Ghatint}
\begin{split}
	\int_0^1 e^{-\lambda v} \, \hat G_{\theta}(v) \, \dd v
	\le (1 + e^{-1}) \int_0^{\frac{1}{\lambda}} \hat G_{\theta}(v)
	\, \dd v
	= \frac{(1+e^{-1}) \, c_{\theta}}{2 + \log \lambda}
	\le \frac{2 \, c_{\theta}}{2 + \log \lambda} \,.
\end{split}
\end{equation}
We have proved that \eqref{eq:Clambda} holds,
provided we chose $\sfc_\theta := 2 \, c_{\theta}$.
\end{proof}

\medskip

\begin{proof}[Proof of Lemma~\ref{lem:phikbd}]
The second inequality in \eqref{eq:estphiklast} follows from
$\sum_{i=0}^k \frac{x^i}{i!}\leq e^x$.

\smallskip

Let us prove the first inequality in \eqref{eq:estphiklast}.
Recall the definition \eqref{eq:phik} of $\phi^{(k)}$. Then
\begin{equation} \label{eq:estphi1}
\begin{split}
	& \phi^{(1)}(v)
	= \int_0^{1} \frac{\dd s}{\sqrt{s(s + v)}} =
	\int_0^{\frac{1}{v}} \frac{\dd z}{\sqrt{z(z + 1)}}
	\le \int_0^{1} \frac{\dd z}{\sqrt{z}}
	+ \int_{1}^{\frac{1}{v}} \frac{\dd z}{z}
	= 2 + \log\frac{1}{v}\,.
\end{split}
\end{equation}
To iterate this argument and bound $\phi^{(k)}$, we claim that
\begin{equation} \label{eq:estphik}
	\phi^{(k)}(v) \le
	\sum_{i=0}^k \frac{c_{k,i}}{2^i \, i!} \,
	\bigg(\log \frac{e^2}{v}\bigg)^{i}  \qquad
	\forall v \in (0,1) \,,
\end{equation}
for suitable choices of the coefficients $c_{k,i}$. For $k=1$, we see from \eqref{eq:estphi1} that
\begin{equation} \label{eq:c1}
	c_{1,0} = 0 \,, \qquad c_{1,1} = 2 \,.
\end{equation}
Inductively, we assume that \eqref{eq:estphik} holds for $k-1$ and we will deduce it for $k$.
Note that plugging \eqref{eq:estphik} for $k-1$ into \eqref{eq:phik} gives
\begin{equation}\label{eq:phikrec1}
	\phi^{(k)}(v) \le
	\sum_{j=0}^{k-1}
	\frac{c_{k-1,j}}{2^j \, j!} \int_0^{1}
	\frac{\big(\log \frac{e^2}{s}\big)^{j}}{\sqrt{s(s + v)}} \, \dd s \,.
\end{equation}
To identify $c_{k, i}$ for $0\leq i\leq k$, we need the following Lemma, proved later.

\begin{lemma}\label{th:intlog}
For all $k \in \N_0$, we have
\begin{equation} \label{eq:intlog}
	\int_0^{1}
	\frac{\big(\log \frac{e^2}{s}\big)^k}{\sqrt{s(s + v)}} \, \dd s
	\le 2^{k+1}  k! \sum_{i=0}^{k+1}
	\frac{\big(\log \tfrac{e^2}{v}\big)^{i}}{2^i \, i!}
	 \qquad \forall v \in (0,1)\,.
\end{equation}
\end{lemma}

If we plug \eqref{eq:intlog} with $k=j$ into \eqref{eq:phikrec1}
we get that, for all $v \in (0,1)$,
\begin{equation*}
\begin{split}
	\phi^{(k)}(v)
	& \le
	\sum_{j=0}^{k-1}
	\frac{c_{k-1,j}}{2^j \, j!} \, \Bigg\{ 2^{j+1} \, j! \sum_{i=0}^{j+1}
	\frac{\big(\log \tfrac{e^2}{v}\big)^{i}}{2^i \, i!} \Bigg\}
	 =
	\sum_{i=0}^{k} \frac{1}{2^i \, i!}
	\Bigg(2 \sum_{j=(i-1)^+}^{k-1} c_{k-1,j} \Bigg)
	\big(\log \tfrac{e^2}{v}\big)^{i} \,.
\end{split}
\end{equation*}
This shows that \eqref{eq:estphik} indeed holds, with
\begin{equation}\label{eq:rec}
\begin{split}
	c_{k,i} & = 2 \sum_{j=(i-1)^+}^{k-1} c_{k-1,j} \,.
\end{split}
\end{equation}
We have the following combinatorial bound on the coefficients $c_{k, i}$,
which we prove later by comparing with the number of paths for a suitable random walk.

\begin{lemma}\label{lem:Sk}
For every $k\in\N$ and $i \in \{0,\ldots,k\}$ we have
$c_{k, i} \le 32^k$.
\end{lemma}

\noindent
Plugging this bound into \eqref{eq:estphik} we obtain,
for all $k\in\N$ and $v \in (0,1)$,
\begin{equation} \label{eq:estphiklast-}
	\phi^{(k)}(v) \le 32^k \, \sum_{i=0}^k \frac{1}{i!} \,
	\bigg(\frac{1}{2} \, \log \frac{e^2}{v}\bigg)^{i} \,,
\end{equation}
which is the first inequality in \eqref{eq:estphiklast}.
This concludes the proof of Lemma \ref{lem:phikbd}.
\end{proof}

\medskip

It remains to prove Lemmas~\ref{th:intlog} and~\ref{lem:Sk}.

\medskip

\begin{proof}[Proof of Lemma~\ref{th:intlog}]
By a change of variable $s=vz$,
\begin{equation}\label{eq:AB}
	\int_0^{1}
	\frac{\big(\log \frac{e^2}{s}\big)^k}{\sqrt{s(s + v)}} \, \dd s =
	\int_0^{\frac{1}{v}}
	\frac{\big(\log \frac{e^2}{v z}\big)^k}{\sqrt{z (z + 1)}} \, \dd z \le
	\underbrace{\int_0^{1}
	\frac{\big(\log \frac{e^2}{v z}\big)^k}{\sqrt{z}} \, \dd z}_{A} \, + \,
	\underbrace{\int_1^{\frac{1}{v}}
	\frac{\big(\log \frac{e^2}{v z}\big)^k}{z} \, \dd z}_{B}  \,.
\end{equation}
Let us look at $B$: the change of variable $z = v^{\alpha-1}$,
with $\alpha \in (0,1)$, gives
\begin{equation*}
	B = \int_0^1 \frac{\big(\log \frac{e^2}{v^{\alpha}}\big)^k}
	{v^{\alpha - 1}} \,
	v^{\alpha-1} \, \log\tfrac{1}{v} \, \dd \alpha
	= \log\tfrac{1}{v} \int_0^1 \big(2 +
	\alpha \log\tfrac{1}{v}\big)^k \, \dd \alpha \leq \tfrac{1}{k+1} \big(
	\log\tfrac{e^2}{v}\big)^{k+1}.
\end{equation*}
We now look at $A$: the change of variable $z = x \, e^2 /  v$, with
$x \in (0, \frac{v}{e^2})$, followed by $x = e^{-2y}$, with $y \in (\frac{1}{2}
\log \frac{e^2}{v}, \infty)$, yields
\begin{equation*}
\begin{split}
	A &= \frac{e}{\sqrt{v}} \int_0^{\frac{v}{e^2}}
	\frac{\big(\log \frac{1}{x}\big)^k}{\sqrt{x}} \, \dd x
	= \frac{e}{\sqrt{v}} \int_{\frac{1}{2}\log \frac{e^2}{v}}^\infty
	\frac{(2y)^k}{e^{-y}} \, 2 \, e^{-2y} \, \dd y
	= \frac{e \, 2^{k+1}}{\sqrt{v}} \int_{\frac{1}{2}\log \frac{e^2}{v}}^\infty
	y^k \, e^{-y} \, \dd y \,.
\end{split}
\end{equation*}
Let $(\bsN_t)_{t\geq 0}$ be a Poisson process with intensity one,
and let $(X_i)_{i\geq 1}$ denote its jump sizes,
which are i.i.d.\ exponential variables with parameter one.
For all $t \ge 0$ we can write
\begin{equation} \label{eq:gammaint}
\begin{split}
	\int_t^\infty y^k \, e^{-y} \, \dd y
	&=\Gamma(k+1)\, \P\bigg( \sum_{i=1}^k X_i\geq t\bigg)
	=\Gamma(k+1)\ \P\big( {\rm \bsN}_t\leq k\big)
	= e^{-t} \sum_{i=0}^k \frac{k!}{i!} \, t^{i} \,.
\end{split}
\end{equation}
Choosing
$t = \frac{1}{2}\log \frac{e^2}{ v}$, it follows that
\begin{equation}\label{eq:Aid}
\begin{split}
	A = \frac{e \, 2^{k+1}}{\sqrt{v}}
	e^{- \frac{1}{2}\log \frac{e^2}{v}}
	\sum_{i=0}^k \frac{k!}{i!} \, \big(\tfrac{1}{2}\log \tfrac{e^2}{v}\big)^{i}
	= \sum_{i=0}^k \frac{k!}{i!} \, 2^{k+1-i} \,
	\big(\log \tfrac{e^2}{v}\big)^{i} \,.
\end{split}
\end{equation}
We have thus shown that
\begin{equation*}
	A + B = \sum_{i=0}^{k+1} \frac{k!}{i!} \, 2^{k+1-i} \,
	\big(\log \tfrac{e^2}{v}\big)^{i}\,,
\end{equation*}
which coincides with the RHS of \eqref{eq:intlog}.
\end{proof}

\medskip

\begin{proof}[Proof of Lemma~\ref{lem:Sk}]
We iterate the recursion relation \eqref{eq:rec}, to get
\begin{equation} \label{eq:ckiter}
\begin{split}
	c_{k,i}
	& = 2 \sum_{j_{k-1}=(i-1)^+}^{k-1} c_{k-1,j_{k-1}}
	= 4 \sum_{j_{k-1}=(i-1)^+}^{k-1} \
	\sum_{j_{k-2}=(j_{k-1}-1)^+}^{k-2} c_{k-2,j_{k-2}} = \cdots \\
	& = 2^{k-1} \sum_{j_{k-1}=(i-1)^+}^{k-1} \
	\sum_{j_{k-2}=(j_{k-1}-1)^+}^{k-2} \ldots
	\sum_{j_{2} = (j_{3}-1)^+}^{2}
	\ \sum_{j_{1} = (j_{2}-1)^+}^{1} c_{1, j_{1}} \,.
\end{split}
\end{equation}
Since $c_{1, 1} = 2$ and $c_{1, 0} = 0$, see \eqref{eq:c1}, we can restrict to $j_1 = 1$. Also observe that
\begin{equation*}
	j_i \ge (j_{i+1} - 1)^+ \qquad \text{if and only if}
	\qquad j_i \ge 0 \ \ \text{and} \ \ j_{i+1} \le j_i + 1 \,,
\end{equation*}
hence we can reverse the order of the sums in \eqref{eq:ckiter} and write
\begin{equation} \label{eq:estck}
	c_{k,i} = 2^{k} \, \big| \cS_{k}(i) \big| \,,
\end{equation}
where $|\cS_{k}(i)|$ denotes the cardinality of the set
\begin{equation}\label{eq:set0}
	\cS_{k}(i) := \big\{(j_1, \ldots, j_k) \in \N_0^k: \
	j_1 = 1, \ j_k = i, \ j_{n+1} \le j_n + 1 \ \forall\,
	n=1, \ldots, k-1 \big\} \,.
\end{equation}
In words, $\cS_{k}(i)$ is the set of non-negative integer-valued paths $(j_1, \ldots, j_k)$
that start from $j_1 = 1$, arrive at $j_k = i$, and can make upward jumps
of size at most $1$, while the downward jumps can be of arbitrary size
(with the constraint that the path is non-negative).

To complete the proof, it remains to show that
\begin{equation*}
	| \cS_{k}(i) | \le 16^k  \,.
\end{equation*}
We define a correspondence which associates to any path $\bs{j} = (j_1, \ldots, j_k) \in \cS_{k}(i)$ a
\emph{nearest neighbor path} $\bs{\ell} = (\ell_1, \ldots, \ell_n)$,
with length $n = n(\bs{j}) \in \{k, \ldots, 2k\}$,
with increments in $\{-1,0, 0^*,+1\}$, where by $0^*$ we mean an increment of size $0$
with an extra label ``$*$'' (that will be useful to get an injective map).
The correspondence is simple: whenever the path $\bs{j}$ has a downward jump
(which can be of arbitrary size), we transform it into a sequence of downward jumps of size $1$, followed by a jump
of size $0^*$.

Note that if $m = m(\bs{j})$ denotes the number of downward jumps in the path $\bs{j}$, then the new path
$\bs{\ell} = (\ell_1, \ldots, \ell_n)$ has length
\begin{equation*}
	n = n(\bs{j}) = k + (\sigma_1 + \ldots + \sigma_m) \,,
\end{equation*}
where $\sigma_i$ is the size of the $i$-th downward jump of $\bs{j}$.
The total size of downward jumps is
$$
(\sigma_1 + \ldots + \sigma_m) = \Delta^-(\bs{j}) := \sum_{i=1}^{k-1} (j_{i+1}-j_i)^- \,.
$$
Defining
$\Delta^+(\bs{j}) := \sum_{i=1}^{k-1} (j_{i+1}-j_i)^+$, we have
\begin{equation*}
	\Delta^+(\bs{j}) - \Delta^-(\bs{j}) = j_k - j_1 = i-1 \,.
\end{equation*}
However $\Delta^+(\bs{j}) \le k-1$, because the upward jumps are of size at most $1$, hence
\begin{equation*}
	\Delta^-(\bs{j}) \le (k-1) - (i-1) \le k \,,
\end{equation*}
which shows that $n = n(\bs{j}) \le 2k$, as we claimed.

Note that the correspondence $\bs{j} \mapsto \bs{\ell}$ is injective:
the original path $\bs{j}$ can be reconstructed from $\bs{\ell}$,
thanks to the labeled increments $0^*$,
which distinguishes consecutive downward jumps from a single downward jump with
the same total length.
Since the path $\bs{\ell} = (\ell_1, \ldots, \ell_n)$ has $n-1$ increments,
each of which takes four possible values, we get the desired estimate:
\begin{equation*}
	\big| \cS_{k}(i) \big| \le \sum_{n=k}^{2k} 4^{n-1}
	\le \sum_{n=1}^{2k} 4^{n-1}
	= \frac{16^{k}-1}{3} \le 16^k \,.\qedhere
\end{equation*}
\end{proof}

\subsection{Proof of Lemma \ref{lem:INmbd}}
\label{sec:INmbd}
We follow the same strategy as in the proof of Lemma~\ref{lem:Imconv}.

We first prove the exponential bound \eqref{eq:INmbound}.
We recall that $I^{(N, m)}_{Nt}(\phi, \psi):=I^{(N, m)}_{0, Nt}(\phi, \psi)$,
see \eqref{eq:prel2Q}. We may take $t=1$, $\phi\geq 0$, and $\psi\equiv 1$,
so that the last terms in \eqref{eq:prel2Q} are
$q^N_{b_{m-1}, t}(y_{m-1}, \psi) \equiv 1$, $q^N_{b_{m}, t}(y_{m}, \psi) \equiv 1$.
We can thus rewrite \eqref{eq:prel2Q} as follows:
\begin{equation} \label{eq:prel2Qb}
\begin{split}	
    I^{(N, m)}_N(\phi, 1) & :=	\frac{\sigma_N^{2m}}{N^{3}}
	\sumtwo{0 < a_1 \le b_1 < a_2 \le b_2 < \ldots < a_m \le b_m < N}
	{x_1, y_1, x_2, y_2, \ldots, x_m, y_m \in \Z^2} \!\!\!\!\!\!\!\!\!\!
	q^N_{0,a_1}(\phi,x_1)^2 \, q^N_{0,a_2}(\phi,x_2) \,\cdot
	\\
	& \qquad \cdot U_N(b_1-a_1, y_1 - x_1) \, q_{b_1, a_2}(y_1,x_2) \,
	U_N(b_2-a_2, y_2 - x_2) \,\cdot
	\\
	& \qquad  \cdot \prod_{i=3}^{m}
	\Big\{ q_{b_{i-2}, a_i}(y_{i-2},x_i) \, q_{b_{i-1}, a_i}(y_{i-1},x_i) \,
	U_N(b_i-a_i, y_i - x_i) \Big\} \,.
\end{split}
\end{equation}

Similar to \eqref{eq:Imphipsibd2}, we first prove the following
bound:
\begin{align}
	& I^{(N,m)}_N(\phi, 1)  \leq C_\phi \, J^{(N,m)} \,, \qquad
	\text{where} \label{eq:INmphipsibd2}\\
	& J^{(N,m)} := c^m \, \frac{\sigma_N^{2m}}{N^2} \!\!\!\!
	\sum\limits_{\substack{0 < a_1 \le b_1 < \ldots \ \ \\
	\ \ldots 	< a_m \le b_m < N}} \!\!\!\!
	U_N(b_1 - a_1) \, U_N(b_2 - a_2) \prod_{i=3}^{m}
	\frac{U_N(b_i - a_i)}{\sqrt{(a_i-b_{i-1})(a_i-b_{i-2})}}
	\,, \nonumber
\end{align}	
for suitable constants $C_\phi, c < \infty$.
We first note that $y_m$ appears in \eqref{eq:prel2Qb}
only in the term $U_N(b_m-a_m, y_m-x_m)$ and hence we can sum it out as
\begin{equation}\label{eq:summU}
	\sum_{y_m\in\Z^2} U_N(b_m-a_m, y_m-x_m) =: U_N(b_m-a_m).
\end{equation}
We next sum over $x_m$: since $q_{s,t}(x,y) \le
\sup_{z} q_{t-s}(z) \le \frac{c}{t-s}$, see \eqref{eq:EZptp} and
\eqref{eq:qas}, we have
\begin{equation}\label{eq:qconvbd}
\begin{aligned}
	 \sum_{x_m\in\Z^2} q_{b_{m-1},a_m}(y_{m-1}, x_m) \,
	& q_{b_{m-2},a_m}(y_{m-2}, x_m) \le
	\sup_{x,y \in \Z^2} q_{b_{m-2},a_m}(y, x)
	\\
	&
	 \leq \frac{c}{(a_m-b_{m-2})}
	\leq \frac{c}{\sqrt{(a_m-b_{m-1})(a_m-b_{m - 2})}} \,.
\end{aligned}
\end{equation}
We can now iterate,
integrating out $y_i$ for $i \ge 2$ and $x_i$ for $i \ge 3$, to obtain
\begin{align}\label{eq:INmphibd1}
 I^{(N, m)}_N(\phi, 1) & \leq\frac{c^{m-2}\sigma_N^{2m}}{N^{3}}   \!\!\!\!\!
	\sum_{0<a_1\leq b_1<a_2\cdots <a_m\leq b_m < N} \,\, \sum_{x_1,x_2,y_1\in\Z^2}
	q^N_{0,a_1}(\phi,x_1)^2 \, q^N_{0,a_2}(\phi,x_2) \,
	\\
	& \  \cdot U_N(b_1-a_1, y_1-x_1) \, q_{b_1, a_2}(y_1,x_2) \,
	U_N(b_2-a_2)\, \prod_{i=3}^{m}
        \frac{	U_N(b_i-a_i)}{\sqrt{(a_i-b_{i-1} )(a_i-b_{i-2}) }}  \,.
        \notag
\end{align}
After bounding $q^N_{0,a_2}(\phi, x_2) \le |\phi|_\infty$,
see \eqref{eq:qNavg},
the sum over $x_2$ gives $1$, because $q_{b_1, a_2}(y_1,\cdot)$
is a probability kernel. Then
the sum over $y_1$ gives $U_N(b_1- a_1)$.
Finally, the sum over $x_1$ gives
\begin{equation} \label{eq:qualaa}
	\sum_{x_1 \in \Z^2} q^N_{0,a_1}(\phi,x_1)^2
	\le |\phi|_\infty \sum_{x_1 \in \Z^2} q^N_{0,a_1}(\phi,x_1)
	= |\phi|_\infty
	\sum_{z\in\Z^2} \phi\big(\tfrac{z}{\sqrt N}\big) \leq c_\phi \, N
\end{equation}
for a suitable $c_\phi < \infty$, because $\phi$ has compact support.
This completes the proof of \eqref{eq:INmphipsibd2}.

\smallskip

Next we bound $J^{(N,m)}$ in \eqref{eq:INmphipsibd2},
similarly to the continuum analogue \eqref{eq:Jmbd}.
Namely, we denote $u_i=a_i-b_{i-1}$ and $v_i:=b_i-a_i$ for $1\leq i\leq m$, with $b_0:=0$,
we insert the factor $e^{\lambda} \prod_{i=1}^m e^{-\lambda (\frac{v_i}{N})}>1$,
and then we use $a_i-b_{i-2} \ge u_{i-1}+u_i$
to obtain the bound
\begin{equation}\label{eq:INmphibd11}
 J^{(N,m)} \leq e^\lambda \, c^m \,
 \Bigg(\sigma_N^2\sum_{v=0}^N e^{-\lambda \frac{v}{N}} U_N(v)\Bigg)^m
\,\Bigg\{ \frac{1}{N^2}   \sum_{1\leq u_1, \ldots, u_m\leq N} \, \prod_{i=3}^m
 \frac{1}{\sqrt{u_i(u_{i-1}+u_i)}} \Bigg\} \,.
\end{equation}

Note that $\sigma_N^2 \le \frac{c_1}{\log N}$, see
\eqref{eq:sigmaN}, and $U_N(u) \le c_2 \, (\ind_{\{u=0\}} + \frac{\log N}{N}
\, \hat G_{\theta}(\frac{u}{N}) )$,
see \eqref{est:UNunif} and \eqref{eq:Gthetaunif}.
Since $\hat G_\theta(\cdot)$ is decreasing, we can bound the Riemann sum
by the integral and get
\begin{equation}
\label{eq:UNsumbd}
\begin{aligned}
\sigma_N^2\sum_{v=0}^N e^{-\lambda \frac{v}{N}} U_N(v) & \leq
\frac{c_1 c_2}{\log N} \Bigg(1  +
\frac{\log N}{N} \sum_{v=1}^N
e^{-\lambda \frac{v}{N}} \, \hat G_\theta\big(\tfrac{v}{N}\big)
 \Bigg) \\
& \le \frac{c_1 c_2}{\log N} \bigg( 1 + \log N \,
\int_0^1 e^{-\lambda v} \, \hat G_\theta(v) \, \dd v \bigg)
\le c_1 c_2 \bigg( \frac{1}{\log N} + C_\lambda \bigg) \,,
\end{aligned}
\end{equation}
where in the last inequality we have applied \eqref{eq:Clambda}.

The multiple sum over the $u_i$'s
in \eqref{eq:INmphibd11} is bounded
by the iterated integral in \eqref{eq:Jmbd}, by monotonicity
(note that if we replace $u_i$ by $N u_i$,
with $u_i \in \frac{1}{N}\Z \cap (0,1)$,
then we get the correct prefactor $1/N^{m}$,
thanks to the term $1/N^2$ in \eqref{eq:INmphibd11}). Then
\begin{equation} \label{eq:clearby}
\begin{split}
	J^{(N,m)}
	& \le e^\lambda \, c^{m} \,
	\big( \tfrac{1}{\log N} + C_\lambda \big)^m \,
	\int_0^1 \phi^{(m-2)}(u) \, \dd u \le e^\lambda \, (32 \, c)^{m} \,
	\big( \tfrac{1}{\log N} + C_\lambda \big)^m \, \,,
\end{split}
\end{equation}
because the integral is at most $32^m$, by \eqref{eq:estphiklast} (see
also \eqref{eq:Jmexpbd}).
Since $C_\lambda=\frac{\sfc_\theta}{2+\log \lambda}$,
see \eqref{eq:Clambda}, if we choose $\lambda$ and $N$ large enough,
then it is clear by \eqref{eq:clearby}
that $J^{(N,m)}$ decays faster than any exponential in $m$.
This proves \eqref{eq:INmbound}.

\medskip

We next prove \eqref{eq:INmconv},
for simplicity with $t=1$. This is easily guessed because
$I^{(N, m)}_{1}(\phi, \psi)$ (see \eqref{eq:prel2Q})
is close to a Riemann sum for $\pi^m\cI^{(m)}_{1}(\phi, \psi)$
(see \eqref{eq:ImQfree2}),
by the asymptotic relations
\begin{gather}
	\sigma_N^2 \sim \frac{\pi}{\log N} \,, \qquad
	q^N_{0,a}(\phi,x) \sim \Phi_{\frac{a}{N}}\big( \tfrac{x}{\sqrt{N}}\big) \,, \qquad
	q^N_{b, 1}(y, \psi) \sim \Psi_{1-\frac{b}{N}}\big( \tfrac{y}{\sqrt{N}}\big) \,,
	\label{eq:asqq} \\
	q_{b,a}(y,x) \sim \frac{1}{N} \, g_{\frac{a-b}{N}}\big(\tfrac{x-y}{\sqrt{N}}\big) \,,
	\qquad U_N(b-a,y-x) \sim \frac{\log N}{N^2} \, G_{\theta}\big(\tfrac{b-a}{N},
	\tfrac{y-x}{\sqrt{N}}\big) \,,
	\label{eq:asqU}
\end{gather}
see \eqref{eq:olap}, \eqref{eq:sigmaN}, \eqref{eq:convPhi}
and \eqref{eq:qas}, \eqref{eq:Unas}.\footnote{For
simplicity, in relations \eqref{eq:asqU}
we have omitted the ``periodicity correction''
$2 \, \ind_{\{(n,x) \in \Z^3_{\mathrm{even}}\}}$, see \eqref{eq:qas} and \eqref{eq:Unas},
because this disappears upon summation.}
We stress that plugging \eqref{eq:asqq}-\eqref{eq:asqU} into \eqref{eq:prel2Q}
we obtain the correct prefactor $1/N^{2m}$,
thanks to the extra term $1/N^3$ in \eqref{eq:prel2Q}.

To justify the replacements \eqref{eq:asqq}-\eqref{eq:asqU}, we proceed by approximations.
Henceforth $m\geq 2$ is fixed.
We define $\cI^{(m), (\eps)}_1(\phi, \psi)$ by restricting the integral
in \eqref{eq:ImQfree2} to the set
\begin{equation} \label{eq:restrii0}
	\big\{ a_{i}-b_{i-1} \ge \epsilon \ \ \forall 1 \le i \le m+1 \,, \quad \
	b_i-a_i\geq \epsilon \ \
	\forall 1\leq i\leq m \big\} \,,
\end{equation}
where $b_0 := 0$ and $a_{m+1} := 1$.
Note that $\cI^{(m)}_1(\phi, \psi)-\cI^{(m),(\eps)}_1(\phi, \psi)$
is small, if we choose $\epsilon > 0$ small,
simply because the integrated integral $\cI^{(m)}_1(\phi, \psi)$ is finite.

We similarly define $I^{(N,m),(\eps)}_N(\phi, \psi)$ by restricting the sum
in \eqref{eq:prel2Q} to the set
\begin{equation} \label{eq:restrii}
	\big\{ a_{i}-b_{i-1} \ge \epsilon N \ \ \forall 1 \le i \le m+1 \,, \quad \
	b_i-a_i\geq \epsilon N \ \
	\forall 1\leq i\leq m \big\} \,,
\end{equation}
where $b_0 := 0$ and $a_{m+1} := N$.
The difference $I^{(N,m)}_N(\phi, \psi)-I^{(N,m),(\eps)}_N(\phi, \psi)$
is bounded by the sum in \eqref{eq:INmphipsibd2}
restricted to the complementary set of \eqref{eq:restrii}.
By the uniform bound \eqref{est:UNunif},
this sum is bounded by the integral in
\eqref{eq:Imphipsibd2} restricted to the complementary set of \eqref{eq:restrii0}.
Then $I^{(N,m)}_N(\phi, \psi)-I^{(N,m),(\eps)}_N(\phi, \psi)$
is small, uniformly in large $N$, for $\eps>0$ small.

As a consequence, to prove \eqref{eq:INmconv}
it suffices to show that
$$
	\lim_{N\to\infty} I^{(N,m),(\eps)}_N(\phi, \psi)
	= \pi^m \, \cI^{(m),(\eps)}_1(\phi, \psi) \qquad \text{for each }\eps>0.
$$

\smallskip

We next make a second approximation.
For large $M > 0$, we define $\cI^{(m),(\eps, M)}_1(\phi, \psi)$
by further restricting the integral
in \eqref{eq:ImQfree2} to the bounded set
\begin{equation} \label{eq:xyrestrict}
\begin{split}
	\big\{ |x_1| \le M, \quad
	& |y_i-x_i| \le M \sqrt{b_i-a_i} \ \ \forall 1 \le i \le m \,, \\
	& |x_{i}-y_{i-1}| \le M\sqrt{a_{i}-b_{i-1}} \ \ \forall\, 2 \le i \le m
	\big\} \,.
\end{split}
\end{equation}
We similarly define $I^{(N,m),(\eps,M)}_N(\phi, \psi)$,
by further restricting the sum in \eqref{eq:prel2Q}
to the set
\begin{equation} \label{eq:xyrestrict2}
\begin{split}
	\big\{ |x_1| \le M \sqrt{N}, \quad
	& |y_i-x_i| \le M \sqrt{b_i-a_i} \ \ \forall 1 \le i \le m \,, \\
	& |x_{i}-y_{i-1}| \le M\sqrt{a_{i}-b_{i-1}} \ \ \forall\, 2 \le i \le m
	\big\} \,.
\end{split}
\end{equation}
Clearly,
$\lim_{M\to\infty} \cI^{(m),(\eps, M)}_1(\phi, \psi) = \cI^{(m),(\eps)}_1(\phi, \psi)$.
We claim that, analogously,
\begin{equation}\label{eq:ImMdiffbd}
	\lim_{M\to\infty} \limsup_{N\to\infty}
	\big| I^{(N,m),(\eps, M)}_N(\phi, \psi) - I^{(N,m),(\eps)}_N(\phi, \psi) \big| =0 \,.
\end{equation}
Then we can complete the proof
of Lemma~\ref{lem:INmbd}: the asymptotic relations \eqref{eq:asqq}
and \eqref{eq:asqU} hold uniformly on the restricted sets
\eqref{eq:restrii} and \eqref{eq:xyrestrict2}, so
by dominated convergence
$$
	\text{for every } \epsilon > 0 \, , \ M < \infty: \qquad
	\lim_{N\to\infty} I^{(N,m),(\eps, M)}_N(\phi, \psi) = \pi^m \,
	\cI^{(m),(\eps, M)}_1(\phi, \psi) \,.
$$

It remains to prove \eqref{eq:ImMdiffbd}.
We can upper bound the difference in \eqref{eq:ImMdiffbd}
as
in  \eqref{eq:INmphipsibd2}-\eqref{eq:INmphibd1}:
we sum out the spatial variables recursively, starting from $y_m$,
then $x_m$, then $y_{m-1}$, etc.
\begin{itemize}
\item When we sum out $y_m$, if  $|y_m-x_m|>M \sqrt{b_m-a_m}$, then by
\eqref{eq:summU} and \eqref{eq:Udiffusive} we pick up at most a
fraction $\delta(M) \le C / M^2$ of
the upper bound in \eqref{eq:INmphibd1}.
The same applies when we sum out $y_i$ for $2 \le i \le m-1$,
if $|y_i-x_i|>M\sqrt{b_i-a_i}$.
\item
When we sum out $x_m$, if $|x_m-y_{m-1}|>M\sqrt{a_m-b_{m-1}}$, then we
restrict the sum in \eqref{eq:qconvbd} accordingly,
and we pick up again at most a fraction
$\delta(M) \le 1 /M^2$ of the upper bound in \eqref{eq:INmphibd1},
simply because $\sum_{|x| > M \sqrt{n}} q_n(x) = \P(|S_n| > M \sqrt{n})
\le 1 / M^2$.
The same applies when we sum out $x_i$ for $3 \le i \le m-1$.
\item The same argument applies to the sums over $x_2$ and $y_1$,
see the lines following \eqref{eq:INmphibd1}.
\item For the last sum over $x_1$, if $|x_1| > M \sqrt{N}$,
by \eqref{eq:qNavg} and the fact that
$\phi$ has compact support,
we pick up at most a fraction $\delta(M) = O(1/M^2)$ of the sum \eqref{eq:qualaa}.
\end{itemize}
Since for fixed $m$, there are only finitely many cases that violate \eqref{eq:xyrestrict2},
while $\delta(M) \to 0$ as $M \to \infty$, then \eqref{eq:ImMdiffbd} follows readily.
\qed

\smallskip
\section{Further bounds without triple intersections}\label{sec:furtherbds}

We recall that the
centered third moment $\bbE \big[ \big( Z^{N, \beta_N}_{s,t}(\phi, \psi)
- \bbE[Z^{N, \beta_N}_{s,t}(\phi, \psi)] \big)^3 \big]$
of the partition function averaged over both endpoints
admits the expansion \eqref{eq:MNQ}.
We then denoted by $M^{N, \rm NT}_{s, t}(\phi, \psi)$ the contribution
to \eqref{eq:MNQ}
coming from \emph{no triple intersecitons}, see \eqref{eq:trinotri}.

\smallskip

We now consider the partition functions
$Z^{N, \beta_N}_{s, t}(w, \psi)$, $Z^{N, \beta_N}_{s, t}(\phi, z)$
averaged over \emph{one} endpoint,
see \eqref{eq:avgxpsi}, \eqref{eq:avgphiy},
and also the point-to-point
partition function $Z^{\beta_N}_{s, t}(w, z)$, see \eqref{eq:Qab}
(we sometimes write $Z^{N, \beta_N}_{s, t}(w, z)$, even though it
carries no explicit dependence on $N$).

The centered third moment $\bbE \big[ \big( Z^{N, \beta_N}_{s,t}(*, \dagger)
- \bbE[Z^{N, \beta_N}_{s,t}(*, \dagger)] \big)^3 \big]$
for $* \in \{\phi, w\}$, $\dagger \in \{\psi, z\}$
can be written as in \eqref{eq:MNQ},
starting from the polynomial chaos expansions \eqref{eq:polypointQ}-\eqref{avgmulti}.
In analogy with \eqref{eq:trinotri},
we decompose
\begin{equation} \label{eq:quantities}
	\bbE \big[ \big( Z^{N, \beta_N}_{s,t}(*, \dagger)
	- \bbE[Z^{N, \beta_N}_{s,t}(*, \dagger)] \big)^3 \big]
	= M^{N, \rm NT}_{s, t}(*, \dagger) + M^{N, \rm T}_{s, t}(*, \dagger) \,,
\end{equation}
where $M^{N, \rm T}_{s, t}(*, \dagger)$
and $M^{N, \rm NT}_{s, t}(*, \dagger)$ are the contributions
with and without triple intersections.
In this section we prove the following bounds,
which will be used to prove~Proposition \ref{prop:notripleQ}.

\begin{lemma}[Bounds without triple intersections]\label{lem:notriplebds}
Let $\phi\in C_c(\R^2)$, $\psi \in C_b(\R^2)$
and  $w,z \in \Z^2$. For any $\eps>0$, as $N\to\infty$, we have
\begin{align}
	M^{N, \rm NT}_{0, N}(w, \psi) & =O(N^\eps) \,, \label{eq:ntpb1}\\
	\rule{0pt}{1.2em}\sum_{1\leq a\leq N} \sum_{z\in \Z^2}
	M^{N, \rm NT}_{0, a}(w, z) & =O(1) \,,  \label{eq:ntpb3} \\
	\sum_{1\leq a\leq N} \sum_{z\in \Z^2}
	M^{N, \rm NT}_{0, a}(\phi, z) & = O(N^{\frac{5}{2}+\eps}) \,.
	\label{eq:ntpb2}
\end{align}
\end{lemma}

We prove relations \eqref{eq:ntpb1}--\eqref{eq:ntpb2} separately below.
For the quantity $M^{N, \rm NT}_{s, t}(*, \dagger)$, when both arguments
$*, \dagger$ are functions, we derived the representation \eqref{eq:prel2Q}.
Analogous representations hold when one of the arguments $*, \dagger$ is a point.
For instance, in the point-to-point case:
\begin{equation} \label{eq:prel3Q}
\begin{split}
	M^{N, \rm NT}_{s, t}(w, z) & = \sum_{m=2}^{\infty} 3 \cdot 2^{m-1} \,
	I^{(N, m)}_{s, t}(w, z), \qquad \text{where} \\
    I^{(N, m)}_{s, t}(w, z) & := \sigma_N^{2m}
	\sumtwo{s < a_1 \le b_1 < a_2 \le b_2 < \ldots < a_m \le b_m < t}
	{x_1, y_1, x_2, y_2, \ldots, x_m, y_m \in \Z^2} \!\!\!\!\!\!\!\!\!
	q_{s,a_1}(w,x_1)^2 \, q_{s,a_2}(w,x_2) \,
	\cdot\\
	& \qquad \cdot U_N(b_1-a_1, y_1 - x_1) \, q_{b_1, a_2}(y_1,x_2) \,	
	U_N(b_2-a_2, y_2 - x_2) \,\cdot
	\\
	& \qquad \cdot \prod_{i=3}^{m}
	\Big\{ q_{b_{i-2}, a_i}(y_{i-2},x_i) \, q_{b_{i-1}, a_i}(y_{i-1},x_i) \,
	U_N(b_i-a_i, y_i - x_i) \Big\} \,\cdot \\
	& \qquad
	\cdot q_{b_{m-1},t}(y_{m-1}, z) \,
	q_{b_{m},t}(y_{m}, z)^2 \,.
\end{split}
\end{equation}
Note that in contrast to \eqref{eq:prel2Q} there is no factor $N^{-3}$,
because the definition of $Z^{\beta_N}_{s,t}(w, z)$,
unlike $Z^{N,\beta_N}_{s,t}(\phi, \psi)$, contains no such factor,
cf.\ \eqref{eq:Qab} and \eqref{avgform22}.

The identity \eqref{eq:prel3Q} holds also for $M^{N, \rm NT}_{s, t}(\phi, z)$
(replace $q_{s,a_i}(w,x_i)$ by $q^N_{s,a_i}(\phi,x_i)$, $i=1,2$)
and for $M^{N, \rm NT}_{s, t}(w, \psi)$
(replace $q_{b_{i},t}(y_{i}, z)$ by $q^N_{b_{i},t}(y_{i}, \psi)$,
$i=m-1,m$).

\medskip
\noindent
\textbf{Proof of \eqref{eq:ntpb1}.}
To estimate $I^{(N, m)}_{0, N}(w, \psi)$,
we replace $\psi$ by the constant $|\psi|_\infty$,
and we take $|\psi|_\infty \le 1$. We then focus on
$I^{(N, m)}_{0, N}(w, 1)$, and we can set $w=0$,
by translation invariance.
By the analogue of \eqref{eq:prel3Q} (note that $q^N_{b_i,t}(y_i,\psi) \equiv 1$
for $\psi \equiv 1$), we get
\be\label{eq:ntpb4}
\begin{aligned}
	I^{(N, m)}_{0, N}(w, \psi) & \le \sigma_N^{2m}
	\sumtwo{0 < a_1 \le b_1 < a_2 \le b_2 < \ldots < a_m \le b_m < N}
	{x_1, y_1, x_2, y_2, \ldots, x_m, y_m \in \Z^2} \!\!\!\!\!\!\!
	q_{a_1}(x_1)^2 U_N(b_1-a_1, y_1 - x_1) \,\cdot\\
	& \qquad \cdot \prod_{i=2}^{m}
	\Big\{ q_{b_{i-2}, a_i}(y_{i-2},x_i) \, q_{b_{i-1}, a_i}(y_{i-1},x_i) \,
	U_N(b_i-a_i, y_i - x_i) \Big\},
\end{aligned}
\ee
where we stress that
the product starts from $i=2$ and
we set $b_0:=0$ and $y_0:= 0$.
By the definition of $U_N$ in \eqref{eq:UNVar}-\eqref{def:UN}, we have
the following identity, for fixed $b_1 \in \N$, $y_1 \in \Z^2$:
\begin{equation}\label{eq:renU}
	\sigma_N^2 \sum_{0<a_1\leq b_1, \ x_1\in \Z^2}
	q_{a_1}(x_1)^2 \,\,U_N(b_1-a_1,y_1-x_1) = U_N(b_1,y_1).
\end{equation}
Therefore we can rewrite \eqref{eq:ntpb4} as
\be\label{eq:ntpb44}
\begin{aligned}
	I^{(N, m)}_{0, N}(w, \psi) & \le \sigma_N^{2(m-1)}
	\sumtwo{0 < b_1 < a_2 \le b_2 < \ldots < a_m \le b_m < N}
	{y_1, x_2, y_2, \ldots, x_m, y_m \in \Z^2} \!\!\!\!\!\!\!
	U_N(b_1, y_1) \,\cdot\\
	& \qquad \cdot \prod_{i=2}^{m}
	\Big\{ q_{b_{i-2}, a_i}(y_{i-2},x_i) \, q_{b_{i-1}, a_i}(y_{i-1},x_i) \,
	U_N(b_i-a_i, y_i - x_i) \Big\}.
\end{aligned}
\ee

We now sum out the spatial variables $y_m$, $x_m$, \ldots,
$y_2$, $x_2$, $y_1$,
arguing as in \eqref{eq:summU}-\eqref{eq:qconvbd},
to get the following upper bound, analogous to \eqref{eq:INmphipsibd2},
for a suitable $c < \infty$:
\begin{equation} \label{eq:louba}
\begin{split}
	& I^{(N, m)}_{0, N}(w, \psi) \le c^{m} \,
	\sigma_N^{2(m-1)} \!\!\!
	\sum\limits_{\substack{0 < b_1 < a_2 \le b_2 < \ldots \\
	\ldots < a_m \le b_m < N}} \!\!\!\!
	U_N(b_1) \prod_{i=2}^{m}
	\frac{U_N(b_i - a_i)}{\sqrt{(a_i-b_{i-1})(a_i-b_{i-2})}} \,.
\end{split}
\end{equation}

Then
we set $u_i := a_i-b_{i-1}$, $v_i:=b_i-a_i$ for $2\leq i\leq m$,
and we rename $u_1 := b_1$.
This allows to bound $a_i - b_{i-2} \ge u_i + u_{i-1}$ for all $i \ge 2$
(including $i=2$, since $a_i - b_{i-2} = a_2 \ge u_2 + b_1$).
Then, for $\lambda > 0$,
we insert the factor $e^{\lambda} \prod_{i=2}^m e^{-\lambda (\frac{v_i}{N})}>1$
and we estimate, as in \eqref{eq:INmphibd11},
\begin{equation}\label{eq:louba3bis}
\begin{split}
	I^{(N, m)}_{0, N}(w, \psi) & \le
	e^\lambda \, c^m \,
	 \Bigg(\sigma_N^2\sum_{v=0}^N e^{-\lambda \frac{v}{N}} U_N(v)\Bigg)^{m-1} \, \cdot \\
	& \qquad\ \cdot \Bigg\{
	\sum_{u_1 = 1}^N U_N(u_1) \!\!\!\!\!
	\sum_{0 < u_2, \ldots, u_m < N} \,
	\prod_{i=2}^{m}
	\frac{1}{\sqrt{u_i(u_i + u_{i-1})}} \Bigg\} \,.
\end{split}
\end{equation}
The first parenthesis is
$\le c \big( \tfrac{1}{\log N} + C_\lambda \big)$, see \eqref{eq:UNsumbd}.
Then we replace $u_i$ by $N u_i$, with $u_i \in \frac{1}{N}\Z$,
and bound Riemann sums by integrals, by
monotonicity. This yields (for a possibly larger $c$)
\begin{equation}\label{eq:louba4bis}
\begin{split}
	I^{(N, m)}_{0, N}(w, \psi) & \le e^\lambda \, c^m \,
	\big( \tfrac{1}{\log N} + C_\lambda \big)^{m-1} \, \cdot\\
	& \qquad \cdot
	\Bigg\{ \sumtwo{u_1 \in \frac{1}{N}\Z}{\frac{1}{N} \le u_1 \le 1} U_N(N u_1) \!\!\!\!\!\!\!
	\int\limits_{0 < u_2, \ldots, u_m < 1} \,
	\prod_{i=2}^{m}
	\frac{1}{\sqrt{u_i(u_i + u_{i-1})}} \, \dd \vec{u}\Bigg\} \,.
\end{split}
\end{equation}
The integral equals $\phi^{(m-1)}(u_1)$, see \eqref{eq:phik}.
We bound $U_N(N u_1) \le c_2 \, \frac{\log N}{N}
\, \hat G_{\theta}(u_1) $
by \eqref{est:UNunif} and \eqref{eq:Gthetaunif}, since $u_1 > 0$.
Recalling that $\phi^{(m-1)}(\cdot)$ is decreasing, we get
\begin{equation}\label{eq:lastspl}
\begin{split}
	I^{(N, m)}_{0, N}(w, \psi) & \le e^\lambda \, c^m \,
	\big( \tfrac{1}{\log N} + C_\lambda \big)^{m-1} \,
	\Bigg\{ \frac{\log N}{N}
	\sum_{\substack{u_1 \in \frac{1}{N} \Z \\ 0 < u_1 < 1}} \ \hat G_{\theta}(u_1)
	\Bigg\} \, \phi^{(m-1)}(\tfrac{1}{N}) \\
	& \le e^\lambda \, c^m \,
	\big( \tfrac{1}{\log N} + C_\lambda \big)^{m-1} \,
	(\log N) \, C_\lambda \,
	\phi^{(m-1)}(\tfrac{1}{N}) \,,
\end{split}
\end{equation}
where for the last inequality, recalling that $\hat G_{\theta}(\cdot)$ is decreasing,
we bounded the Riemann sum in brackets by the integral
$\int_0^1 \hat G_{\theta}(u_1) \, \dd u_1 \le C_\lambda$,
see \eqref{eq:Clambda}.

Putting together \eqref{eq:prel3Q} and \eqref{eq:lastspl},
we can finally estimate
\begin{equation*}
\begin{split}
	M^{N, \rm NT}_{0, N}(w, \psi)
	& \le 3 \sum_{m\ge 2} 2^{m} \, I^{(N, m)}_{0, N}(w, \psi)
	\le 3 e^\lambda \, (\log N) \sum_{m\ge 2}
	\big[ 2c \, \big( \tfrac{1}{\log N} + C_\lambda \big)\big]^{m} \, \phi^{(m-1)}(\tfrac{1}{N}) \,,
\end{split}
\end{equation*}
and using the first inequality in \eqref{eq:estphiklast} we obtain
\begin{equation*}
\begin{split}
	M^{N, \rm NT}_{0, N}(w, \psi)
	& \le 3 e^\lambda \, (\log N) \sum_{m\ge 2}
	\big[ 64c \, \big( \tfrac{1}{\log N} + C_\lambda \big)\big]^{m}  \sum_{i=0}^{m} \frac{1}{i!}
	\big( \tfrac{1}{2} \log (e^2 N) \big)^{i} \\
	& \le 3 e^\lambda \, (\log N) \sum_{i \ge 0}
	\frac{1}{i!}
	\big( \tfrac{1}{2} \log (e^2 N) \big)^{i}
	\sum_{m\ge i} \big[ 64c \, \big( \tfrac{1}{\log N} + C_\lambda \big)\big]^{m}  \\
	& = \tfrac{3 e^\lambda}{1-[ 64c \, ( \frac{1}{\log N} + C_\lambda )]} \, (\log N) \,
	\sum_{i \ge 0}
	\frac{1}{i!} \Big(32 c \, \big( \tfrac{1}{\log N} + C_\lambda \big)
	\, \log (e^2 N) \Big)^{i} \\
	& = \tfrac{3 e^\lambda}{1-[ 64c \, ( \frac{1}{\log N} + C_\lambda )]} \, (\log N) \,
	\big( e^2 N \big)^{32 c \, ( \frac{1}{\log N} + C_\lambda )} \,.
\end{split}
\end{equation*}
Since $\lim_{\lambda \to \infty} C_\lambda = 0$,
see \eqref{eq:Clambda}, given $\epsilon > 0$ we can fix $\lambda$ large so that
$32 c \, C_\lambda  < \frac{\epsilon}{2}$. Then for large $N$ the exponent of $(e^2 N)$
in the last term is $< \epsilon$, which proves \eqref{eq:ntpb1}.\qed

\medskip
\noindent
\textbf{Proof of \eqref{eq:ntpb3}.}
From the first line of \eqref{eq:prel3Q} we can write
\begin{equation}\label{eq:ftfl}
	\sum_{1 \le a \le N} \sum_{z \in \Z^2} M^{N, \rm NT}_{0, a}(w, z)
	= \sum_{m\ge 2} 3 \cdot 2^{m-1} \,
	\sum_{1 \le a \le N} \sum_{z \in \Z^2} I^{(N, m)}_{0, a}(w, z)
\end{equation}

To estimate $\sum_{1 \le a \le N} \sum_{z \in \Z^2} I^{(N, m)}_{0, a}(w, z)$,
we use the representation \eqref{eq:prel3Q} with $s=0$ and $t=a$.
We may also set $w = 0$ (by translation invariance).
We first perform the sum over $a_1$ and $b_m$, using
\eqref{eq:renU} and the symmetric relation
\begin{align}
	\label{eq:symrel}
	\sigma_N^2 \sum_{a_m\leq b_m<a,\ y_m\in \Z^2}  U_N(b_m-a_m,y_m-x_m) \,
	q_{b_m,a}(y_m,z)^2 &= U_N(a-a_m,z-x_m) \,.
\end{align}
We then obtain
\begin{align}
	\sumtwo{1\leq a \leq N}{z\in\Z^2}
	I^{(N, m)}_{0, a}(w, z) =\ & \sigma_N^{2(m-2)} \!\!\!\!\!\!\!\!
	\sumtwo{0 < b_1 < a_2 \le b_2 < \ldots < a_m  < a\leq N}
	{y_1, x_2, \ldots, y_{m-1}, x_m, z \in \Z^2} \!\!\!\!\!\!\!\!\!
     U_N(b_1, y_1)  \, \cdot\nonumber \\
	& \cdot \prod_{i=2}^{m-1}
	\Big\{ q_{b_{i-2}, a_i}(y_{i-2},x_i) \, q_{b_{i-1}, a_i}(y_{i-1},x_i) \,
	U_N(b_i-a_i, y_i - x_i) \Big\}   \label{eq:prel4Qb}\\
	& \hspace{-70pt} \cdot \Big\{q_{b_{m-2}, a_m}(y_{m-2},x_m) \,
	q_{b_{m-1}, a_m}(y_{m-1},x_m) \,
	U_N(a-a_m, z - x_m) \Big\}  \, q_{b_{m-1},a}(y_{m-1},z). \nonumber
\end{align}

If we rename $y_m:=z$ and $b_m:=a$, then we see that \eqref{eq:prel4Qb}
differs from \eqref{eq:ntpb44} only for the factor $\sigma_N^{2(m-2)}$
(instead of $\sigma_N^{2(m-1)}$) and for the presence of the last kernel
$q_{b_{m-1},a}(y_{m-1},z)
= q_{b_{m-1},b_m}(y_{m-1},y_m)$. The latter can be estimated using \eqref{eq:qas}:
\begin{equation}
	\label{eq:lastkernel}
	q_{b_{m-1},a}(y_{m-1},z )\leq \frac{c}{b_m-b_{m-1}}
	\leq \frac{c}{\sqrt{b_m-b_{m-1}}}
\end{equation}
for some suitable constant $c$.
As in \eqref{eq:louba}, we
first sum out the spatial variables, getting
\begin{equation*}
\begin{split}
	& \sumtwo{1\leq a \leq N}{z\in\Z^2}
	I^{(N, m)}_{0, a}(w, z)  \le c^{m} \,
	\sigma_N^{2(m-2)} \!\!\!\!\!\!\!\!\!\!
	\sum\limits_{\substack{0 < b_1 < a_2 \le b_2 < \ldots \\
	\ldots < a_m \le b_m < N}} \!\!\!\!\!\!\!\!
	U_N(b_1) \prod_{i=2}^{m}
	\frac{U_N(b_i - a_i)}{\sqrt{(a_i-b_{i-1})(a_i-b_{i-2})}}
	\, \frac{1}{\sqrt{b_m-b_{m-1}}} \,.
\end{split}
\end{equation*}
Then we set
$u_1 := b_1$ and
$u_i := a_i-b_{i-1}$, $v_i:=b_i-a_i$ for $2\leq i\leq m$,
which allows to bound $a_i - b_{i-2} \ge u_i + u_{i-1}$ for $i \ge 2$,
as well as $b_m - b_{m-1} \ge u_m$. Then, for $\lambda > 0$,
we insert the factor $e^{\lambda} \prod_{i=2}^m e^{-\lambda (\frac{v_i}{N})}>1$
and, by \eqref{eq:UNsumbd},
we obtain the following analogue of \eqref{eq:louba3bis}:
\begin{equation}\label{eq:louba4bisnew}
\begin{split}
	\sumtwo{1\leq a \leq N}{z\in\Z^2}
	I^{(N, m)}_{0, a}(w, z) & \le
	 \, c^m \, e^\lambda \, (\log N)
	\big( \tfrac{1}{\log N} + C_\lambda \big)^{m-1} \, \cdot \\
	& \quad \cdot
	\Bigg\{
	\sum_{u_1 = 1}^N U_N(u_1) \!\!\!\!\!
	\sum_{0 < u_2, \ldots, u_m < N} \,
	\prod_{i=2}^{m}
	\frac{1}{\sqrt{u_i(u_i + u_{i-1})}} \, \frac{1}{\sqrt{u_m}} \Bigg\}  \,,
\end{split}
\end{equation}
where the extra $\log N$
comes from having $\sigma_N^{2(m-2)}$
instead of $\sigma_N^{2(m-1)}$ (by \eqref{eq:sigmaN} and \eqref{eq:olap}).

We now switch to macroscopic variables,
replacing $u_i$ by $N u_i$, with $u_i \in \frac{1}{N}\Z \cap (0,1)$,
and bound $U_N(N u_1) \le c_1 \, \frac{\log N}{N} \, \hat G_{\theta}(u_1)$
since $u_1 > 0$,
by \eqref{est:UNunif} and \eqref{eq:Gthetaunif}.
We then replace the Riemann sum in brackets by the corresponding integrals,
similar to \eqref{eq:louba4bis}, with an important difference
(for later purposes):
since $u_i \in \frac{1}{N} \Z$ and $u_i > 0$,
we can restrict the integration on $u_i \ge \frac{1}{N}$
(possibly enlarging the value of $c$). This leads to
\begin{equation}\label{eq:louba5bisnew}
\begin{split}
	\sumtwo{1\leq a \leq N}{z\in\Z^2}
	& I^{(N, m)}_{0, a}(w, z) \le
	(\log N) \, e^\lambda \, c^m \,
	\big( \tfrac{1}{\log N} + C_\lambda \big)^{m-1} \, \cdot \\
	& \quad \cdot \frac{\log N}{\sqrt{N}} \,
	\Bigg\{
	\int\limits_{\frac{1}{N} \le u_1, u_2, \ldots, u_m < 1} \!\!\!\!\!\!\!\!
	\hat G_{\theta}(u_1) \,
	\Bigg( \prod_{i=2}^{m}
	\frac{1}{\sqrt{u_i(u_i + u_{i-1})}} \Bigg) \,
	\frac{1}{\sqrt{u_m}} \, \dd \vec{u} \Bigg\}  \,,
\end{split}
\end{equation}
where the factor $\frac{\log N}{\sqrt{N}}$ comes from the estimate on $U_N(N u_1)$
and from the last kernel $1/\sqrt{u_m}$.

If we define $\widehat\phi^{(k)}(\cdot)$ as the following modification
of \eqref{eq:phik}:
\begin{equation} \label{eq:Fkhat}
\begin{aligned}
	\widehat \phi^{(0)}(u) & := \frac{1}{\sqrt u} \,, \qquad \text{and for $k \ge 1$:}\quad \
	\widehat\phi^{(k)}(u) := \int_{\frac{1}{N}}^1
	\frac{1}{\sqrt{s(s + u)}} \, \widehat \phi^{(k-1)}(s) \,  \dd s \,,
\end{aligned}
\end{equation}
then, recalling \eqref{eq:Clambda},
we can rewrite \eqref{eq:louba5bisnew} as follows:
\begin{equation} \label{eq:dein}
\begin{split}
	\sumtwo{1\leq a \leq N}{z\in\Z^2}
	I^{(N, m)}_{0, a}(w, z) & \le e^\lambda \, c^m \, \frac{(\log N)^2}{\sqrt{N}}  \,
	(C_{\lambda,N})^{m-1} \, \int_{\frac{1}{N}}^1 \hat G_{\theta}(u) \,
	\widehat\phi^{(m-1)}(u) \, \dd u \, , \\
	&  \quad \text{where we set } \
	C_{\lambda,N} := \tfrac{1}{\log N} + C_\lambda
	= \tfrac{1}{\log N} + \tfrac{\sfc_\theta}{2 + \log \lambda} \,.
\end{split}
\end{equation}

Similar to Lemma \ref{lem:phikbd}, we have the
following bound on $\widehat \phi^{(k)}$, that we prove later.

\begin{lemma}\label{lem:phikbdhat}
For all $k\in\N$,
the function $\widehat\phi^{(k)}(v)$ is decreasing on $(0,1)$, and satisfies
\begin{equation} \label{eq:estphiklasthat}
	\widehat \phi^{(k)}(u) \le 32^k  \sum_{i=0}^k \frac{1}{2^i\, i!}
	\frac{\big(\log (e^2Nu)\big)^{i}}{\sqrt u}
	\le 32^k \, e \, \sqrt{N} \,,
	\qquad \forall u \in \big(\tfrac{1}{N}, 1\big).
\end{equation}
\end{lemma}

We need to estimate the integral in \eqref{eq:dein}, when we plug
in the bound \eqref{eq:estphiklasthat}.
We first consider the contribution from $u < \frac{1}{\sqrt{N}}$.
In this case $\hat G_{\theta}(u) \le \frac{4c_{\theta}}{(\log N)^2} \frac{1}{u}$,
see \eqref{eq:Gthetaunif}, hence
\begin{equation}
\begin{aligned}
	& \int_{\frac{1}{N}}^{\frac{1}{\sqrt N}}
	\hat G_\theta(u) \frac{\big(\log (e^2Nu)\big)^{i}}{\sqrt u} \dd u
	\leq \frac{4c_{\theta}}{(\log N)^2}
	\int_{\frac{1}{N}}^{\frac{1}{\sqrt N}}
	\frac{\big(\log (e^2Nu)\big)^{i}}{u^{\frac{3}{2}}} \dd u \\
	& \qquad \leq \frac{4e c_{\theta} \, \sqrt{N}}{(\log N)^2}
	\int_1^\infty
	\frac{\big(\log w\big)^{i}}{w^{\frac{3}{2}}} \dd w
	= \frac{8e c_{\theta} \, \sqrt{N}}{(\log N)^2} \int_0^\infty
	2^i  s^i e^{-s} \dd s
	= C \, \frac{\sqrt{N}}{(\log N)^2}
	\, 2^i \, i! \,,
\end{aligned}
\end{equation}
where we first made the change of variables $e^2Nu=w$, and then $w=e^{2s}$,
and denote $C = 8e c_{\theta}$ for short.
Then it follows by \eqref{eq:estphiklasthat} that
\begin{equation*}
	\int_{\frac{1}{N}}^{\frac{1}{\sqrt N}}
	\hat G_\theta(u) \, \widehat\phi^{(m-1)}(u) \, \dd u
	\le C \, 32^{m} \, \underset{\mathsf{A}_{m,N}}{\underbrace{ m \, \frac{\sqrt{N}}{(\log N)^2} }}
	\,.
\end{equation*}
We then consider the contribution from $u\geq \frac{1}{\sqrt N}$.
Since $\hat G_{\theta}(u) \le \frac{c_{\theta}}{u}$, we have
\begin{equation*}
	\int_{\frac{1}{\sqrt N}}^{1}
	\hat G_\theta(u) \frac{\big(\log (e^2Nu)\big)^{i}}{\sqrt u} \dd u
	\le c_{\theta} \, \big( \log (e^2 N) \big)^i \int_{\frac{1}{\sqrt N}}^{1}
	\frac{1}{u^{3/2}} \, \dd u =
	2 c_{\theta} \, N^{\frac{1}{4}} \, \big( \log (e^2 N) \big)^i \,,
\end{equation*}
hence by \eqref{eq:estphiklasthat}
\begin{equation*}
	\int_{\frac{1}{\sqrt N}}^1
	\hat G_\theta(u) \, \widehat\phi^{(m-1)}(u) \, \dd u
	\le C \, 32^m \, \underset{\mathsf{B}_{m,N}}{\underbrace{ N^{\frac{1}{4}} \, \sum_{i=0}^{m-2}
	\frac{1}{2^i \, i!} \, \big( \log (e^2 N) \big)^i }} \,.
\end{equation*}

By \eqref{eq:ftfl} and \eqref{eq:dein},
we finally see that
\begin{equation*}
\begin{split}
	\sum_{1 \le a \le N} \sum_{z \in \Z^2}
	& M^{N, \rm NT}_{0, a}(w, z)
	\le 3 \sum_{m\ge 2} 2^{m-1}
	\, e^\lambda \, c^m  \, \frac{(\log N)^2}{\sqrt{N}}  \,
	(C_{\lambda,N})^{m-1}  \, C \, 32^m
	\big\{  \mathsf{A}_{m,N} + \mathsf{B}_{m,N} \big\} \\
	& \qquad \le C' \, e^\lambda
	\Bigg\{ \sum_{m\ge 2} (64 \, c \, C_{\lambda,N})^{m-1} \, m \\
	& \qquad\qquad\qquad\qquad + \sum_{m\ge 2} (64 \, c \, C_{\lambda,N})^{m-1}
	\bigg( \frac{(\log N)^2}{N^{\frac{1}{4}}} \sum_{i=0}^{m-2}
	\frac{1}{2^i \, i!} \, \big( \log (e^2 N) \big)^i \bigg) \Bigg\} \,,
\end{split}
\end{equation*}
with $C' := 3 \cdot 32 c $. If we fix $\lambda$ large enough, then for large $N$ we have
$64 \, c \, C_{N,\lambda} < 1$ (recall \eqref{eq:dein}),
then the first sum in the RHS is finite, in agreement with our goal \eqref{eq:ntpb3}.
Concerning the second sum, we can estimate it by
\begin{equation*}
\begin{split}
	& \le \frac{(\log N)^2}{N^{\frac{1}{4}}}  \,
	\sum_{i\ge 0} \frac{1}{2^i \, i!} \, \big( \log (e^2 N) \big)^i
	\sum_{m \ge i+2} (64 \, c \, C_{\lambda,N})^{m-1} \\
	& \le \frac{(\log N)^2}{N^{\frac{1}{4}}}  \,
	\sum_{i\ge 0} \frac{1}{2^i \, i!} \, \big(
	\log (e^2 N) \big)^i \, \frac{(64 \, c \, C_{\lambda,N})^i}{1-64 \, c \, C_{\lambda,N}}
	= \frac{(\log N)^2}{N^{\frac{1}{4}}}  \,
	\frac{(e^2 N)^{32 \, c \, C_{\lambda,N}}}{1-64 \, c \, C_{\lambda,N}} \,.
\end{split}
\end{equation*}
If we fix $\lambda$ large enough, then for large $N$ we have
that the exponent is $32 \, c \, C_{\lambda,N} < \frac{1}{4}$, hence
the last term is $o(1)$ as $N\to\infty$.
This completes the proof of \eqref{eq:ntpb3}.\qed

\smallskip

In order to prove Lemma~\ref{lem:phikbdhat}, we need the
following analogue of Lemma~\ref{th:intlog}.

\begin{lemma}\label{lem:intloghat}
For all $i\in\N_0$ and $v\in \big(\frac{1}{N}, 1\big)$,
\begin{equation}\label{eq:logitbd}
\int_{\tfrac{1}{N}}^1 \frac{\big(\log(e^2Ns))\big)^i}{s\sqrt{s+v}} \,\dd s \leq \frac{2^{i+1}}{\sqrt{v}}
i! \sum_{j=0}^{i+1} \frac{\big(\log(e^2Nv)\big)^j}{2^j j!}.
\end{equation}
\end{lemma}
\begin{proof} We can bound
$$
\int_{\tfrac{1}{N}}^1 \frac{\big(\log(e^2Ns))\big)^i}{s\sqrt{s+v}} \,\dd s \leq
\underbrace{\frac{1}{\sqrt v}\int_{\tfrac{1}{N}}^v \frac{\big(\log(e^2Ns))\big)^i}{s} \,\dd s}_{B}
+ \underbrace{\int_v^1 \frac{\big(\log(e^2Ns))\big)^i}{s\sqrt{s}} \,\dd s}_{A}.
$$
For $B$, we make the change of variable $u= \log (e^2Ns)$ to obtain
\begin{equation}
B= \frac{1}{\sqrt v} \int_2^{\log e^2Nv} u^i \dd u \leq \frac{(\log e^2Nv)^{i+1}}{\sqrt{v}(i+1)}.
\end{equation}
For $A$, we make the change of variable $y=\frac{1}{2}\log e^2Ns$ and apply \eqref{eq:gammaint} to obtain
\begin{equation}
A \leq  2^{i+1}e\sqrt{N} \int_{\frac{1}{2}\log e^2Nv}^\infty y^i e^{-y} \dd y \leq 2^{i+1}e\sqrt{N} e^{-\frac{1}{2}\log e^2Nv} \sum_{j=0}^i \frac{i!}{j!}\Big(\frac{1}{2}\log e^2Nv\Big)^j.
\end{equation}
Combined with the bound for $B$, this gives precisely \eqref{eq:logitbd}.
\end{proof}

\begin{proof}[Proof of Lemma~\ref{lem:phikbdhat}]
We follow the proof of Lemma \ref{lem:phikbd}.
We first show that for all $k\in\N$
\begin{equation} \label{eq:estphikhat}
	\widehat\phi^{(k)}(v) \le \frac{2^k}{\sqrt v} \, \sum_{i=0}^k \frac{\hat c_{k,i}}{2^i \, i!} \,
	\big(\log (e^2Nv)\big)^{i}  \qquad
	\forall v \in \Big(\frac{1}{N}, 1\Big) \,,
\end{equation}
for suitable coefficients $\hat c_{k,i}$.
For $k=1$, note that by \eqref{eq:Fkhat}
$$
	\widehat \phi^{(1)}(v)
	= \int_{\frac{1}{N}}^1 \frac{1}{\sqrt{s(s + v)}} \frac{1}{\sqrt s}  \dd s
	\le \int_{\frac{1}{N}}^v \frac{1}{\sqrt v}\frac{1}{s} \, \dd s
	+ \int_v^1 \frac{1}{s\sqrt{s}} \, \dd s
	\leq \frac{\log (e^2Nv)}{\sqrt v}.
$$
Therefore \eqref{eq:estphikhat} holds for $k=1$ with $\hat c_{1, 0}=0$ and $\hat c_{1, 1}=1$.

Assume that we have established \eqref{eq:estphikhat} up to $k-1$, then
\begin{equation}\label{eq:phihatrec}
	\widehat\phi^{(k)}(v)
	= \int_{\frac{1}{N}}^1 \frac{1}{\sqrt{s(s + v)}} \, \widehat \phi^{(k-1)}(s) \,  \dd s
	\leq 2^{k-1} \sum_{i=0}^{k-1} \frac{\hat c_{k-1,i}}{2^i \, i!} \,
	\int_{\frac{1}{N}}^1 \frac{\big(\log (e^2Ns)\big)^{i}}{s\sqrt{s + v}} \,  \dd s.
\end{equation}
Applying Lemma~\ref{lem:intloghat}, we obtain
$$
	\widehat\phi^{(k)}(v)
	\leq 2^{k-1} \sum_{i=0}^{k-1} \frac{\hat c_{k-1,i}}{2^i \, i!} \,
	\frac{2^{i+1}}{\sqrt{v}} i! \sum_{j=0}^{i+1} \frac{\big(\log(e^2Nv)\big)^j}{2^j j!}
	= \frac{2^k}{\sqrt v} \sum_{j=0}^k \Big(\sum_{i=(j-1)^+}^{k-1}\hat c_{k-1, i}\Big)
	\frac{(\log e^2Nv)^j}{2^j j!} \,.
$$
This shows that \eqref{eq:estphikhat} holds, provided the coefficients $\hat c_{k,i}$ satisfy
the recursion
\begin{equation}
	\hat c_{k, j} = \sum_{i=(j-1)^+}^{k-1} \hat c_{k-1, i},
\end{equation}
which differs from the recursion \eqref{eq:rec} for $c_{k,i}$
by a missing factor of 2. Note that $\hat c_{1,1}$ here is also only half of
$c_{1,1}$ in \eqref{eq:c1}.
Therefore we have the identity $\hat c_{k,i} = 2^{-k} c_{k,i}$, and Lemma~\ref{lem:Sk}
gives the bound $\hat c_{k, i}\leq 16^k$.
Substituting this bound into \eqref{eq:estphikhat} then proves Lemma~\ref{lem:phikbdhat}.
\end{proof}

\medskip
\noindent
\textbf{Proof of \eqref{eq:ntpb2}.}
We start from the analogue of \eqref{eq:prel3Q},
with $q_{s,a_i}(w,x_1), q_{s,a_i}(w,x_1)$ replaced by $q^N_{s,a_i}(\phi,x_1),
q^N_{s,a_i}(\phi,x_2)$. Applying relation \eqref{eq:symrel}, we can write
\begin{align}
	\sumtwo{1\leq a \leq N}{z\in\Z^2}I^{(N, m)}_{0, a}(\phi, z) &
	=\ \sigma_N^{2(m-1)} \!\!\!\!\!\!\!\!
	\sumtwo{0 < a_1\leq b_1 < a_2 <\ldots < a_m  < a\leq N}
	{x_1, y_1, x_2, y_2, \ldots, x_m, z \in \Z^2} \!\!\!\!\!\!\!\!\!
     q^N_{0, a_1}(\phi, x_1)^2q^N_{0, a_2}(\phi, x_2) \, \cdot \nonumber \\
    & \ \cdot U_N(b_1-a_1, y_1-x_1) q_{b_1, a_2}(y_1, x_2) U_N(b_2-a_2, y_2-x_2) \, \cdot \label{eq:prel5Q}\\
	 \cdot \prod_{i=3}^{m} &
	\Big\{ q_{b_{i-2}, a_i}(y_{i-2},x_i) \, q_{b_{i-1}, a_i}(y_{i-1},x_i) \,
	U_N(b_i-a_i, y_i - x_i) \Big\}  q_{b_{m-1},a}(y_{m-1},z) . \nonumber
\end{align}

We rename $y_m:=z$, $b_m:=a$ and bound
$q_{b_{m-1},a}(y_{m-1},z) \le (c\sqrt{b_m - b_{m-1}})^{-1}$, as in
\eqref{eq:lastkernel}.
Next we sum over the space variables
$y_m, x_m, \ldots$ until $y_3, x_3, y_2$, as in \eqref{eq:summU}-\eqref{eq:qconvbd},
which has the effect of replacing $U_N(b_i-a_i, y_i - x_i)$ by $U_N(b_i-a_i)$ and
$q_{b_{i-2}, a_i}(y_{i-2},x_i) \, q_{b_{i-1}, a_i}(y_{i-1},x_i)$ by
$c \, (\sqrt{(a_i-b_{i-1})(a_i - b_{i-2})})^{-1}$. Then
we bound $q^N_{0,a_2}(\phi, x_2) \le |\phi|_\infty$,
see \eqref{eq:qNavg},
after which the sum over $x_2$ gives $1$,
the sum over $y_1$ gives $U_N(b_1- a_1)$,
and the sum over $x_1$ is bounded by $c \, N$,
as in \eqref{eq:qualaa}. This leads to estimate the RHS of \eqref{eq:prel5Q} by
\begin{equation*}
\begin{split}
	& c^{m} \, N \,
	\sigma_N^{2(m-1)} \!\!\!\!\!\!\!\!\!\!
	\sum\limits_{\substack{0< a_1 \le b_1 < \ldots \\
	\ldots < a_m \le b_m < N}} \!\!\!\!\!\!\!\!
	U_N(b_1-a_1) \, U_N(b_2-a_2) \prod_{i=3}^{m}
	\frac{U_N(b_i - a_i)}{\sqrt{(a_i-b_{i-1})(a_i-b_{i-2})}}
	\, \frac{1}{\sqrt{b_m-b_{m-1}}} \,.
\end{split}
\end{equation*}

We now set
$u_i := a_i-b_{i-1}$ and $v_i:=b_i-a_i$ for $1\leq i\leq m$,
with $b_0 := 0$,
and bound $a_i - b_{i-2} \ge u_i + u_{i-1}$,
while $b_m - b_{m-1} \ge u_m$. Then
we insert the factor $e^{\lambda} \prod_{i=1}^m e^{-\lambda (\frac{v_i}{N})}>1$,
for $\lambda > 0$, and by \eqref{eq:UNsumbd} we bound
the last display by
\begin{equation}\label{eq:louba5bisnew2}
\begin{split}
	c^m \, e^\lambda \, N \, (\log N) \,
	\big( \tfrac{1}{\log N} + C_\lambda \big)^{m} \,
	\Bigg\{
	\sum_{0 < u_1, \ldots, u_m < N} \,
	\prod_{i=3}^{m}
	\frac{1}{\sqrt{u_i(u_i + u_{i-1})}} \, \frac{1}{\sqrt{u_m}} \Bigg\}  \,,
\end{split}
\end{equation}
which is an analogue of \eqref{eq:louba4bisnew}. The exponent
of $(\tfrac{1}{\log N} + C_\lambda)$ equals $m$,
because we have $m$ factors $U_N(b_i - a_i)$,
and the extra $\log N$ comes from having $m-1$ powers
of $\sigma_N^2$.

We now switch to macroscopic variables,
replacing $u_i$ by $N u_i$, with $u_i \in \frac{1}{N}\Z \cap (0,1)$,
and replace the Riemann sum in brackets by the corresponding integrals,
where as in \eqref{eq:louba5bisnew}
we restrict the integration on $u_i \ge \frac{1}{N}$
(possibly enlarging the value of $c$). This leads to
\begin{equation}\label{eq:louba5bisnew3}
\begin{split}
	\sumtwo{1\leq a \leq N}{z\in\Z^2}
	& I^{(N, m)}_{0, a}(\phi, z) \le
	c^m \, e^\lambda \, N \, (\log N) \,
	\big( \tfrac{1}{\log N} + C_\lambda \big)^{m} \, \cdot \\
	& \quad \cdot N^{\frac{3}{2}} \,
	\Bigg\{
	\int\limits_{\frac{1}{N} \le u_1, u_2, \ldots, u_m < 1}
	\Bigg( \prod_{i=3}^{m}
	\frac{1}{\sqrt{u_i(u_i + u_{i-1})}} \Bigg) \,
	\frac{1}{\sqrt{u_m}} \, \dd \vec{u} \Bigg\}  \,,
\end{split}
\end{equation}
where the factor $N^{\frac{3}{2}}$ arises by matching
the normalization factor $N^{-m}$ of the Riemann sum
and the term $N^{-(m-2)-\frac{1}{2}}$ generated by the square roots,
when we set $u_i \rightsquigarrow N u_i$.

Note that the variable $u_1$ does not
appear in the function to be integrated in \eqref{eq:louba5bisnew3},
so the integral over $u_1$ is at most $1$.
Recalling the definition \eqref{eq:Fkhat} of $\widehat \phi^{(k)}$, we have
\begin{equation*}
	\sumtwo{1\leq a \leq N}{z\in\Z^2}
	I^{(N, m)}_{0, a}(\phi, z) \le
	c^m \, e^\lambda \, N^{\frac{5}{2}} \, (\log N) \,
	\big( \tfrac{1}{\log N} + C_\lambda \big)^{m}
	\, \int_{\frac{1}{N}}^1 \phi^{(m-2)}(u_2) \, \dd u_2 \,.
\end{equation*}
By Lemma \ref{lem:phikbdhat}, we have
$$
	\int_{\frac{1}{N}}^{1} \widehat \phi^{(m-2)}(u) \dd u
	\leq 32^{m-2}
	\sum_{i=0}^{m-2} \frac{1}{2^i \, i!}
	\int_{0}^1 \frac{(\log (e^2Nu))^{i}}{\sqrt u}  \, \dd u \leq 32^{m-2}
	\sum_{i=0}^{m-2} \frac{(\log (e^2N) )^{i}}{2^i i!} \, 2 \,.
$$
Therefore, if we set $C_{\lambda,N} :=
C_{\lambda,N} := \frac{1}{\log N} + C_\lambda$ as in \eqref{eq:dein},
recalling \eqref{eq:prel3Q} we get
\begin{equation*}
\begin{split}
	\sumtwo{1\leq a\leq N}{z\in \Z^2} M^{N, \rm NT}_{0, a}(\phi, z)
	& \le 3 \sum_{m\ge 2} 2^{m}
	\, \sumtwo{1\leq a\leq N}{z\in \Z^2} I^{N, m}_{0, a}(\phi, z) \\
	& \le 3 \, e^\lambda \, N^{\frac{5}{2}} \, (\log N) \,
	\sum_{m\ge 2} (64c \, C_{\lambda,N})^m
	\sum_{i=0}^{m} \frac{(\log (e^2N) )^{i}}{2^i i!} \\
	& \le 3 \, e^\lambda \, N^{\frac{5}{2}} \, (\log N) \,
	\sum_{i\ge 0} \frac{(\log (e^2N) )^{i}}{2^i i!}
	\sum_{m\ge i} (64c \, C_{\lambda,N})^m 	\\
	& \le \frac{3 \, e^\lambda}{1-64c \, C_{\lambda,N}} \, N^{\frac{5}{2}} \, (\log N) \,
	\sum_{i \ge 0} \frac{\big( 32c \, C_{\lambda,N}
	\, \log (e^2N) \big)^{i}}{i!} \\
	& = \frac{3 \, e^\lambda}{1-64c \, C_{\lambda,N}} \, N^{\frac{5}{2}} \, (\log N) \,
	(e^2 N)^{32c \, C_{\lambda,N}} \,.
\end{split}
\end{equation*}
Given $\epsilon > 0$ we can fix $\lambda$ large so that
$32 c \, C_\lambda  < \frac{\epsilon}{2}$. Then
we have $C_{\lambda,N} = \frac{1}{\log N} + C_\lambda
< \frac{2}{3}\epsilon$ for large $N$. This concludes the proof of \eqref{eq:ntpb2}.
\qed

\smallskip
\section{Bounds on triple intersections}
\label{sec:triple}

In this section, we prove Proposition~\ref{prop:notripleQ}.
First we derive a representation for $M^{N, \rm T}_{s, t}(\phi, \psi)$,
which denotes
the sum in \eqref{eq:MNQ} restricted to $\bsA \cap \bsB\cap \bsC \ne \emptyset$
(recall \eqref{eq:trinotri}).

We denote by $\bsD=(D_1, \ldots, D_{|\bsD|}):=\bsA \cap \bsB\cap \bsC$,
with $D_i=(d_i, w_i)$, the locations of the triple intersections.
If we fix two consecutive triple intersections, say
$D_{i-1} = (a,w)$ and $D_i = (b,z)$, the contribution to \eqref{eq:MNQ} is given by
\begin{equation*}
	\bbE\big[\big(Z^{N, \beta_N}_{a, b}(w, z)\big)^3\big] -
	M^{N, \rm T}_{a, b}(w, z) \,,
\end{equation*}
where $M^{N, \rm T}_{a, b}(w, z)$ is defined in \eqref{eq:quantities},
together with $M^{N, \rm T}_{a, b}(\phi, z)$ and $M^{N, \rm T}_{a, b}(w, \psi)$.
Then we obtain from \eqref{eq:MNQ} the following representation
for $M^{N, \rm T}_{s, t}(\phi, \psi)$
(where $\bbE[\xi^3] := \bbE[\xi_{n,z}^3]$):
\be\label{eq:notriple}
\begin{split}
M^{N, \rm T}_{s, t}(\phi, \psi) & := \frac{1}{N^{3}}
	\sumtwo{\bsD \subseteq \{s+1, \ldots, t-1\} \times \Z^2}{|\bsD| \ge 1} \!\!\!\!\!\!
	\bbE[\xi^3]^{|\bsD|}\
    \Big(\bbE\big[\big(Z^{N, \beta_N}_{s, d_1}(\phi, w_1)\big)^3\big]
    - M^{N, \rm T}_{s, d_1}(\phi, w_1)\Big)  \, \cdot\\
    & \qquad \qquad \qquad \cdot \prod_{i=2}^{|\bsD|}
    \Big(\bbE\big[\big(Z^{N, \beta_N}_{d_{i-1}, d_i}(w_{i-1}, w_i)\big)^3\big]
    - M^{N, \rm T}_{d_{i-1}, d_i}(w_{i-1}, w_i) \Big) \, \cdot\\
    & \qquad \qquad \qquad \cdot \Big(\bbE\big[\big(
    Z^{N, \beta_N}_{d_{|\bsD|}, t}(w_{|\bsD|}, \psi)\big)^3\big]
    - M^{N, \rm T}_{d_{|\bsD|}, t}(w_{|\bsD|}, \psi)\Big) \,.
\end{split}
\ee

To prove Proposition \ref{prop:notripleQ}
we may assume $t=1$, by Remark \ref{rm:ZNscaling},
and also $\phi\geq 0$, $\psi\geq 0$ (otherwise just replace $\phi$ by $|\phi|$
and $\psi$ by $|\psi|$ to obtain upper bounds).
If we rename $(d_1, w_1) = (a,x)$ and $(d_{|\bs{D}|}, w_{|\bs{D}|}) = (b,y)$
in \eqref{eq:notriple}, we get the upper bound
\be\label{eq:MNNTbound}
\begin{aligned}
|M^{N, \rm T}_{0, N}(\phi, \psi)| \leq
& \ |\bbE[\xi^3]|  \cdot \underset{A_N}{\underbrace{
\frac{1}{N^{3}} \sum_{1\leq a\leq N \atop x\in \Z^2}
\Big(\bbE\big[\big(Z^{N, \beta_N}_{0, a}(\phi, x)\big)^3\big]
    - M^{N, \rm T}_{0, a}(\phi, x)\Big)}}\\
& \cdot \Big(\sum_{n=0}^\infty \rho_N^n\Big)
\cdot \underset{B_N}{\underbrace{\sup_{1\leq b\leq N \atop y\in\Z^2} \Big(\bbE\big[\big(
    Z^{N, \beta_N}_{b, N}(y, \psi)\big)^3\big]
    - M^{N, \rm T}_{b, N}(y, \psi)\Big)}} \,,
\end{aligned}
\ee
where we set
\be\label{eq:rhoN}
\rho_N:= |\bbE[\xi^3]|
\sum_{1\leq a\leq N \atop z\in \Z^2}
\Big(\bbE\big[\big(Z^{N, \beta_N}_{0, a}(0, z)\big)^3\big]
    - M^{N, \rm T}_{0, a}(0, z) \Big) \,.
\ee

Note that $\bbE[\xi^3]$ actually depends on $N$, and vanishes as $N\to\infty$.
Indeed, recalling that $\xi_{n,z} = e^{\beta_N \omega_{n,z} - \lambda(\beta_N)}-1$
and $\lambda(\beta) = \frac{1}{2}\beta^2 +O(\beta^3)$
as $\beta \to 0$, see \eqref{eq:xi} and \eqref{eq:genomega}, we have
\begin{equation}\label{eq:3momxi}
	\bbE[\xi^3] = e^{\lambda(3\beta_N)-3\lambda(\beta_N)}
	- 3 \, e^{\lambda(2\beta_N)-2\lambda(\beta_N)} +2 = O(\beta_N^3)
	= O\big((\log N)^{-\frac{3}{2}}\big) \,,
\end{equation}
where the last equality holds by \eqref{eq:sigmaN} and \eqref{eq:olap}.

Then, to prove Proposition~\ref{prop:notripleQ},
by the bound \eqref{eq:MNNTbound}
it would suffice to show that
\begin{equation*}
	\limsup_{N\to\infty} A_N \cdot B_N < \infty \qquad \text{and} \qquad
	\limsup_{N\to\infty} \rho_N < 1 \,,
\end{equation*}
so that the series
$\sum_{n=0}^\infty \rho_N^n = (1-\rho_N)^{-1}$ is bounded.
We are going to prove the following stronger result, which
implies the bound $|M^{N, \rm T}_{0, N}(\phi, \psi)|
= o(N^{-1/2 + \eta})$, for any fixed $\eta > 0$.

\begin{lemma}\label{lem:portmanteau}
The following relations hold as $N\to\infty$, for any fixed $\eps>0$:
\begin{enumerate}
\renewcommand{\theenumi}{\alph{enumi}}
\item\label{it:a} $A_N = o(N^{\eps-1/2})$;
\item\label{it:b} $B_N = o(N^\epsilon)$;
\item\label{it:c} $\rho_N = O((\log N)^{-1/2})$.
\end{enumerate}
\end{lemma}

Before the proof,
we recall that $\bbE\big[\big(Z^{N, \beta_N}_{a, b}(*, \dagger) -
q^N_{a, b}(*, \dagger)\big)^3\big] = M^{N, \rm T}_{a, b}(*, \dagger)
+ M^{N, \rm NT}_{a, b}(*, \dagger)$,
for any $* \in \{w, \phi\}$, $\dagger \in \{z, \psi\}$, hence
\begin{equation} \label{eq:baseq}
\begin{split}
	& \bbE\big[\big(Z^{N, \beta_N}_{a, b}(*, \dagger)\big)^3\big] -
	M^{N, \rm T}_{a, b}(*, \dagger) \\
	& \ \ = q^N_{a, b}(*,\dagger)^3 + 3 \, q^N_{a, b}(*,\dagger) \,
	\bbvar\big(Z^{N, \beta_N}_{a, b}(*,\dagger)\big) + M^{N, \rm NT}_{a, b}(*,\dagger).
\end{split}
\end{equation}
Also note that $M^{N, \rm NT}_{a, b}(*, \dagger)\geq 0$,
see \eqref{eq:MNQ} and \eqref{eq:prel3Q}.

\begin{proof}[Proof of Lemma \ref{lem:portmanteau}]
We first prove point \eqref{it:b}. By definition, see \eqref{eq:qNavg0},
$$
	q^N_{b, N}(y, \psi) = \sum_{z\in \Z^2} q_{N-b}(z-y)
	\psi\big(\tfrac{z}{\sqrt N}\big) \leq |\psi|_\infty.
$$
If we replace $\psi$ by the constant $1$ in the averaged partition function
$Z^{N,\beta_N}_{b, N}(y, \psi)$ we obtain
the point-to-plane partition function $Z^{\beta_N}_{N-b}(y)$,
see \eqref{eq:avgxpsi} and \eqref{eq:Zx}.
Then, by \eqref{eq:inter2},
\begin{equation}\label{eq:varlog}
	\bbvar\big(Z^{N, \beta_N}_{b, N}(y, \psi)\big) \leq
	\bbE\big[Z^{N, \beta_N}_{b, N}(y, \psi)^2 \big]
	\le |\psi|_\infty^2 \, \bbE\big[ Z^{\beta_N}_{N-b}(y)^2 \big] = O(\log N) \,.
\end{equation}
Lastly, by \eqref{eq:ntpb1}, we have
$$
	M^{N, \rm NT}_{b, N}(y, \psi) = O(N^\eps) \,.
$$
It suffices to plug these estimates into \eqref{eq:baseq} with $* = y$, $\dagger = \psi$ and
point \eqref{it:b} follows.

\smallskip

Next we prove point \eqref{it:a}.
First note that
\begin{align*}
\frac{1}{N^3} \sum_{1\leq a\leq N \atop x\in \Z^2} q^N_{0,a}(\phi, x)^3
& = \, \frac{1}{N^3}
\sum_{1\leq a\leq N \atop x\in \Z^2}
\Big(\sum_{y\in\Z^2} \phi\big(\tfrac{y}{\sqrt N}\big)q_a(x-y)\Big)^3 \\
& \le \, \frac{|\phi|_\infty^2}{N^3} \sum_{1\leq a\leq N \atop x\in\Z^2}
\sum_{y\in\Z^2} q_a(x-y)
\, \phi\big(\tfrac{y}{\sqrt N}\big)
= \frac{|\phi|_\infty^2}{N^3} \sum_{1\leq a\leq N} \sum_{y\in\Z^2}
\, \phi\big(\tfrac{y}{\sqrt N}\big) \\
&
= \,\frac{|\phi|_\infty^2 }{N} \sum_{y\in\Z^2} \frac{1}{N} \phi\big(\tfrac{y}{\sqrt N}
\big) = O\Big(\frac{1}{N}\Big),
\end{align*}
where the last sum converges to $\int \phi(x) \dd x$ by Riemann sum approximation.
Next note that we can bound
$\bbvar\big(Z^{N, \beta_N}_{0,a}(\phi, x)\big) \le
|\phi|_\infty^2 \, \bbE[Z^{\beta_N}_a(x)^2] = O(\log N)$,
arguing as in \eqref{eq:varlog}, hence
\begin{align*}
\frac{1}{N^3} \sum_{1\leq a\leq N \atop x\in \Z^2} q^N_{0, a}(\phi, x)
\bbvar\big(Z^{N, \beta_N}_{0,a}(\phi, x)\big)
& \, \leq \frac{1}{N^3} \sum_{1\leq a\leq N \atop x, y\in \Z^2}
\phi\big( \tfrac{y}{\sqrt N} \big) q_a(x-y) \, O(\log N) \\
& \, = \frac{1}{N} \, O(\log N) \sum_{y\in \Z^2} \frac{1}{N}
\phi\big(\tfrac{y}{\sqrt N}\big) = O\Big(\frac{\log N}{N}\Big) \,.
\end{align*}
Lastly, by \eqref{eq:ntpb2}, we have
\begin{align*}
	\frac{1}{N^3} \sum_{1\leq a\leq N \atop x\in \Z^2} M^{N, \rm NT}_{0,a}(\phi, x)
	=O(N^{\eps-\frac{1}{2}}).
\end{align*}
Plugging these estimates into \eqref{eq:baseq} with $* = \phi$ and $\dagger = x$,
point \eqref{it:a} follows.

\smallskip

We finally prove point \eqref{it:c}.
By the local limit theorem \eqref{eq:qas}
we have $q_a(x)\leq \frac{c}{a}$ for some $c < \infty$,
uniformly in $a\in\N$ and $x\in \Z^2$. Therefore, recalling \eqref{eq:3momxi}, we have
$$
	\bbE[\xi^3] \sum_{1\leq a\leq N \atop z\in \Z^2} q^N_{0,a}(0, z)^3 \leq \bbE[\xi^3]
	\sum_{1\leq a\leq N \atop z\in \Z^2} \frac{c^2}{a^2} q_a(z)
	= \bbE[\xi^3] \sum_{1\leq a\leq N} \frac{c^2}{a^2} = O((\log N)^{-3/2}) \, .
$$
Next we bound $\bbvar\big(Z^{\beta_N}_{0,a}(0, z)\big) \le
\sigma_N^{-2} \, U_N(a,z)$, see \eqref{eq:UNVar},
and note that
\begin{equation*}
	\sum_{z\in\Z^2} U_N(a,z) = U_N(a) \le C \, c_{ \theta}
	\, \frac{\log N}{a \, \log (e^2 N/a)} \,,
\end{equation*}
by \eqref{eq:UNnrenewal}, \eqref{eq:UnLLT} and \eqref{eq:Gthetaunif}.
Bounding $q_a(x)\leq \frac{c}{a}$
and $\sigma_N^{-2} = O(\log N)$, see \eqref{eq:sigmaN}
and \eqref{eq:olap}, we obtain
\begin{align*}
	 \bbE[\xi^3] \sum_{1\leq a\leq N \atop z\in \Z^2} q^N_{0, a}(0, z)
	 \bbvar\big(Z^{N, \beta_N}_{0,a}(0, z)\big)
	\leq c' \, \bbE[\xi^3] \sum_{1\leq a\leq N} \frac{1}{a^2} \,
	\frac{(\log N)^{2}}{\log (e^2 N/a)} \,.
\end{align*}
For $a \le \sqrt{N}$ we can bound
$\log (e^2 N/a) \ge \log (e^2 \sqrt{N}) \ge \frac{1}{2} \log N$,
while for $\sqrt{N} < a \le N$ we can simply bound
$\log (e^2 N/a) \ge \log e^2 = 2$. This shows that the last sum
is uniformly bounded, since $\sum_{a \ge 1} \frac{2 \, \log N}{a^2} +
\sum_{a > \sqrt{N}} \frac{(\log N)^{2}}{2 \, a^2}
= O(\log N) + O(\frac{(\log N)^{2}}{\sqrt{N}})$.
We thus obtain
\begin{equation*}
	 \bbE[\xi^3] \sum_{1\leq a\leq N \atop z\in \Z^2} q^N_{0, a}(0, z)
	 \bbvar\big(Z^{N, \beta_N}_{0,a}(0, z)\big)
	 = O\big(  \bbE[\xi^3] \, \log N \big) = O\big( (\log N)^{-1/2} \big) \,.
\end{equation*}
Lastly, by \eqref{eq:ntpb3}, we also have
\begin{align*}
	\bbE[\xi^3] \sum_{1\leq a\leq N \atop z\in \Z^2} M^{N, \rm NT}_{0,a}(0, z)
	= \bbE[\xi^3] \, O(1) = O\big( (\log N)^{-3/2} \big) \,.
\end{align*}
If we plug the previous bounds into \eqref{eq:baseq} with $* = 0$ and $\dagger = z$,
point \eqref{it:c} is proved.
\end{proof}

\smallskip
\section{Proof for the stochastic heat equation}\label{sec:SHE}

In this section we prove Theorems \ref{th:SHEvariance} and \ref{th:SHE3rdmom}
on the variance and third moment of the solution to the stochastic heat equation.

We first give a useful representation
of $u^\epsilon(t,\phi):=\int_{\R^2} \phi(x) u^\epsilon(t,x)\,\dd x$.
By a Feynman-Kac representation
and the definition of the Wick exponential (see \cite{CSZ17b} for details),
it follows that $u^\epsilon(t,\phi)$
 is equal in distribution to the Wiener chaos expansion
\begin{align}\label{SHEexpand}
	u^\epsilon(t,\phi)& \stackrel{d}{=}\int_{\mathbb{R}^2} \phi(x)\,\dd x +\sum_{r\geq 1}
	\beta_\epsilon^r \int_{0<t_1<\cdots<t_r<\epsilon^{-2}t} \int_{(\R^2)^r} \,\, \prod_{i=1}^r
	W(\dd t_i \,\dd x_i) \, \cdot \nonumber\\
	& \qquad\qquad  \cdot \bigg\{\int_{\R^2} \dd x \,\,
	\epsilon^2\phi(\epsilon x) \int_{(\mathbb{R}^2)^r}  \prod_{i=1}^r
	g_{t_i-t_{i-1}}(\hat x_{i-1},\hat x_{i} ) \, j(\hat x_i-x_i) \,\dd \hat x_i \, \bigg\} \,
\end{align}
with the convention that $t_0 := 0$ and $\hat x_0=x$.

Expression \eqref{SHEexpand} is the starting point to prove both
Theorems~\ref{th:SHEvariance} and \ref{th:SHE3rdmom}.
To analyze this expression, we first need to
extend the renewal theory framework, described in Subsections~\ref{sec:outline}
and~\ref{sec:renewal}, to continuum distributions. The key results,
described in the next subsection,
are analogous to those obtained in the discrete setting, see \cite[Remark 1.7]{CSZ18}.

\subsection{Renewal framework}
\label{sec:rf}

Fix a continuous function $r: [0,\infty) \to (0,\infty)$ such that\footnote{The precise constant
$4 \pi$ in \eqref{eq:ras} is the one relevant for us, but any
other positive constant would do.}
\begin{equation} \label{eq:ras}
	r(t) = \frac{1}{4\pi t} \big(1+o(1)\big) \qquad \text{as } t \to \infty \,.
\end{equation}
For $\epsilon > 0$, we consider i.i.d.\ random variables $\big(\cT_{i}^{\,\epsilon}\big)_{i \ge 1}$
 with density
\begin{align}\label{law_cT}
	\P(\cT^{\,\epsilon}_{i} \in \dd t) = \frac{r(t)}{\cR_\epsilon }  \,
	\ind_{[0,\epsilon^{-2}]}(t) \,\dd t \,,
\end{align}
where $\cR_\epsilon := \int_0^{\epsilon^{-2}} r(t) \, \dd t$ is the normalization constant.
Note that $\cT^\epsilon_1 + \ldots + \cT^\epsilon_k$
is a continuum analogue of $\tau^{(N)}_k$ in \eqref{eq:tau},
see \eqref{eq:un}-\eqref{eq:XN}, with the identification $N = \epsilon^{-2}$.

Let us quote some relevant results from \cite{CSZ18} that will be needed in the sequel.
\begin{itemize}
\item By \cite[Proposition 1.3]{CSZ18}, we have the convergence in distribution
\begin{equation} \label{eq:convlaw}
	\Big( \epsilon^{2} \big( \cT^\epsilon_1 + \ldots + \cT^\epsilon_{\lfloor
	s \log \epsilon^{-2} \rfloor} \big) \Big)_{s \ge 0}
	\ \xrightarrow[\ \epsilon \to 0\ ]{d} \ (Y_s)_{s \ge 0} \,,
\end{equation}
where $(Y_s)_{s \ge 0}$ is the Dickman subordinator, whose marginal
density is given by \eqref{eq:fst}.

\item By \cite[Lemma 6.1]{CSZ18},
the following large deviations bound holds, with $c \in (0,1)$:
\begin{equation} \label{eq:largedev}
	\P \big( \cT^\epsilon_1 + \ldots + \cT^\epsilon_{\lfloor
	s \log \epsilon^{-2} \rfloor} \le \epsilon^{-2} \big) \le e^{s - c s \log s} \,,
	\qquad \forall \epsilon \in (0,1), \ \forall s \in [0,\infty) \,.
\end{equation}
\end{itemize}

Let us now take $\lambda_\epsilon$ such that
\begin{equation} \label{eq:lambdaeps}
	\lambda_\epsilon := 1 + \frac{\theta}{\log \epsilon^{-2}} \big(1+o(1)\big) \,,
	\qquad \text{for some } \theta \in \R \,.
\end{equation}
Then it follows by Riemann sum approximation (set
$r = s \, \log \epsilon^{-2}$) that for all $T \in [0,1]$
\begin{equation}\label{eq:Riemann}
\begin{split}
	\frac{1}{\log \epsilon^{-2}} \, \sum_{r \ge 1} \lambda_\epsilon^r \: \P\big(
	\cT^\epsilon_1 + \ldots + \cT^\epsilon_{r} \le \epsilon^{-2} T \big)
	\ \xrightarrow[\ \epsilon \to 0 \ ]{} \
	& \int_0^\infty e^{\theta u} \, \P(Y_u \le T) \, \dd u \,.
\end{split}
\end{equation}

This relation will play a crucial role.
We now list some approximations that we can make
in the left hand side of \eqref{eq:Riemann}, without
affecting the convergence.
\begin{enumerate}
\item\label{it:1} \emph{We can restrict the sum
to $r \le K \log \epsilon^{-2}$, for large $K  > 0$.}
Indeed, it is easily seen by \eqref{eq:largedev} and \eqref{eq:lambdaeps}
that the contribution
of $r > K \log \epsilon^{-2}$ to the sum
in \eqref{eq:Riemann} is small, uniformly in $\epsilon$, for $K$ large.

\item\label{it:2}
\emph{We can restrict the probability
to the event ``there are no consecutive short increments''},
where we say that an increment $\cT_i^\epsilon$ is \emph{short} if and only if $\cT_i^\epsilon
\le (\log \epsilon^{-2})^2$.
Indeed, the probability that an increment is short is, by \eqref{eq:ras}-\eqref{law_cT},
\begin{equation} \label{eq:peps}
	p_\epsilon := \P\big( \cT_i^\epsilon \le (\log \epsilon^{-2})^2 \big)
	= \frac{\int_0^{(\log \epsilon^{-2})^2} r(t) \, \dd t}
	{\int_0^{\epsilon^{-2}} r(t) \, \dd t}
	= O \bigg( \frac{\log (\log \epsilon^{-2})}{\log \epsilon^{-2}}\bigg) \,,
\end{equation}
hence the probability of having two consecutive short increments among
$\cT_1^\epsilon, \ldots, \cT_r^\epsilon$ is
\begin{equation*}
	\P\Bigg( \bigcup_{i=1}^{r-1}
	\big\{ \cT_i^\epsilon \le (\log \epsilon^{-2})^2,
	\cT_{i+1}^\epsilon \le (\log \epsilon^{-2})^2 \big\} \Bigg)
	\le r \, p_\epsilon^2
	\le O \bigg( \frac{r \, \big(\log (\log \epsilon^{-2}) \big)^2}
	{(\log \epsilon^{-2})^2}\bigg) \,,
\end{equation*}
which vanishes as $\epsilon \to 0$, when we restrict to $r \le K \log \epsilon^{-2}$.

\item\label{it:3} \emph{We can further restrict the probability
to the event ``the first increment $\cT_1^\epsilon$ is long, i.e.\ not short''},
simply because $p_\epsilon \to 0$ as $\epsilon \to 0$, see \eqref{eq:peps}.

\end{enumerate}

\subsection{Proof of Theorem~\ref{th:SHEvariance}}

It follows from the expansion \eqref{SHEexpand} that
$\bbE[u^\epsilon(t,\phi)]=\int_{\R^2} \phi(x) \,\dd x$ and that the variance of $u^\epsilon(t,\phi)$ is given by
\begin{align} \label{Var}
	\bbvar \big(u^\epsilon(t,\phi) \big) =
	\epsilon^4\,\int_{\R^2\times \R^2} \,\,\phi( \epsilon \hat x) \, \phi(\epsilon \tilde x)\, \,
	K^{\epsilon}_{t}(\hat x , \tilde x)   \,\, \dd \hat x \, \dd \tilde x
\end{align}
where, using the conventions $\hat x_0=\hat x$, $\tilde x_{0}=\tilde x$
and $\vec{t} = (t_1,\ldots, t_r)$, we define
\begin{equation} \label{eq:argu}
\begin{split}
	K^{\epsilon}_{t}(\hat x , \tilde x)&:=
	 \sum_{r\geq 1} \beta_\epsilon^{2 r} \int_{0<t_1<\cdots<t_r<\epsilon^{-2}t} \, \dd \vec{t}
	 \, \int_{(\mathbb{R}^2)^r} \prod_{i=1}^{r}
	\dd x_i \, \int_{(\R^2)^{2r}}
	\, \prod_{i=1}^{r}\dd \hat x_i \, \dd {\tilde x}_i\, \\
	&\qquad\qquad \cdot \prod_{i=1}^r g_{t_i-t_{i-1}}( \hat x_{i-1}, \hat x_i )  \,
	g_{t_i-t_{i-1}}(\tilde x_{i-1}, \tilde x_i )  \, j(\hat x_i-x_i) \,
	 j(\tilde x_i-x_i)  \\
	 &= \sum_{r\geq 1} \beta_\epsilon^{2 r} \int_{0<t_1<\cdots<t_r<\epsilon^{-2}t} \,\,\,
	 \dd  \vec t  \,\, \int_{(\R^2)^{2r}}
	\, \prod_{i=1}^{r}\dd \hat x_i \, \dd {\tilde x}_i    \\
	&\qquad\qquad \cdot \prod_{i=1}^r g_{t_i-t_{i-1}}( \hat x_{i-1}, \hat x_i )  \,
	g_{t_i-t_{i-1}}(\tilde x_{i-1}, \tilde x_i ) \, J(\hat x_i-\tilde x_i) \,,
\end{split}
\end{equation}
where the second equality holds because
$j(-x) = j(x)$ and we recall that $J = j * j$.

We now exploit the identity
\begin{equation}\label{eq:identity}
	g_t(x) \,g_t(y) = 4 \, g_{2t}(x-y) \, g_{2t}(x+y) \,.
\end{equation}
If we set $\hat x_i-\tilde x_i =: z_i$ and $\hat x_i+\tilde x_i=w_i$
and take into account that the Jacobian
of the transformation $(x,y)\mapsto (x-y,x+y)$ on $(\mathbb{R}^2)^2$ equals $1/4$,
we obtain, with $z_0=\hat x - \tilde x$,
\begin{equation}\label{NewMathcalU}
\begin{split}
	K^{\epsilon}_{t}(\hat x , \tilde x)
	&= \sum_{r\geq 1} \beta_\epsilon^{2r}
	\int_{0<t_1<\cdots<t_r<\epsilon^{-2}t} \, \dd \vec{t} \, \int_{(\R^2)^{2r}} \,
	\dd \vec{z} \, \dd \vec{w} \\
	& \qquad\qquad\qquad\prod_{i=1}^r g_{2(t_i-t_{i-1})}( w_i- w_{i-1} )
	\, g_{2(t_i-t_{i-1})}(z_i-z_{i-1} ) \, J(z_i) \, \\
	&= \sum_{r\geq 1} \beta_\epsilon^{2r} \int_{0<t_1<\cdots<t_r<\epsilon^{-2}t}
	\, \dd\vec{t} \, \int_{(\R^2)^{r}} \, \dd \vec{z} \, \,
	\prod_{i=1}^r \, g_{2(t_i-t_{i-1})}( z_i- z_{i-1} )    \, J(z_i)  \,.
\end{split}
\end{equation}

Note that variables $z_i$ with $i \ge 1$ lie in $\mathrm{supp}(J)$,
which is a compact subset of $\R^2$, while
$z_0 = \hat x - \tilde x$ is of order $\epsilon^{-1}$, in view of \eqref{Var}.
For this reason, it is convenient to isolate the integrals over $t_1$, $z_1$
and change variable $t_1 \to \epsilon^{-2} t_1$.
Observing that $g_{\epsilon^{-2}t}(x) = \epsilon^2 g_t(\epsilon x)$,
and renaming $(t_1, z_1)$ as $(s, z)$, we obtain
\begin{equation} \label{eq:Kdue}
\begin{split}
	K^{\epsilon}_{t}(\hat x , \tilde x)
	&= \int_0^{t} \dd s \, \int_{\R^2} \dd z \,
	g_{2 s}\big(\epsilon ( z - (\hat x - \tilde x))\big) \, J(z) \,
	\bK^\epsilon_{t-s}(z) \,,
\end{split}	
\end{equation}
where we define the new kernel $\bK^\epsilon_T(z)$ as follows:
\begin{equation}\label{eq:bK}
	\bK^\epsilon_T(z) :=
	\sum_{r \ge 0} \beta_\epsilon^{2(r+1)} \, \int_{0 < t_1 < \ldots < t_r < \epsilon^{-2} T} \, \dd\vec{t} \,
	\int_{(\R^2)^r} \dd \vec{z} \,
	\prod_{i=1}^r \, g_{2(t_i-t_{i-1})}( z_i- z_{i-1} )    \, J(z_i)  \,,
\end{equation}
where $z_0 := z$ and we agree that for $r=0$ the integrals equal $1$.

This key expression will be analyzed using renewal theory.
Note that by \eqref{eq:gu}
\begin{equation} \label{eq:asg2}
	g_{2t}(y-x) = \frac{1}{4\pi t}+ O\bigg(\frac{1}{t^2}\bigg) \qquad \text{as } t \to \infty \,,
	\qquad \text{uniformly in } x,y \in \mathrm{supp}(J) \,,
\end{equation}
so the dependence on the space variables $z_i$ in \eqref{eq:bK}
should decouple.
We can  make this precise using the approximations described
in Subsection~\ref{sec:rf}. We proceed in three steps.

\medskip

\noindent
{\bf Step 1: First approximation.}
Note that $\beta_\epsilon$, see \eqref{eq:betaeps2}, may be rewritten as follows:
\be\label{eq:betaeps2b}
	\beta_\eps^2 = \frac{4\pi}{\log \epsilon^{-2}}
	+ \frac{4 \rho+o(1)}{(\log \epsilon^{-2})^2} \,.
\ee

We first obtain a domination of $\bK^\epsilon_T(z)$ by a renewal quantity.
Let us define
\begin{equation} \label{eq:errebar}
	\bar r(t) := \sup_{z' \in \mathrm{supp}(J)} \int_{\R^2} g_{2t}(z-z') \, J(z) \, \dd z \,.
\end{equation}
Note that $\bar r(t) = \frac{1}{4\pi t} + O(\frac{1}{t^2})$ as $t\to\infty$,
thanks to \eqref{eq:asg2}, hence
\begin{equation} \label{eq:barcR}
	\bar \cR_\epsilon = \int_0^{\epsilon^{-2}} \bar r(t) \, \dd t
	= \frac{1}{4\pi} \log \epsilon^{-2} + O(1) \qquad \text{as } \epsilon \to 0 \,.
\end{equation}
If we denote by $(\bar \cT_i^\epsilon)_{i\in\N}$ i.i.d.\ random variables defined as in
\eqref{law_cT}, more precisely
\begin{align}\label{law_cTbar}
	\P(\bar\cT^{\,\epsilon}_{i} \in \dd t) = \frac{\bar r(t)}{\bar\cR_\epsilon }  \,
	\ind_{[0,\epsilon^{-2}]}(t) \,\dd t \,,
\end{align}
we can bound $\bK^\epsilon_T(z)$ from above for $T \le 1$,
uniformly in $z \in \mathrm{supp}(J)$, as follows:
\begin{equation}\label{eq:bKub1}
\begin{split}
	\sup_{z\in\mathrm{supp}(J)} \bK^\epsilon_T(z)
	& \le \beta_\epsilon^2 \,
	\sum_{r \ge 0} \beta_\epsilon^{2r} \, \int_{0 < t_1 < \ldots < t_r < \epsilon^{-2} T} \,
	\prod_{i=1}^r \, \bar r(t_i-t_{i-1})  \, \dd\vec{t} \\
	& = \beta_\epsilon^2 \, \sum_{r \ge 0} (\beta_\epsilon^2 \, \bar\cR_\epsilon)^{r} \,
	\P(\bar\cT_1^\epsilon + \ldots + \bar\cT_r^\epsilon < \epsilon^{-2}T) \\
	& \le \tfrac{c_1}{\log \epsilon^{-2}} \, \sum_{r \ge 0}
	\big( 1 + \tfrac{c_2}{\log \epsilon^{-2}} \big)^{r} \,
	\P(\bar\cT_1^\epsilon + \ldots + \bar\cT_r^\epsilon < \epsilon^{-2}T) \,,
\end{split}
\end{equation}
where the last inequality holds by \eqref{eq:betaeps2b} and \eqref{eq:barcR},
for suitable $c_1, c_2 \in (0,\infty)$.

\smallskip

The last line of \eqref{eq:bKub1}
is comparable to the left hand side of \eqref{eq:Riemann},
so we can apply the approximations \eqref{it:1}-\eqref{it:3} described
in Subsection~\ref{sec:rf}. In terms of $\bK_T^\epsilon(z)$,  see \eqref{eq:bK}, these
approximations correspond to restricting the sum to $r \le K \log \epsilon^{-2}$
for a large constant $K > 0$, by \eqref{it:1}, and to restricting the integral over $\vec{t}$ to the
following set, by \eqref{it:2}-\eqref{it:3}:
\begin{equation}\label{eq:set}
\begin{split}
	\mathfrak{I}_T^\epsilon
	:= \big\{ & 0<t_1<\cdots<t_r<\epsilon^{-2} \, T : \quad
	t_1 >(\log \epsilon^{-1})^2
	\ \ \text{and, for every $1 \le i \le r-1$,} \\
	& \qquad \qquad \qquad \text{either  } \ t_i - t_{i-1} > (\log \epsilon^{-2})^2
	\ \text{ or } \ t_{i+1} - t_{i} > (\log \epsilon^{-2})^2  \,\big\} \,.
\end{split}
\end{equation}

Summarizing, when we send $\epsilon \to 0$ followed by $K \to \infty$, we can write
\begin{equation}\label{eq:approx1}
	\bK^\epsilon_T(z)
	= \tilde\bK^\epsilon_{T,K}(z) + o(1) \qquad
	\text{uniformly for } z \in \mathrm{supp}(J) \,,
\end{equation}
where we define, with $t_0 := 0$ and $z_0 := z$,
\begin{equation}\label{eq:bKtilde}
	\tilde\bK^\epsilon_{T,K}(z) :=
	\sum_{r = 0}^{K \log \epsilon^{-2}}
	\beta_\epsilon^{2(r+1)} \, \int_{\mathfrak{I}_T^\epsilon} \, \dd\vec{t} \,
	\int_{(\R^2)^r} \dd \vec{z} \,
	\prod_{i=1}^r \, g_{2(t_i-t_{i-1})}( z_i- z_{i-1} )    \, J(z_i) \,.
\end{equation}

\medskip

\noindent
{\bf Step 2: Second approximation.}
Given $r \in \N$,
let us denote by $\mathsf{S}_\epsilon$
and $\mathsf{L}_\epsilon$ the subsets of indexes $i \in \{1,\ldots, r\}$
corresponding to short and long increments:
\begin{gather*}
	\mathsf{S}_\epsilon := \{i \in \{1,\ldots, r\}: \ t_i - t_{i-1} \le (\log \epsilon^{-2})^2\} \,, \\
	\mathsf{L}_\epsilon := \{i \in \{1,\ldots, r\}: \ t_i - t_{i-1} > (\log \epsilon^{-2})^2\} \,.
\end{gather*}
We can then decompose the last product in \eqref{eq:bKtilde} as follows:
\begin{equation*}
\begin{split}
	\prod_{i=1}^r g_{2(t_i-t_{i-1})}( z_i- z_{i-1} )    \, J(z_i)
	= \prod_{i \in \mathsf{S}_\epsilon} g_{2(t_i-t_{i-1})}( z_i- z_{i-1} )    \,
	\prod_{i \in \mathsf{L}_\epsilon} g_{2(t_i-t_{i-1})}( z_i- z_{i-1} )    \,
	\prod_{i=1}^r  J(z_i) \,.
\end{split}
\end{equation*}
We now make replacements and integrations, in order to simplify this expression.

\smallskip

For each $i\in \mathsf{L}_\epsilon$  we replace $g_{2(t_{i}-t_{i-1})}(z_{i}-z_{i-1})$ by
$r(t_i - t_{i-1})$, where we set
\begin{equation} \label{eq:erre}
	r(t) := \langle J,g_{2 t} \,J\rangle
	:=\int_{\R^2} \int_{\R^2} J(x) g_{2t}(x-y) J(y)\,\dd x\,\dd y \,.
\end{equation}
The error
from each such replacement is $ \exp \big\{O\big( (\log \epsilon^{-1})^{-2} \big)\big\}$,
since one easily sees that
$g_{2(t_{i}-t_{i-1})}(z_{i}-z_{i-1})=g_{2(t_{i}-t_{i-1})}(x-y) \,  \exp \big\{O\big( (t_i - t_{i-1})^{-1} \big)\big\} $
and we have $t_{i}-t_{i-1}>(\log \epsilon^{-2})^2$
(recall that $x-y$ and $ z_{i+1}-z_i$ are in the
support of $J$, which is compact).
Since we are restricted to $r\leq K \log \epsilon^{-2}$, see \eqref{eq:bKtilde}, we have
$|\mathsf{L}_\epsilon| \leq K \log \epsilon^{-2}$, hence the total error from all these replacements is
$ \exp \big\{O\big((\log \epsilon^{-1})^{-1} \big)\big\} = (1+o(1))$.
We have shown that
\begin{equation*}
\begin{split}
	\prod_{i=1}^r g_{2(t_i-t_{i-1})}( z_i- z_{i-1} )    \, J(z_i)
	= \big(1+o(1)\big)
	\, \prod_{i \in \mathsf{S}_\epsilon} g_{2(t_i-t_{i-1})}( z_i- z_{i-1} )    \,
	\prod_{i \in \mathsf{L}_\epsilon} r(t_i - t_{i-1}) \,
	\prod_{i=1}^r  J(z_i) \,.
\end{split}
\end{equation*}

We now proceed by integrating successively $\dd z_i$ for $i=r,r-1, \ldots,1$ as follows:
\begin{itemize}
\item for  $i\in \mathsf{L}_\epsilon$
the integral over $\dd z_i$ amounts to $\int_{\R^2} J(z_i) \,\dd z_i=1$;

\item for $i\in \mathsf{S}_\epsilon$  we integrate both $\dd z_{i-1}$ and $\dd z_{i}$ which
gives, see \eqref{eq:erre},
\begin{align*}
	\int_{(\R^2)^2} J(z_i) \, g_{2(t_{i}-t_{i-1})} (z_{i}-z_{i-1}) \, J(z_{i}) \,\dd z_{i-1}\,\dd z_{i}
	= \langle J, g_{2(t_{i}-t_{i-1})} J\rangle = r(t_i - t_{i-1}) \,,
\end{align*}

\end{itemize}
This sequence of integrations is consistent, i.e. it does not result to integrating a variable $\dd z_i$ twice,
because on the set $\mathfrak{I}_T^\epsilon$,
see \eqref{eq:bKtilde}
and \eqref{eq:set}, there are not two consecutive indices $i$ in $\mathsf{S}_\epsilon$.
Therefore, uniformly for $r \le K \log \epsilon^{-2}$, we have shown that
\begin{equation} \label{eq:pluggo}
	\int_{(\R^2)^r} \dd \vec{z} \, \prod_{i=1}^r g_{2(t_i-t_{i-1})}( z_i- z_{i-1} )    \, J(z_i)
	= \big(1+o(1)\big) \, \prod_{i=1}^r r(t_i - t_{i-1}) \,.
\end{equation}

Note that $r(t) = \frac{1}{4\pi t} + O(\frac{1}{t^2})$, by \eqref{eq:erre} and \eqref{eq:asg2},
so we can consider i.i.d.\ random variables $\cT_i^\epsilon$
with law \eqref{law_cT}. When we plug \eqref{eq:pluggo} into \eqref{eq:bKtilde},
the approximations \eqref{it:1}-\eqref{it:3}
described in Subsection~\ref{sec:rf} show that
we can remove the restrictions $r \le K \log \epsilon^{-2}$ on the sum
and $\vec{t} \in \mathfrak{I}_T^\epsilon$ on the integral.
Recalling \eqref{eq:approx1}, we have finally shown that as $\epsilon \to 0$
\begin{equation}\label{eq:approx2}
	\bK^\epsilon_T(z) =
	\big(1+o(1) \big) \, \hat\bK^\epsilon_T \,+\, o(1)  \qquad
	\text{uniformly for } z \in \mathrm{supp}(J) \,,
\end{equation}
where, recalling \eqref{eq:betaeps2b}, we define
\begin{equation}\label{eq:bKhat}
\begin{split}
	\hat\bK^\epsilon_T
	&:= 	\beta_\epsilon^2 \,
	\sum_{r \ge 0} \beta_\epsilon^{2r} \, \int_{0 < t_1 < \ldots < t_r < \epsilon^{-2} T} \,
	\prod_{i=1}^r \, r(t_i-t_{i-1})  \, \dd\vec{t} \\
	& = 	\big(4\pi+o(1)\big) \, \frac{1}{\log \epsilon^{-2}}
	\, \sum_{r \ge 0} (\beta_\epsilon^2 \, \cR_\epsilon)^{r} \,
	\P(\cT_1^\epsilon + \ldots + \cT_r^\epsilon < \epsilon^{-2}T) \,.
\end{split}
\end{equation}

\medskip

\noindent
{\bf Step 3: Variance computation.}
We can finally complete the proof of Theorem~\ref{th:SHEvariance},
by proving relation \eqref{eq:SHE2ndmom}.
Assume that we have shown that, for some $\theta \in \R$,
\begin{equation}\label{eq:precas}
	\beta_\epsilon^2 \, \cR_\epsilon
	= 1 + \frac{\theta}{\log \epsilon^{-2}} \big(1+o(1)\big) \,.
\end{equation}
Then, by \eqref{eq:Riemann} and \eqref{eq:bKhat}, we can write
\begin{equation} \label{eq:bKhatlim}
	\lim_{\epsilon \to 0} \, \hat \bK_T^\epsilon
	= 4\pi \, \int_0^\infty e^{\theta u} \, \P(Y_u \le T) \, \dd u \, \,,
\end{equation}
and the convergence is uniform in $T \in [0,1]$ (because
both sides are increasing and the right hand side is continuous in $T$).
Looking back at \eqref{Var}, \eqref{eq:Kdue} and \eqref{eq:approx2},
after the change of variables $\hat x, \tilde x \to
\epsilon^{-1} \hat x, \epsilon^{-1} \tilde x$, we obtain
\begin{equation*}
\begin{split}
	\bbvar \big(u^\epsilon(t,\phi) \big) = \big(1+o(1)\big) \,
	& \int_{\R^2\times \R^2}
	\, \dd \hat x \, \, \dd \tilde x \,\,\phi( \hat x) \, \phi(\tilde x)  \\
	& \qquad \cdot \int_0^{t} \dd s \, \int_{\R^2} \dd z \,
	g_{2 s}\big(\epsilon  z - (\hat x - \tilde x)\big) \, J(z) \,
	\hat\bK^\epsilon_{t-s} \, + \, o(1) \,.
\end{split}
\end{equation*}
Recalling \eqref{eq:bKhatlim},
since $\int_{\R^2} J(z) \, \dd z = 1$, we have shown that
\begin{equation*}
\begin{split}
	\lim_{\epsilon \to 0} \bbvar \big(u^\epsilon(t,\phi) \big)
	= \int_{\R^2\times \R^2}
	\, \dd \hat x \, \, \dd \tilde x \,\,\phi( \hat x) \, \phi(\tilde x) \,
	Q(\hat x - \tilde x) \,,
\end{split}
\end{equation*}
where
\begin{equation*}
\begin{split}
	Q(x) := 4\pi \, \int_0^{t} \dd s \,
	g_{2 s}(x) \,
	\int_0^\infty e^{\theta u} \, \P(Y_u \le t-s) \, \dd u \,.
\end{split}
\end{equation*}
Recalling that $f_s(\cdot)$ denotes the density of $Y_s$,
see \eqref{eq:fst},
and using the definition \eqref{eq:Gtheta} of $G_\theta(\cdot)$,
we can rewrite $Q(x)$ as follows:
\begin{equation*}
\begin{split}
	Q(x) & = 4\pi \, \int_0^{t} \dd s \,
	g_{2 s}(x) \,
	\int_0^\infty e^{\theta u} \, \bigg( \int_0^{t-s} f_u(r) \, \dd r \bigg) \, \dd u \\
	& = 4\pi \, \int_0^{t} \dd s \,
	g_{2 s}(x) \,
	\int_0^{t-s} G_\theta(r) \, \dd r
	= 4\pi \, \int_{0 < s < v < t}
	g_{2 s}(x) \,
	G_\theta(v-s) \, \dd s \, \dd v \\
	&= 2\pi \, \int_{0 < s < v < t}
	g_{s}(x/\sqrt{2}) \,
	G_\theta(v-s) \, \dd s \, \dd v \,.
\end{split}
\end{equation*}
A look at \eqref{eq:Kt} shows that $Q(x) = 2 K_{t, \theta}(x/\sqrt{2})$, hence
relation \eqref{eq:SHE2ndmom} is proved.

\smallskip

It only remains to prove \eqref{eq:precas} and to identify $\theta$.
Note that by \eqref{eq:gu}
\begin{equation*}
	\int_0^{\epsilon^{-2}} g_{2t}(x-y) \,\dd t
	= \frac{1}{4\pi} \int_0^{\epsilon^{-2}}
	\frac{e^{-\frac{|x-y|^2}{4t}}}{t} \, \dd t
	= \frac{1}{4\pi} \int_{\frac{\epsilon^2 |x-y|^2}{4}}^\infty
	\frac{e^{-u}}{u} \, \dd u
\end{equation*}
Using the following representation of the Euler-Mascheroni constant:
\begin{equation*}
	\int_0^\infty \bigg( \frac{1}{t(t+1)} - \frac{e^{-t}}{t} \bigg) \, \dd t
	= \gamma \,,
\end{equation*}
see \cite[Entry 8.367 (9), page 906]{GR}, and since
\begin{equation*}
	\int_a^\infty \frac{1}{t(t+1)} \, \dd t =
	\int_a^\infty \bigg(\frac{1}{t} - \frac{1}{t+1}\bigg) \, \dd t=  \log(1+a^{-1}) \,,
\end{equation*}
we see that as $\epsilon \to 0$,
\begin{equation*}
	\int_0^{\epsilon^{-2}} \, g_{2t} (x-y) \, \dd t
	= \frac{1}{4\pi} \bigg\{ \log\bigg(1+\frac{4}{\epsilon^2 |x-y|^2}\bigg)
	- \gamma + o(1) \bigg\} \,.
\end{equation*}
Recalling the definition \eqref{eq:erre} of $r(t)$, we have
\begin{equation*}
\begin{split}
	\cR_\epsilon
	& := \int_0^{\epsilon^{-2}} r(t) \, \dd t
	= \int_{(\R^2)^2} J(x) J(y) \int_0^{\epsilon^{-2}} g_{2t}(x-y) \,\dd t \,\,\dd x\dd y \\
	& =  \frac{1}{4\pi} \bigg\{ \log \epsilon^{-2} + \log 4 +2 \int_{\R^2}\int_{\R^2}
	J(x) \log \frac{1}{ |x-y|}\,J(y)\,\dd x \, \dd y - \gamma + o(1) \bigg\} \,.
\end{split}
\end{equation*}
Finally, recalling \eqref{eq:betaeps2b}, we obtain
\begin{equation*}
	\beta_\epsilon^2 \, \cR_\epsilon =
	1 +  \frac{\log 4 +2 \int_{\R^2}\int_{\R^2}
	J(x) \log \frac{1}{ |x-y|}\,J(y)\,\dd x \, \dd y - \gamma
	+ \rho/\pi}{\log \epsilon^{-2}} \big(1+o(1)\big) \,.
\end{equation*}
This shows that \eqref{eq:precas} holds,
with the expression in \eqref{formula_theta} for $\theta$.\qed

\begin{figure}[t]
\begin{tikzpicture}[scale=0.5]
\draw  [fill] (-8,2)  circle [radius=0.1];\draw[thick, dashed](-8,2) to [out=-30,in=100] (-6,0.4); \draw[thick](-6, 0.2) circle (0.6 cm);
\draw [fill](-6,0.4) circle [radius=0.1];
 \draw  [fill] (-8,-1)  circle [radius=0.1];  \node at (-8.2,1.5) {\scalebox{0.7}{$(0,z_1)$}};
 \draw [thick, dashed] (-6,0.4) to [out=80,in=100] (-4.7,0.4); \draw  [fill] (-4.7,0.4)  circle [radius=0.1];
  \draw[thick, dashed] (-6, 0.4) to [out=-80,in=-100] (-4.7,0.1); \draw  [fill] (-4.7,0.1) circle [radius=0.1];
  \draw[thick]  (-4.7,0.2) circle [radius=0.6];
  \draw [thick, dashed] (-4.7,0.4) to [out=80,in=100] (-2.9,0.4); \draw  [fill] (-3.0,0.4)  circle [radius=0.1];
  \draw[thick, dashed] (-4.7,0.1) to [out=-80,in=-100] (-2.9,0.1); \draw  [fill] (-3.0,0.1) circle [radius=0.1];
   \draw [thick, dashed] (-3.0,0.4) to [out=60,in=120] (6,0); \draw  [fill] (6,0)  circle [radius=0.1];
  \draw[thick]  (-3.0,0.2) circle [radius=0.6];
   \node at (-8.2,-3.5) {\scalebox{0.7}{$(0,z_3)$}};
  \draw [thick, dashed] (-2.9,0.1) to [out=-80,in=160] (-1,-2.4); \draw  [fill] (-1,-2.4)  circle [radius=0.1];
  \draw[thick, dashed] (-8,-3) to [out=-50,in=-130] (-1,-2.8); \draw  [fill] (-1.0,-2.8) circle [radius=0.1];
  \draw[thick]  (-1,-2.6) circle [radius=0.6];
   \draw [thick, dashed] (-1,-2.4) to [out=70,in=110] (0.7,-2.4); \draw  [fill] (0.7,-2.4)  circle [radius=0.1];
  \draw[thick, dashed] (-1,-2.8) to [out=-70,in=-110] (0.7,-2.8); \draw  [fill] (0.7,-2.8) circle [radius=0.1];
  \draw[thick]  (0.7,-2.6) circle [radius=0.6];
  \node at (-1,-3.8) {\scalebox{0.7}{$(t_i,\tilde x_i)$}};
   \node at (-1,-1.4) {\scalebox{0.7}{$(t_i,\hat x_i)$}};
   \draw [thick, dashed] (0.7,-2.4) to [out=70,in=110] (2.4,-2.4); \draw  [fill] (2.4,-2.4)  circle [radius=0.1];
  \draw[thick, dashed] (0.7,-2.8) to [out=-70,in=-110] (2.4,-2.8); \draw  [fill] (2.4,-2.8) circle [radius=0.1];
  \draw[thick]  (2.4,-2.6) circle [radius=0.6];
   \draw [thick, dashed] (2.4,-2.4) to [out=70,in=110] (4.1,-2.4); \draw  [fill] (4.1,-2.4)  circle [radius=0.1];
  \draw[thick, dashed] (2.4,-2.8) to [out=-70,in=-110] (4.1,-2.8); \draw  [fill] (4.1,-2.8) circle [radius=0.1];
  \draw[thick]  (4.1,-2.6) circle [radius=0.6];
   \draw [thick, dashed] (4.1,-2.4) to [out=70,in=-160] (6,-0.4); \draw  [fill] (6,-0.4)  circle [radius=0.1];
  \draw[thick, dashed] (4.1,-2.8) to [out=-70,in=180] (8,-4.8); \draw  [fill] (4.1,-2.8) circle [radius=0.1];
  \draw[thick]  (6,-0.2) circle [radius=0.6];
  \draw [thick, dashed] (6,-0.4) to [out=-70,in=-160] (7.5,-1.4);
 \draw [thick, dashed] (6,0) to [out=60,in=160] (7.5,1.2);
  \draw[thick]  (6,-0.2) circle [radius=0.6];
  \node at (-8.2,-1.5) {\scalebox{0.7}{$(0,z_2)$}};
 \draw[thick, dashed] (-7.8,-0.9) to [out=200,in=260] (-6,0);  \draw  [fill] (-6,0)  circle [radius=0.1]; \draw  [fill] (-6,0)  circle [radius=0.1];
     \draw  [fill] (-8,-3)  circle [radius=0.1];
\end{tikzpicture}
\caption{Diagramatic representation of the expansion \eqref{mathcal_I}
of the third moment of the solution of SHE.
Due to the space-mollification of the noise, we have non trivial correlations
 between space-time points $(t_i,\hat x_i)$ and $(t_i,\tilde x_i)$
--- which intuitively belong to two copies of the continuum polymer path, i.e.\ Brownian motion ---
only when $\hat x_i-\tilde x_i$ is in the support of $J(\cdot)$.
This is slightly different from the lattice case, cf.\
the corresponding expansion \eqref{eq:prel2Q} for the directed
polymer,
where non trivial correlations occur only if $\hat x_i = \tilde x_i$,
i.e.\ two copies of the polymer exactly meet.
The disks represent the support of $J(\cdot)$
and should be understood as disks in space $\R^2$
(we drew them in space-time for graphical clarity).
An array of consecutive disks represents the quantity
${\mathcal{U}_\epsilon}(s,t; \bx, \by)$ in \eqref{mathcalU}, with $(s,\bx)$ and $(t,\by)$ corresponding to
 space time location of the points inside the
first and the last disk in a
sequence. They are the analogues of the wiggled lines in Figure~\ref{figure}.}
\label{SHEfigure}
\end{figure}
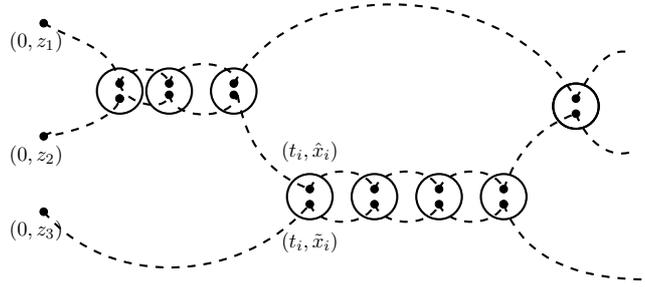

\subsection{Proof of Theorem \ref{th:SHE3rdmom}}

We use the expansion \eqref{SHEexpand} to evaluate
\begin{align}\label{SHEthird_mom}
	\E \Big[ \Big(u^\epsilon(t,\phi)
	- {\textstyle\int_{\mathbb{R}^2} \phi(x)\,\dd x} \Big)^3 \Big] \,.
\end{align}
We are going to expand the third power and compute the expectation, which
amounts to ``pairwise matchings'' of the instances of the noise $W(\dd t_i \,\dd x_i)$
(note that ``triple matchings'' are automatically ruled out, because Gaussians have vanishing third moment).
This will lead to an expression, see \eqref{mathcal_I} below,
which is similar to the one we found for the directed polymer, cf.\ \eqref{eq:prel2Q},
with some additional complications due to the continuous setting.

\smallskip

Before entering the technical details, let us give the heuristic picture,
which is represented in Figure~\ref{SHEfigure}. When taking the third
power of the expansion \eqref{SHEexpand}, we have three sets of coordinates,
that we may label $a, b, c$, that have to match in pairs. Each
matching can  be of three types $ab, bc, ac$, and we say that consecutive
matching of the same type form a \emph{stretch}.
The contribution of each stretch is encoded by a quantity
$\cU_\epsilon(s,t; \bx, \by)$.

The rest of the proof is divided in two steps.
\begin{itemize}
\item In the first step, we define the single-stretch quantity
$\cU_\epsilon(s,t; \bx, \by)$ and we provide some key estimates on it,
based on local renewal theorems obtained in \cite{CSZ18}.

\item In the second step, we express the centered third moment \eqref{SHEthird_mom}
as a sum over the contributions of stretches, see \eqref{mathcal_I}.
We then derive the asymptotic behavior of this expression and show that it
is bounded, completing the proof of Theorem~\ref{th:SHE3rdmom}.
\end{itemize}

\medskip

\noindent
\textbf{Step 1: Single stretch.}
We  introduce a quantity $\cU_\epsilon(s,t; \bx, \by)$
which is an analogue of $U_N(t-s,y-x)$ in the discrete setting,
see \eqref{eq:UNVar}, linked to the point-to-point variance.
Due to the presence of the mollifier,
the space variables are couples $\bx = (\hat x, \tilde x),
\by = (\hat y, \tilde y) \in (\R^2)^2$. Here is the definition:
\begin{equation}\label{mathcalU}
\begin{split}
	{\mathcal{U}_\epsilon}(s,t; \bx, \by) :=
	& \ \beta_\epsilon^2 \,
	g_{t-s} (\hat x, \hat y) \, g_{t-s}(\tilde x, \tilde y) \,
	J(\hat y - \tilde y) \\
	& \ + \ \beta_\epsilon^2 \, \sum_{r\geq 1} \beta_\epsilon^{2 r}
	\int\limits_{s<t_1<\cdots<t_r
	< t} \
	\prod_{i=1}^r \dd t_i \ \int\limits_{(\R^2)^{2r}}
	\, \prod_{i=1}^{r}\dd \hat z_i \, \dd {\tilde z}_i\,
	 \\
	 & \qquad\qquad\qquad\qquad \cdot g_{t_1 - s}(\hat x, \hat z_1) \,
	 g_{t_1 - s}(\tilde x, \tilde z_1) \, J(\hat z_1 - \tilde z_1)  \, \cdot\\
	&\qquad\qquad\qquad\qquad  \cdot \, \prod_{i=2}^{r} g_{t_i-t_{i-1}}( \hat z_{i-1}, \hat z_i )
	\, g_{t_i-t_{i-1}}(\tilde z_{i-1}, \tilde z_i )  \, J(\hat z_i - \tilde z_i)  \, \cdot\\
	 & \qquad\qquad\qquad\qquad \cdot \, g_{ t - t_r}(\hat z_r, \hat y) \,
	 g_{ t - t_r}(\tilde z_r, \tilde y) \,  J(\hat y - \tilde y)\,,
\end{split}
 \end{equation}
where we recall that $J = j * j$
and we agree that the product equals $1$ for $r=1$.

\smallskip

Let us now evaluate $\int_{(\R^2)^2} {\mathcal{U}_\epsilon}(s,t; \bx, \by)
\, \dd \by$. We use the identity \eqref{eq:identity}
and make the change of variables $w_i := \hat z_i + \tilde z_i$,
 $z_i := \hat z_i - \tilde z_i$ for $i=1,\ldots, r$,
as well as $w_{r+1} := \hat y + \tilde y$,
 $z_{r+1} := \hat y - \tilde y$. Integrating out all the $w_i$'s
 for $i = r+1, r, \ldots, 1$,
as we did in \eqref{eq:argu}-\eqref{NewMathcalU},
we obtain
\begin{equation}\label{mathcalU20}
\begin{split}
	\int_{(\R^2)^2} {\mathcal{U}_\epsilon}(s,t; \bx, \by)
	\, \dd \by =
	& \ \beta_\epsilon^2 \, \sum_{r\geq 0} \beta_\epsilon^{2 r}
	\int\limits_{ s<t_1<\cdots<t_r
	< t} \
	\prod_{i=1}^r \dd t_i \ \int\limits_{(\R^2)^{r+1}}
	\, \prod_{i=1}^{r+1}\dd z_i
	 \\
	 & \qquad\qquad\qquad\qquad  \cdot g_{2(t_1 - s)}\big(z_1 - (\hat x - \tilde x) \big) \,
	 J( z_1) \, \cdot \\
	&\qquad\qquad\qquad\qquad  \cdot\, \prod_{i=2}^{r+1} g_{2(t_i-t_{i-1})}( z_i - z_{i-1} )
	\, J (z_i)  \,,
\end{split}
 \end{equation}
where we set $t_{r+1} := t$.
We can rewrite this relation more compactly as follows:
 \begin{equation}\label{eq:mathcalU2}
	\int_{(\R^2)^2} {\mathcal{U}_\epsilon}(s,t; \bx, \by)
	\, \dd \by = \bU_\epsilon(t-s; \hat x - \tilde x) \,,
\end{equation}
where we set, with $t_0 := 0$ and $z_0 := z$,
\begin{equation}\label{mathcalU21}
\begin{split}
	\bU_\epsilon(t; z) :=
	& \ \sum_{r\geq 0} \beta_\epsilon^{2 (r+1)} \!\!\!\!
	\int\limits_{0 <t_1<\cdots<t_r
	< t} \
	\dd \vec{t} \ \int\limits_{(\R^2)^{r+1}}
	\, \dd \vec{z}
	\ \prod_{i=1}^{r+1} g_{2(t_i-t_{i-1})}( z_i - z_{i-1} )
	\, J (z_i)  \,.
\end{split}
 \end{equation}
Note that $\bU_\epsilon(t;z)$ looks similar to $\bK_t^\epsilon(z)$, see
\eqref{eq:bK}, with an important difference: the product in \eqref{mathcalU21}
includes one more term $i=r+1$.
This extra term makes $\bU_\epsilon(t;z)$ close to a \emph{local renewal function},
as we now explain.

Since we content ourselves with an upper bound,
recalling the definition \eqref{eq:errebar} of $\bar r(t)$, we can estimate
\begin{equation} \label{eq:thenre}
	\sup_{z \in \mathrm{supp}(J)} \bU_\epsilon(t; z) \le
	\sum_{r\geq 0} \beta_\epsilon^{2 (r+1)} \!\!\!\!
	\int\limits_{0 <t_1<\cdots<t_r
	< t} \
	\dd \vec{t} \ \prod_{i=1}^{r+1} \, \bar r(t_i - t_{i-1}) \,.
\end{equation}
Let us introduce i.i.d.\ random variables $(\bar\cT_i^\epsilon)_{i\in\N}$
as in \eqref{law_cTbar}, and denote by $\bar f_k^\epsilon(t)$ the density of the
associated random walk:
\begin{equation*}
	\bar f_k^\epsilon(t) :=
	\frac{\P(\bar\cT_1^\epsilon + \ldots + \bar\cT_k^\epsilon \in \dd t)}{\dd t} \,.
\end{equation*}
We can then rewrite \eqref{eq:thenre} as follows:
\begin{equation} \label{eq:thenre2}
	\sup_{z \in \mathrm{supp}(J)} \bU_\epsilon(t; z) \le
	\sum_{r\geq 0} \lambda_\epsilon^{r+1}
	\, \bar f_{r+1}^\epsilon(t) \,, \qquad \text{where} \qquad
	\lambda_\epsilon := \beta_\epsilon^2 \, \bar \cR_\epsilon \,.
\end{equation}
The right hand side can be viewed as a (weighted) \emph{renewal function}: it is the continuum version
of the quantity $U_N(n)$ in \eqref{eq:UNnrenewal}
(with the usual identification $N = \epsilon^{-2}$).
We already remarked that $\lambda_\epsilon = 1 + O(\frac{1}{\log \epsilon^{-2}})$,
by \eqref{eq:betaeps2b} and \eqref{eq:barcR}.
Proposition~\ref{UNtime} holds in this continuum setting
\cite[Remark~1.7]{CSZ18},
hence by the analogue of relation \eqref{est:UNunif}
we get
\begin{equation} \label{eq:thenre3}
	\sup_{z \in \mathrm{supp}(J)} \bU_\epsilon(t; z) \le
	C \, \frac{\log \epsilon^{-2}}{\epsilon^{-2}} \,
	G_\theta(\epsilon^2 t) \,.
\end{equation}

In conclusion, by \eqref{eq:mathcalU2}, we have proved the crucial upper bound
\begin{equation}\label{eq:cruub}
	\sup_{\bx \in (\R^2)^2: \, \hat x - \tilde x \,\in\,
	\mathrm{supp}(J)} \int_{(\R^2)^2} {\mathcal{U}_\epsilon}(s,t; \bx, \by)
	\, \dd\by \, \le \, C \, \epsilon^2 \, \log \epsilon^{-2} \,
	G_\theta(\epsilon^2(t-s)) \,.
\end{equation}

\medskip

\noindent
\textbf{Step 2: Third moment computation.}
We expand the third power in \eqref{SHEthird_mom} using the
Wiener chaos representation  \eqref{SHEexpand}.
We then compute the expectation,
which forces pairwise matchings of the instances of the noise $W(\dd t_i \,\dd x_i)$.
Since
\begin{equation*}
	\int_{\R^2} j(\hat x_i - x_i) \, j(\tilde x_i - x_i) \, \dd x_i = J (\hat x_i - \tilde x_i) \,,
\end{equation*}
we obtain the following expression (analogous to
the directed polymer case, see \eqref{eq:prel2Q}),
where $\cU_\epsilon(a_i,b_i; \bx_i, \by_i)$ are the contributions
of stretches of consecutive pairwise matchings:
\begin{equation}\label{mathcal_I}
\begin{split}
	& \E \Big[ \Big(  u^\epsilon(t,\phi)
	- {\textstyle\int_{\mathbb{R}^2} \phi(x)\,\dd x} \Big)^3 \Big] =\sum_{m\geq 2} 3\,
	\mathcal{I}^{(\epsilon,m)}_t \hskip 2cm \text{with} \\
	\mathcal{I}^{(\epsilon,m)}_t &:= \beta_\epsilon^{2m}
	\!\!\!\!\!\!\!\!\!\!\!\!\!\!\!
	\idotsint\limits_{\substack{0 < a_1 < b_1 < a_2 < b_2 < \ldots < a_m < b_m < \epsilon^{-2} t \\
	\bx_1,\by_1,\dots,\bx_m, \by_m \,\in\, (\R^2)^2}}
	\!\!\!\!\!\!\!\!\!\! \!\!\!\!\!\!\!\!  \dd \vec a\, \dd\vec b \,\dd\vec\bx \,\dd\vec\by
	\int\limits_{(\R^2)^3} \! \dd z_1 \,\dd z_2 \,\dd z_3 \,\,\epsilon^6 \,
	\phi(\epsilon z_1) \phi(\epsilon z_2) \phi( \epsilon z_3)  \, \cdot\\
	& \rule{0pt}{1.4em}\qquad \cdot \, g_{a_1}(z_1,\tilde x_1) \, g_{a_1}(z_2,\hat x_1)
	\,J(\hat x_1 - \tilde x_1) \,\,{\mathcal{U}_\epsilon}(a_1,b_1; \bx_1, \by_1) \, \cdot  \\
	& \rule{0pt}{1.4em} \qquad \cdot \sum_{\hat Y_1\in\{\hat y_1,\tilde y_1\}}
	g_{a_2}(z_3,\tilde x_2)  \, g_{a_2-b_1}(\hat Y_1, \hat x_2 )  \,
	J(\hat x_2- \tilde x_2) \,\,{\mathcal{U}_\epsilon}(a_2,b_2; \bx_2, \by_2)  \, \cdot  \\
	& \rule{0pt}{1.4em} \qquad \cdot \prod_{i=3}^m\,\,
	\sumtwo{\hat Y_{i-1}\in\{\hat y_{i-1}, \tilde y_{i-1}\} }{ \tilde Y_{i-2}\in\{\hat y_{i-2},
	\tilde y_{i-2}\}\setminus\{\hat Y_{i-2}\}}
	g_{a_i-b_{i-1}}(\hat Y_{i-1},\hat x_i)\, g_{a_i-b_{i-2}}(\tilde Y_{i-2}, \tilde x_i)  \, \cdot\\
	& \qquad\qquad\qquad\qquad\qquad\qquad\qquad\qquad
	 \cdot \, J(\hat x_i- \tilde x_i)  \,\,
	{\mathcal{U}_\epsilon}(a_i,b_i; \bx_i, \by_i) \,.
\end{split}
\end{equation}

\begin{remark}
This formula looks actually slightly different than the corresponding expansion
for the directed polymer \eqref{eq:prel2Q}, for the presence of the sums over
$\hat Y_{i-1}$ and $\tilde Y_{i-2}$.
The reason is that,
each time that two copies of the continuum polymers ``spilt''
(i.e.\ at the end of each stretch)
we have to decide which one will meet the unmatched copy and which one will
wait until the next split. But since the two continuum polymers do not match exactly but rather lie
inside the support of $J(\cdot)$, the symmetry
that was present in the discrete case is broken. This gives rise to the sum over
$\hat Y_{i-1}\in\{\hat y_{i-1}, \tilde y_{i-1}\}$
and $ \tilde Y_{i-2}\in\{\hat y_{i-2}, \tilde y_{i-2}\}\setminus\{\hat Y_{i-2}\}$.
\end{remark}

We estimate $\mathcal{I}^{(\epsilon,m)}_t$ as follows. We start by integrating $\by_m$
using \eqref{eq:cruub}, to get
\begin{equation*}
	\int_{(\mathbb{R}^2)^2} {\mathcal{U}_\epsilon}(a_m,b_m; \bx_m, \by_m) \, \dd \by_m
	\le C \, \epsilon^2 \, \log \epsilon^{-2} \,
	G_\theta(\epsilon^2(b_m - a_m)) \,,
\end{equation*}
uniformly over the allowed $\bx_m$. Next we integrate over
$\hat x_m$ and $\tilde x_m$, to get:
\begin{equation*}
\begin{split}
	& \int_{(\R^2)^2} \dd \hat x_m \, \dd \tilde x_m \,
	g_{a_m-b_{m-1}}(\hat Y_{m-1},\hat x_m)\, g_{a_m-b_{m-2}}(\tilde Y_{m-2}, \tilde x_m)
	J(\hat x_m- \tilde x_m) \\
	& \quad = \big(g_{a_m-b_{m-1}} * J * g_{a_m-b_{m-2}}\big)(
	\hat Y_{m-1} - \tilde Y_{m-2}) \\
	& \quad
	= \big(g_{2a_m-b_{m-1} -b_{m-2}} * J\big)(
	\hat Y_{m-1} - \hat Y_{m-2}) \\
	& \quad
	 \le \|g_{2a_m-b_{m-1} -b_{m-2}} * J\|_\infty \le
	\| g_{2a_m-b_{m-1} -b_{m-2}} \|_\infty \\
	& \quad
	 = \frac{1}{2\pi (2a_m-b_{m-1} -b_{m-2})}
	\le \frac{1}{4\pi \sqrt{( a_m-b_{m-1}) \,(a_m-b_{m-2} )}} \,,
\end{split}
\end{equation*}
having used $2xy \le x^2 + y^2$ in the last equality.

We iterate this procedure for $i=m-1, m-2, \ldots$ until $i = 3$: we can first
integrate out $\by_i$ and then $\hat x_i$ and $\tilde x_i$. This
replaces ${\mathcal{U}_\epsilon}(a_i,b_i; \bx_i, \by_i)$
by $C \, \epsilon^2 \, \log \epsilon^{-2} \, G_\theta(\epsilon^2(b_i - a_i))$
and $g_{a_i-b_{i-1}}(\hat Y_{i-1},\hat x_i)\, g_{a_i-b_{i-2}}(\tilde Y_{i-2}, \tilde x_i)$
by $(4\pi \sqrt{( a_i-b_{i-1}) \,(a_i-b_{i-2} )})^{-1}$.
We also recall that $\beta_\epsilon^2 \le C (\log \epsilon^{-2})^{-1}$,
see \eqref{eq:betaeps2b}. Looking back at \eqref{mathcal_I}, we obtain for some $C < \infty$
\begin{equation*}
\begin{split}
	\mathcal{I}^{(\epsilon,m)}_t & \le
	(\beta_\epsilon^2)^2 \, (C \, \epsilon^2)^{m-2}
	\!\!\!\!\!\!\!\!\!\!\!\!\!\!
	\idotsint\limits_{\substack{0 < a_1 < b_1 < a_2 < b_2 < \ldots < a_m < b_m < \epsilon^{-2} t \\
	\bx_1,\by_1,\bx_2, \by_2 \,\in\, (\R^2)^2}}
	\!\!\!\!\!\!\!\!\!\!\!\!\!\!\!\!\!\!\!\! \!  \dd \vec a\, \dd\vec b \,\dd\vec\bx \,\dd\vec\by
	\int\limits_{(\R^2)^3} \! \dd z_1 \,\dd z_2 \,\dd z_3 \,\,\epsilon^6 \,
	\phi(\epsilon z_1) \phi(\epsilon z_2) \phi( \epsilon z_3)  \,\cdot \\
	& \rule{0pt}{1.4em}\qquad\qquad \cdot \, g_{a_1}(z_1,\tilde x_1) \, g_{a_1}(z_2,\hat x_1)
	\, J(\hat x_1 - \tilde x_1) \,\,{\mathcal{U}_\epsilon}(a_1,b_1; \bx_1, \by_1)  \, \cdot \\
	& \rule{0pt}{1.4em} \qquad\qquad \cdot \sum_{\hat Y_1\in\{\hat y_1,\tilde y_1\}}
	g_{a_2}(z_3,\tilde x_2)  \, g_{a_2-b_1}(\hat Y_1, \hat x_2 )  \,
	J(\hat x_2- \tilde x_2) \,\,{\mathcal{U}_\epsilon}(a_2,b_2; \bx_2, \by_2)  \, \cdot \\
	& \rule{0pt}{1.4em} \qquad\qquad \cdot \prod_{i=3}^m\,\,
	\frac{G_\theta(\epsilon^2(b_i - a_i))}{\sqrt{( a_m-b_{m-1}) \,(a_m-b_{m-2} )}}  \,.
\end{split}
\end{equation*}

We can now conclude with the last bounds.
\begin{itemize}
\item We integrate out $\by_2$,
replacing ${\mathcal{U}_\epsilon}(a_2,b_2; \bx_2, \by_2)$
by $C \, \epsilon^2 \, \log \epsilon^{-2} \, G_\theta(\epsilon^2(b_2 - a_2))$,
see \eqref{eq:cruub}.
Then we bound $\phi(\epsilon z_3) \le \|\phi\|_\infty$ and we integrate out $z_3$
(which makes $g_{a_2}(z_3, \tilde x_2)$ disappear)
followed by $\tilde x_2$ and $\hat x_2$
(which make $g_{a_2-b_1}(\hat Y_1, \hat x_2 )  \, J(\hat x_2- \tilde x_2)$ disappear).

\item We integrate out $\by_1$,
replacing ${\mathcal{U}_\epsilon}(a_1,b_1; \bx_1, \by_1)$
by $C \, \epsilon^2 \, \log \epsilon^{-2} \, G_\theta(\epsilon^2(b_1 - a_1))$,
see \eqref{eq:cruub}. Then we bound $\phi(\epsilon z_1) \le \|\phi\|_\infty$
and we integrate out $z_1$
(which makes $g_{a_1}(z_1, \tilde x_1)$ disappear)
followed by $\tilde x_1$ and $\hat x_1$
(which make $g_{a_1}(z_2, \hat x_1 )  \, J(\hat x_1- \tilde x_1)$ disappear).
Lastly, we integrate out $z_1$, which turns the factor $\epsilon^6$ into $\epsilon^4$.
\end{itemize}
This leads to
\begin{equation*}
\begin{split}
	\mathcal{I}^{(\epsilon,m)}_t \le
	(C \, \epsilon^2)^{m} \, \epsilon^4 \,
	\!\!\!\!\!\!\!\!\!\!\!
	\idotsint\limits_{0 < a_1 < b_1 < a_2 < b_2 < \ldots < a_m < b_m < \epsilon^{-2} t}
	\!\!\!\!\!\!\!\!\!\!\!  \dd \vec a\, \dd\vec b \
	& G_\theta(\epsilon^2(b_1-a_1)) \, G_\theta(\epsilon^2(b_2 - a_2))  \, \cdot\\
	& \, \cdot\prod_{i=3}^m\,\,
	\frac{G_\theta(\epsilon^2(b_i - a_i))}{\sqrt{( a_m-b_{m-1}) \,(a_m-b_{m-2} )}}  \,.
\end{split}
\end{equation*}
Finally, the change of variables $a_i \to \epsilon^{-2} a_i$,
$b_i \to \epsilon^{-2} b_i$ gives
\begin{equation*}
\begin{split}
	\mathcal{I}^{(\epsilon,m)}_t \le
	C^m
	\!\!\!
	\idotsint\limits_{0 < a_1 < b_1 < a_2 < b_2 < \ldots < a_m < b_m < t}
	\!\!\!\!\!\!\!\!\!\!\!  \dd \vec a\, \dd\vec b \
	& G_\theta(b_1-a_1) \, G_\theta(b_2 - a_2)  \, \cdot\\
	& \,\cdot \prod_{i=3}^m\,\,
	\frac{G_\theta(b_i - a_i)}{\sqrt{( a_m-b_{m-1}) \,(a_m-b_{m-2} )}}  \,.
\end{split}
\end{equation*}
Note that the right hand side, \emph{which carries no dependence on $\epsilon$},
coincides for $t=1$ with $J^{(m)}$ defined in \eqref{eq:Imphipsibd2}.
We already showed that $J^{(m)}$ decays super-exponentially fast as $m\to\infty$,
see \eqref{eq:Jmexpbd}-\eqref{eq:Jmex2}.
Looking at the first line of \eqref{mathcal_I}, we see that the proof is completed.
\qed

\section*{Acknowledgements}
F.C. is supported by the PRIN Grant  20155PAWZB ``Large Scale Random Structures''.
R.S. is supported by NUS grant R-146-000-253-114. N.Z.
is supported by EPRSC through grant EP/R024456/1.



\bigskip

\begin{thebibliography}{AAAAa}

\bibitem[ACK17]{ACK17}
T.Alberts, J. Clark, S. Koci\'c,
 The intermediate disorder regime for a directed polymer model on a hierarchical lattice,
 {\em Stoch. Process. Appl.} 127, 3291--3330, 2017.

\bibitem[AKQ14]{AKQ14}
T. ~Alberts, K.~ Khanin, J.~ Quastel.
Intermediate Disorder for $1+1$ Dimensional Directed Polymers.
{\em Ann. Probab.} 42, 1212--1256, 2014.

\bibitem[AB16]{AB16}
K. Alexander, Q. Berger,
Local limit theorem and renewal theory with no moments,
{\em Electron. J. Probab.,} Vol. 21 (2016), no 66, 1--18.

\bibitem[BC98]{BC98}
L. Bertini and N. Cancrini.
The two-dimensional stochastic heat equation: renormalizing a multiplicative noise.
{\em J. Phys. A: Math. Gen.} 31, 615, 1998.

\bibitem[BD00]{BD00}
E. Brunet, B. Derrida
Probability distribution of the free energy of a directed polymer in a random medium,
{\em Physical Review E }  61 (6), 6789-6801, 2000.

\bibitem[CDR10]{CDR10}
P. Calabrese, P. Le Doussal, A. Rosso.
Free-energy distribution of the directed polymer at high temperature.
{\em EPL (Europhysics Letters)}, 90(2), 20002, 2010.

\bibitem[CSZ16]{CSZ16}
F. Caravenna, R. Sun, N. Zygouras.
Scaling limits of disordered systems and disorder relevance,
{\em to appear in Proceedings of XVIII International Congress on Mathematical Physics},
 arXiv:1602.05825.

\bibitem[CSZ17a]{CSZ17a}
F. Caravenna, R. Sun, N. Zygouras.
Polynomial chaos and scaling limits of disordered systems.
{\em J. Eur. Math. Soc.} 19, 1--65, 2017.

\bibitem[CSZ17b]{CSZ17b}
F. Caravenna, R. Sun, N. Zygouras.
Universality in marginally relevant disordered systems.
{\em Ann. Appl. Probab.}  27, 3050--3112, 2017.

\bibitem[CSZ18]{CSZ18}
F. Caravenna, R. Sun, N. Zygouras.
The Dickman subordinator, renewal theorems, and disordered systems.
Preprint (2018), arXiv:1805.01465v2 [math.PR].

\bibitem[Com17]{C17}
F. Comets.
{\em Directed Polymers in Random Environments}.
Lecture Notes in Mathematics, 2175. Springer, Cham, 2017.

\bibitem[CY06]{CY06}
F. Comets and N. Yoshida.
Directed polymers in random environment are diffusive at weak disorder.
{\em Ann. Probab.} 34, 1746--1770, 2006.

\bibitem[Cla17]{Cl17}
J. Clark.
High-temperature scaling limit for directed polymers on a hierarchical lattice
with bond disorder.
Preprint (2017), arXiv:1709.01181 [math.PR].

\bibitem[DR04]{DR04}
J. Dimock and S. Rajeev. Multi-particle Schr\"odinger operators with point interactions in the plane.
{\em J. Phys. A: Math. Gen.} 37(39):9157, 2004.

\bibitem[G10]{G10}
G.~Giacomin.
Disorder and critical phenomena through basic probability models.
\'Ecole d'\'Et\'e de Probabilit\'es de Saint-Flour XL – 2010.
{\em Springer Lecture Notes in Mathematics} 2025.

\bibitem[GR07]{GR}
I. S. Gradshtein and I. M. Ryzhik.
Table of Integrals, Series, and Products.
\textit{Academic Press},
Seventh edition (2007).

\bibitem[GRZ17]{GRZ}
Yu Gu, L. Ryzhik, O. Zeitouni.
The Edwards-Wilkinson limit of the random heat equation in dimensions three and higher.
Preprint (2017), arXiv.org: 1710.00344 [math.PR].

\bibitem[GIP15]{GIP15}
M. Gubinelli, P. Imkeller, N. Perkowski.
Paracontrolled distributions and singular PDEs.
\emph{Forum of Mathematics}, Pi (2015) Vol. 3, no. 6.

\bibitem[GQT19]{GQT19}
Y. Gu, J. Quastel, L.-C. Tsai.
Moments of the 2d SHE at criticality.
Preprint (2019), arXiv:1905.11310 [math.PR].

\bibitem[H13]{H13}
M. Hairer.
Solving the KPZ equation.
{\em Ann. of Math.} 178, 559--664, 2013.

\bibitem[H14]{H14}
M. Hairer.
A theory of regularity structures.
{\em Inventiones Math.} 198, 269--504, 2014.

\bibitem[K97]{Kal}
O. Kallenberg.
Foundations of modern probability.
Springer (1997).

\bibitem[K14]{K14}
A. Kupiainen.
Renormalization Group and Stochastic PDE's.
Ann. Henri Poincar\'e 17, 497--535, 2016.

\end{thebibliography}
\end{document}